\documentclass[]{article}

\usepackage{graphicx}
\usepackage{amsmath}
\usepackage{amssymb}
\usepackage{amsthm}
\usepackage{pxfonts}
\usepackage{enumerate}
\usepackage{color}
\usepackage{mathdots}
\usepackage{sectsty}
\usepackage{tikz}
\usepackage{adjustbox}
\usepackage{caption}
\usepackage{bbold}
\usepackage[hidelinks]{hyperref}
\allowdisplaybreaks
\newcommand{\dle}{\rotatebox[origin=c]{45}{$\le$}}
\newcommand{\dless}{\rotatebox[origin=c]{-45}{$<$}}

\sectionfont{\scshape\centering\fontsize{11}{14}\selectfont}
\subsectionfont{\scshape\fontsize{11}{14}\selectfont}
\usepackage{fancyhdr}
\usepackage[nottoc,notlot,notlof]{tocbibind}

\newcommand\shorttitle{Exact solution of interacting particle systems related to random matrices}
\newcommand\authors{Theodoros Assiotis}

\newtheorem{thm}{Theorem}[section]
\newtheorem{cor}[thm]{Corollary}
\newtheorem{lem}[thm]{Lemma}
\newtheorem{defn}[thm]{Definition}
\newtheorem{rmk}[thm]{Remark}
\newtheorem{prop}[thm]{Proposition}
\newtheorem*{theorem*}{Theorem}

\title{\large \bf EXACT SOLUTION OF INTERACTING PARTICLE SYSTEMS RELATED TO RANDOM MATRICES}
\author{\small THEODOROS ASSIOTIS}
\date{}

\begin{document}

\maketitle

\begin{abstract} We consider one-dimensional diffusions, with polynomial drift and diffusion coefficients, so that in particular the motion can be space-inhomogeneous, interacting via one-sided reflections. The prototypical example is the well-known model of Brownian motions with one-sided collisions, also known as Brownian TASEP, which is equivalent to Brownian last passage percolation. We obtain a formula for the finite dimensional distributions of these particle systems, starting from arbitrary initial condition, in terms of a Fredholm determinant of an explicit kernel. As far as we can tell, in the space-inhomogeneous setting and for general initial condition this is the first time such a result has been proven. We moreover consider the model of non-colliding diffusions, again with polynomial drift and diffusion coefficients, which includes the ones associated to all the classical ensembles of random matrices. We prove that starting from arbitrary initial condition the induced point process has  determinantal correlation functions in space and time with an explicit correlation kernel. A key ingredient in our general method of exact solution for both models is the application of the backward in time diffusion flow on certain families of polynomials constructed from the initial condition. 
\end{abstract}

\tableofcontents

\section{Introduction}
This paper is about the exact solution of certain interacting particle systems connected to random matrices. We begin with the simplest motivating example, which albeit in the discrete setting is very closely related, namely the totally asymmetric simple exclusion process (TASEP), see \cite{Spitzer}. The study of TASEP has received a lot of attention in the last few decades, arguably culminating, at least for the purposes we are interested in, in the exact solution (in a sense to be discussed later) for arbitrary initial condition in \cite{KPZfixedpoint}. This led in the 1:2:3 scaling to the construction of the KPZ fixed point \cite{KPZfixedpoint}, the central object in the KPZ universality class \cite{CorwinKPZsurvey}. The continuous space analogue of TASEP is the model of Brownian motions with one-sided collisions or reflections \cite{ReflectedBrownianKPZ}, also called Brownian TASEP, and which is also equivalent to Brownian last passage percolation \cite{OConnellYor,DirectedLandscape}. It was first discovered in \cite{OConnellYor,BougerolJeulin}, and will be discussed in a more general setting in this paper, that this particle system is intimately related to Hermitian Brownian motion \cite{Dyson}. The Brownian model has also recently been solved for general initial condition and shown to converge to the KPZ fixed point \cite{NicaQuastelRemenik}. More generally, particle systems of this type, at least in the discrete setting, have been intensely studied and we give a more detailed literature review in Section \ref{SectionPreviousResults}.

A different type of interacting particle system of significant interest is that of non-colliding (also called non-intersecting) diffusions. Such systems were first studied since they arise as eigenvalue evolutions of Hermitian matrix valued diffusions \cite{MatrixYamadaWatanabe}. The quintessential example is that of non-intersecting Brownian motions, also called Dyson's Brownian motion, which arises as the eigenvalue evolution of Brownian motion on Hermitian matrices. This model has been studied from many different points of view for decades \cite{Dyson,IntroductionToRandomMatrices,UniversalityBook,JohanssonUniversality,KatoriTanemura,SpohnDyson,Tsai}. Our interest here is in exact solvability, for arbitrary deterministic initial condition, in the sense of obtaining explicit formulae for the space-time correlations of the model. These turn out to be given in terms of determinants of an explicit correlation kernel. For fixed time this result goes back to the works of Johansson \cite{JohanssonUniversality} and Brezin and Hikami \cite{brezin1997extension} and for multiple times it is due to Katori and Tanemura in \cite{KatoriTanemura}. Another such system which can be solved exactly in this sense is that of non-intersecting squared Bessel processes \cite{KatoriTanemuraBessel} and we give a more detailed literature review in Section \ref{SectionPreviousResults} and Section \ref{SectionExamples}.

In this paper we study one-dimensional diffusions with polynomial drift and diffusion coefficients which interact via one-sided collisions, namely they solve a simple system of stochastic differential equations with reflection, see equations (\ref{IPS1}) and (\ref{IPS2}). We note that the individual one-dimensional diffusions are called Pearson diffusions by virtue of their relation to the important family of Pearson distributions \cite{PearsonStats}. They were first considered by Kolmogorov in 1931 \cite{KolmogorovPearson}, revisited by Wong in the 1960s \cite{Wong}, and in the past decades they have been much studied in statistics and mathematical finance \cite{PearsonStats}. Moreover, we study the model of non-colliding Pearson diffusions, which also solve a system of interacting stochastic differential equations, see equation (\ref{NonCollidingSDE}). This model includes as special cases all the eigenvalue evolutions of matrix processes related to the classical ensembles of random matrices \cite{ForresterBook}, see Section \ref{SectionExamples} for more details.

Our main results, Theorems \ref{ThmOneSidedUp}, \ref{ThmOneSidedDown} and \ref{ThmBessel} below, on Pearson diffusions with one-sided collisions, starting from arbitrary deterministic initial condition, give a formula for the finite dimensional distributions of this particle system, at a fixed time, in terms of a Fredholm determinant of an explicit kernel.
Our main result, Theorem \ref{NoncollidingThm} below, on non-colliding Pearson diffusions starting from arbitrary initial condition determines the space-time correlations of the induced point process in terms of determinants of an explicit kernel. Except for the Brownian case, and for the non-colliding model also the squared Bessel, all our theorems are new for all other Pearson diffusions.

The contribution of this paper is two-fold. First, in the case of TASEP-like particle systems, both in discrete and continuous space, we solve exactly (in the sense of Theorems \ref{ThmOneSidedUp} and \ref{ThmOneSidedDown}) for the first time, for general\footnote{For a special, the fully-packed, initial condition, in the discrete setting, some general space-inhomogeneous models have been solved, see for example \cite{InhomogeneousKnizelPetrovSaenz,InhomogeneousPushTASEP,DeterminantalStructures}.} initial condition, models for which the motion of particles depends in a non-trivial way on their spatial location (in previous works the motion of particles was translation invariant, see Section \ref{SectionPreviousResults}). Second, we show how the backward in time diffusion flow applied to certain families of polynomials can be used as a key tool to solve both diffusions with one-sided reflections and non-colliding diffusions in a uniform way (within each model).

Armed with the explicit Fredholm determinant formulae we obtain in this paper it would be possible to investigate scaling limits and connections to integrable systems for the interacting particle systems we consider, see the discussion in Sections \ref{SectionScalingLimits} and \ref{SectionIntegrableSystems}. We will pursue this in the future. Moreover, it is possible to consider in the discrete setting space-inhomogeneous particle systems with pushing and blocking mechanism. The ideas presented here, if adapted appropriately, should allow to solve exactly such discrete models as well. We leave this for future work.

Finally, we note that the backward in time diffusion flow in the case of Brownian motion appeared recently in the proof of Newman's conjecture from number theory \cite{NewmanConj}, in a remarkable conjecture on deforming the characteristic polynomial of the Ginibre random matrix ensemble to that of the Gaussian unitary ensemble \cite{hall2022heat}, in statistical mechanics \cite{kabluchko2022leeyang} and finite free probability \cite{marcus2022finite}. Whether analogous applications exist for other Pearson diffusions is not clear but would be interesting to find out.

\subsection{Models and main results}

We fix $N\in \mathbb{N}$ and an open interval $(l,r)\subseteq \mathbb{R}$ once and for all throughout the paper. We consider the following differential operator:
\begin{equation*}
\mathsf{L}=\mathsf{a}(x)\frac{    d^2}{dx^2}+\mathsf{b}(x)\frac{d}{dx},
\end{equation*}
where the functions $\mathsf{a}$ and $\mathsf{b}$ are given by the polynomials:
\begin{equation}\label{DiffusionDrift}
\mathsf{a}(x)=a_2x^2+a_1x+a_0, \ \ \mathsf{b}(x)=b_1x+b_0. 
\end{equation}
Moreover, define the polynomials for each $k=1,\dots, N$,
\begin{equation}\label{Drift}
\mathsf{b}^{(k)}(x)=\mathsf{b}(x)+(N-k)\mathsf{a}'(x)   
\end{equation}
and consider the differential operators, for $k=1,\dots,N$,
\begin{equation*}
    \mathsf{L}^{(k)}=\mathsf{a}(x)\frac{    d^2}{dx^2}+\mathsf{b}^{(k)}(x)\frac{d}{dx},
\end{equation*}
so that in particular $\mathsf{L}^{(N)}\equiv\mathsf{L}$. Note that given $\mathsf{L}$ and $N$, the operators $\mathsf{L}^{(k)}$, for $k=1,\dots,N$, are completely determined. The following is the standing, and basically only, assumption throughout the paper (and will not be recalled in every single result statement).

\begin{defn}[Standing assumption]\label{StandingAssumption}
We assume that for each $k=1,\dots,N$ the differential operator $\mathsf{L}^{(k)}$, with $\mathsf{a}$ and $\mathsf{b}$ as in (\ref{DiffusionDrift}) and (\ref{Drift}), with $\mathsf{a}(x)>0$ for all $x\in (l,r)$, is the generator of a one-dimensional diffusion process in $(l,r)$ with each boundary point $l,r$ being either natural or entrance, see \cite{ItoMckean,HandbookBM,KarlinMcGregor,EthierKurtz} for details on this terminology. In particular, the boundary points are inaccessible and the diffusion associated to $\mathsf{L}^{(k)}$ (which acts on a suitable domain of functions, see \cite{ItoMckean,HandbookBM,KarlinMcGregor,EthierKurtz}) is completely determined by $\mathsf{a}$ and $\mathsf{b}$ without needing to specify boundary conditions at $l$ or $r$.
\end{defn}

There are concrete integral conditions due to Feller, involving the functions $\mathsf{a}$ and $\mathsf{b}$, for when a boundary point is natural or entrance, see for example  \cite{ItoMckean,HandbookBM,KarlinMcGregor,EthierKurtz}. These have already been worked out for the diffusions we consider and we will give references in the sequel. We call the diffusion with generator $\mathsf{L}^{(k)}$ the $\mathsf{L}^{(k)}$-diffusion (similarly for $\mathsf{L}$). We write $\left(e^{t\mathsf{L}^{(k)}};t\ge 0\right)$ for the associated semigroup and, abusing notation, $e^{t\mathsf{L}^{(k)}}(x,y)$ for its transition density with respect to the Lebesgue measure in $(l,r)$, and analogously for $\mathsf{L}$. All these transition densities can be written explicitly in terms of hypergeometric functions, see for example \cite{Wong,arista2022explicit,FisherSnedecor}, but we will not make use of such formulae in this paper. We only need some basic qualitative properties. By standard results, see for example \cite{StroockPDEbook}, $e^{t\mathsf{L}^{(k)}}(x,y)$ is smooth in $(x,y)\in (l,r)^2$ and $y\mapsto\left|\partial_x^ie^{t\mathsf{L}^{(k)}}(x,y)\right|$, $i\in \mathbb{N}$, integrates polynomials in $(l,r)$. We moreover note that using the spectral expansion \cite{ItoMckean} of the transition density in terms of hypergeometric functions \cite{Wong,arista2022explicit,FisherSnedecor}, $z\mapsto e^{t\mathsf{L}^{(k)}}(z,y)$, for $y\in (l,r)$, can be extended\footnote{This property will not be used in the proof other than to give a final contour integral expression for the formula in Theorem \ref{NoncollidingThm}, which is aesthetically pleasing and may be better amenable to asymptotic analysis in the future (it is possible to give an expression in terms of divided differences \cite{DividedDifferences} instead, see Section \ref{SectionNonColliding}).} to an analytic function in a complex neighbourhood of any compact subinterval of $(l,r)$. 

Finally, from the stochastic analysis point of view, by our assumption above (the form of $\mathsf{a}$ and $\mathsf{b}$) and the Yamada-Watanabe theorem \cite{RevuzYor,IkedaWatanabe} the $\mathsf{L}^{(k)}$-diffusion is the unique strong solution $\left(\mathsf{x}(t);t\ge 0\right)$ to the stochastic differential equation (SDE) in $(l,r)$:
\begin{equation*}
d\mathsf{x}(t)=\sqrt{2\mathsf{a}(\mathsf{x}(t))}d\mathsf{w}(t)+\mathsf{b}^{(k)}(\mathsf{x}(t))dt,
\end{equation*}
with $\left(\mathsf{w}(t);t\ge 0\right)$ a standard Brownian motion. 

\paragraph{Diffusions with one-sided collisions}

Our interest in this paper is in $\mathsf{L}^{(k)}$-diffusions interacting by one-sided reflections (also called collisions). Towards that end, define the Weyl chambers $\mathbb{W}_N^\uparrow$ and $\mathbb{W}_N^\downarrow$, corresponding to the interval $(l,r)$:
\begin{align*}
\mathbb{W}_N^{\uparrow}&=\left\{x=(x_1,\dots,x_N)\in (l,r)^N:x_1\le \cdots \le x_N \right\}, \\ \mathbb{W}_N^{\downarrow}&=\left\{x=(x_1,\dots,x_N)\in (l,r)^N:x_1\ge \cdots \ge x_N \right\}.
\end{align*}
We write $\mathbb{W}_N^{\uparrow,\circ},\mathbb{W}_N^{\downarrow,\circ}$ for the interiors (when the inequalities are strict) of $\mathbb{W}_N^{\uparrow},\mathbb{W}_N^{\downarrow}$ respectively. 

We consider the following system of SDEs with reflection \cite{RevuzYor,IkedaWatanabe} in the chamber $\mathbb{W}_N^{\uparrow}$:
\begin{equation}\label{IPS1}
 d\mathsf{x}^\uparrow_k(t)=\sqrt{2\mathsf{a}\left(\mathsf{x}^{\uparrow}_{k}(t)\right)}d\mathsf{w}_k(t)+\mathsf{b}^{(k)}\left(\mathsf{x}^\uparrow_k(t)\right)dt+\frac{1}{2}d\mathfrak{l}_k^{\uparrow}(t),
\end{equation}
with the $\mathsf{w}_k$ being independent standard Brownian motions and where the finite variation terms $\mathfrak{l}_k^{\uparrow}$, which only increases when particles collide to keep them ordered (in other words in  $\mathbb{W}_N^{\uparrow}$), can be identified with a semimartingale local time:
\begin{equation}
\mathfrak{l}_k^{\uparrow}= \textnormal{sem. loc. time of } \mathsf{x}^\uparrow_k-\mathsf{x}^\uparrow_{k-1} \textnormal{ at } 0,
\end{equation}
with $\mathfrak{l}_1^{\uparrow} \equiv 0$. These SDEs have a unique strong solution in $\mathbb{W}_N^{\uparrow}$, see \cite{InterlacingDiffusions} (by virtue of the form of $\mathsf{a}$ and $\mathsf{b}$ the Yamada-Watanabe condition therein is satisfied). Write $\left(\mathsf{S}_t^{\uparrow, (N)};t \ge 0\right)$
for the semigroup of the corresponding Markov process; remarkably this has an explicit expression, see Proposition \ref{TransDensityProp1}, which is the starting point of our analysis. In words, the dynamics are as follows: for each $k$ the $k$-th particle evolves as an independent $\mathsf{L}^{(k)}$-diffusion and when it collides with the $(k-1)$-th particle it receives an infinitesimal push $\frac{1}{2}\mathfrak{l}_k^{\uparrow}$ (which is the only form of interaction between the particles) to keep the ordering. Such particle systems are known as diffusions with one-sided collisions or one-sided reflections. The most famous particle system of this type is when the $\mathsf{L}^{(k)}$-diffusion is a Brownian motion in which case it is also called Brownian TASEP \cite{NicaQuastelRemenik,ReflectedBrownianKPZ}. Since the Brownian local time can be written as a running maximum, see \cite{RevuzYor}, it becomes equivalent to so-called Brownian last passage percolation \cite{OConnellYor}. 

Similarly, consider the following system of SDEs with reflection (now to the left) in the chamber $\mathbb{W}_N^{\downarrow}$:
\begin{equation}\label{IPS2}
 d\mathsf{x}^\downarrow_k(t)=\sqrt{2\mathsf{a}\left(\mathsf{x}^\downarrow_{k}(t)\right)}d\mathsf{w}_k(t)+\mathsf{b}^{(k)}\left(\mathsf{x}^\downarrow_k(t)\right)dt-\frac{1}{2}d\mathfrak{l}_k^{\downarrow}(t),
\end{equation}
with the $\mathsf{w}_k$ being independent standard Brownian motions and where again the finite variation terms $\mathfrak{l}_k^{\downarrow}$ can be identified with a local time:
\begin{equation}
\mathfrak{l}_k^{\downarrow}= \textnormal{sem. loc. time of } \mathsf{x}^\downarrow_k-\mathsf{x}^\downarrow_{k-1} \textnormal{ at } 0,
\end{equation}
with $\mathfrak{l}_1^{\downarrow} \equiv 0$. Again, these equations have a unique strong solution in $\mathbb{W}_N^{\downarrow}$, see \cite{InterlacingDiffusions} (by virtue of the form of $\mathsf{a}$ and $\mathsf{b}$). We write $\left(\mathsf{S}_t^{\downarrow, (N)};t \ge 0\right)$
for the semigroup of the corresponding Markov process. This again has an explicit expression, see Proposition \ref{TransDensityProp2}. The dynamics have an analogous intuitive description as the one above for (\ref{IPS1}).

To state our main results, Theorems \ref{ThmOneSidedUp} and \ref{ThmOneSidedDown} below, on the particle systems (\ref{IPS1}) and (\ref{IPS2}) we need to introduce some basic ingredients. We write $\mathbf{1}_{(\mathcal{A})}$ for the indicator function of a set $\mathcal{A}$. 
\begin{defn}\label{PolyDefinition}
Let $x=(x_1,\dots,x_N)\in (l,r)^N$. For $n=1,\dots,N$ and $k=0,\dots,n-1$ we define the polynomial $\mathsf{q}_k^{(n)}(z)=\mathsf{q}_k^{(n)}(z;x)$ of degree $k$ by requiring, for $i=0,\dots,n-1$,
\begin{equation}\label{PolyDef}
   \partial_z^{i}\mathsf{q}_k^{(n)}(z;x)\big|_{z=x_{n-i}} =(-1)^k\mathbf{1}_{(i=k)}.
\end{equation}
\end{defn}
Clearly, we only require condition (\ref{PolyDef}) to hold for $i=0,\dots,k$ to define $\mathsf{q}_k^{(n)}$ uniquely. Moreover, using (\ref{PolyDef}) we can write a triangular system of equations for the coefficients of $\mathsf{q}_k^{(n)}$ which in particular can be solved to give a complicated explicit expression for them. Alternatively, we have the following rather neat integral expression (where we perform each nested integral  $\int_{y_j}^z f(y_{j+1})dy_{j+1}$ consecutively assuming $y_j<z$) for $\mathsf{q}_k^{(n)}(z)$, which however we will not make use of in this paper,
\begin{equation}
\mathsf{q}_k^{(n)}(z)=\mathsf{q}_k^{(n)}(z;x)=\int^{x_n}_z\int^{x_{n-1}}_{y_1}\cdots\int^{x_{n-k+2}}_{y_{k-2}}\int^{x_{n-k+1}}_{y_{k-1}}dy_k dy_{k-1}\cdots dy_1.
\end{equation}
These polynomials appear implicitly in \cite{NicaQuastelRemenik}. They seem quite natural but we do not know whether they have been studied before for different purposes. We finally record the simplest possible example of such polynomials. If all the coordinates of $x$ are equal, namely $x=(x_*,x_*,\dots,x_*)$ then we observe that $\mathsf{q}^{(n)}_k$ is given by
\begin{equation}\label{AllCoordEqPoly}
\mathsf{q}^{(n)}_k(z)=(-1)^k\frac{(z-x_*)^k}{k!}.
\end{equation}

At several places throughout the paper we will need to apply the diffusion flow corresponding to $\mathsf{L}$ (or more generally $\mathsf{L}^{(k)}$) backward in time to certain families of polynomials. In general, solving the diffusion equation backward in time is non-sensical but it is well-defined on polynomials $p(z)$ as a power series, for any $t\in \mathbb{C}$,
\begin{equation}\label{powerseriesdef}
e^{t\mathsf{L}}p(z)=\sum_{j=0}^\infty \frac{t^j}{j!}\mathsf{L}^jp(z).
\end{equation}
The fact that this makes sense and matches, as it should, for $t>0$, the action of the semigroup $e^{t\mathsf{L}}$ on $p$ will be discussed in Section \ref{SectionBackwardFlow}. Moreover, for a multivariate function $g(z_1,\dots,z_m)$ which is a polynomial in the variable $z_j$ we denote by $e^{t\mathsf{L}_{z_j}}g(z_1,\dots,z_m)$, for $t\in \mathbb{C}$, the application of $e^{t\mathsf{L}}$ to the polynomial $z_j\mapsto g(z_1,\dots,z_m)$.

We write $\partial^{-1}$ for the operator, acting on suitably integrable functions $f$ (which integrate polynomials in $(l,r)$ for example; all functions to which we will apply $\partial^{-1}$ in the sequel will be such),
\begin{equation}\label{IntegrationOperDef}
    \partial^{-1}f(x)=\int_l^xf(y)dy.
\end{equation}
It is then easy to see that, for such $f$, for any $m\in \mathbb{N}$, 
\begin{equation}
    \partial^{-m}f(x)=\underbrace{\partial^{-1}\cdots \partial^{-1}}_{\textnormal{m  times}}f(x)=\int_l^x\frac{(x-y)^{m-1}}{(m-1)!}f(y)dy.
\end{equation}
Abusing notation we will also write $\partial^{-m}(x,y)$ for the corresponding integral kernel:
\begin{equation*}
\partial^{-m}\left(x,y\right)=\frac{(x-y)^{m-1}}{(m-1)!}\mathbf{1}_{(y<x)}.
\end{equation*}
We also define the constants $\mathsf{c}^{(k)}=\mathsf{c}^{(k)}(\mathsf{L})$, for $k=1,\dots,N$, 
\begin{equation*}
 \mathsf{c}^{(k)}=2(N-k-1)a_2+b_1. 
\end{equation*}
Finally, for a fixed vector $z=(z_1,\dots,z_M)\in (l,r)^M$ and indices $n_1<\cdots<n_M$ we introduce the multiplication operators
\begin{equation}
\chi_z^+\left(n_j,w\right)=\mathbf{1}_{(w>z_j)}, \ \ \chi_z^-\left(n_j,w\right)=\mathbf{1}_{(w<z_j)}.
\end{equation}
Our two first main results, Theorems \ref{ThmOneSidedUp} and \ref{ThmOneSidedDown} below, give a formula for the distribution of the interacting particle systems (\ref{IPS1}) and (\ref{IPS2}) in terms of a Fredholm determinant, see for example \cite{SimonFredholm} for background on Fredholm determinants, of an explicit kernel.

\begin{thm}\label{ThmOneSidedUp} Under the standing assumption in Definition \ref{StandingAssumption}, consider the interacting particle system $\left(\left(\mathsf{x}_1^\uparrow(t),\dots,\mathsf{x}_N^\uparrow(t)\right);t\ge 0\right)$ evolving according to the dynamics (\ref{IPS1}) with initial condition $x=(x_1,\dots,x_N)\in \mathbb{W}_N^\uparrow$. For any $t>0$, indices $1\le n_1<n_2<\cdots <n_M \le N$ and locations $z=(z_1,\dots,z_M)\in (l,r)^M$, we have
\begin{align}
 \mathbb{P}\left(\mathsf{x}^\uparrow_{n_j}(t)\le z_j, j=1,\dots,M\right) = \det\left(\mathbf{I}-\chi_z^+\mathfrak{K}_t\chi_z^+\right)_{L^2\left(\left\{n_1,\dots, n_M\right\}\times (l,r)\right)},
\end{align}
where $\det$ is the Fredholm determinant ($\mathbf{I}$ is the identity operator), with
\begin{align}\label{IPSKernel}
\mathfrak{K}_t\left[\left(n_1,y_1\right);\left(n_2,y_2\right)\right]&=-\frac{\left(y_1-y_2\right)^{n_2-n_1-1}}{(n_2-n_1-1)!}\mathbf{1}_{(y_2<y_1)}\mathbf{1}_{(n_2>n_1)}+\partial_{y_1}^{n_1}\mathfrak{G}_{n_2}\left(y_1,y_2\right)e^{-t\mathsf{L}^{(n_2)}_{y_2}},
\end{align}
where $\mathfrak{G}_{n}(y_1,y_2)$ is given by (note this is where the initial condition $x=(x_1,\dots,x_N)$ appears)
\begin{equation}\label{GfunctionDef}
\mathfrak{G}_n(y_1,y_2)=\sum_{k=1}^{n}e^{t\sum_{j=k}^{n-1}\mathsf{c}^{(j)}}\partial_{y_1}^{-k}e^{t\mathsf{L}^{(k)}}\left(x_k,y_1\right)\mathsf{q}_{n-k}^{(n)}(y_2;x).
\end{equation}
\end{thm}

\begin{rmk}
Note that, the notation $\partial_{y_1}^{n_1}\mathfrak{G}_{n_2}\left(y_1,y_2\right)e^{-t\mathsf{L}^{(n_2)}_{y_2}}$ in (\ref{IPSKernel}) is shorthand for
\begin{equation}
\partial_{y_1}^{n_1}\sum_{k=1}^{n_2}e^{t\sum_{j=k}^{n_2-1}\mathsf{c}^{(j)}}\partial_{y_1}^{-k}e^{t\mathsf{L}^{(k)}}\left(x_k,y_1\right)\left[e^{-t\mathsf{L}^{(n_2)}}\mathsf{q}_{n_2-k}^{(n_2)}\right](y_2;x).
\end{equation}
\end{rmk}

\begin{thm}\label{ThmOneSidedDown}
Under the standing assumption in Definition \ref{StandingAssumption}, consider the interacting particle system $\left(\left(\mathsf{x}_1^\downarrow(t),\dots,\mathsf{x}_N^\downarrow(t)\right);t\ge 0\right)$ evolving according to the dynamics (\ref{IPS2}) with initial condition $x=(x_1,\dots,x_N)\in \mathbb{W}_N^\downarrow$. For any $t>0$, indices $1\le n_1<n_2<\cdots <n_M \le N$ and locations $z=(z_1,\dots,z_M)\in (l,r)^M$, we have
\begin{align}
 \mathbb{P}\left(\mathsf{x}^\downarrow_{n_j}(t)\ge z_j, j=1,\dots,M\right) = \det\left(\mathbf{I}-\chi_z^-\mathfrak{K}_t\chi_z^-\right)_{L^2\left(\left\{n_1,\dots, n_M\right\}\times (l,r)\right)},
\end{align}
where $\mathfrak{K}_t$ is constructed as in (\ref{IPSKernel}) and (\ref{GfunctionDef}), but with $x\in \mathbb{W}_N^\downarrow$ instead.
\end{thm}

We now give a probabilistic representation, in terms of a discrete time random walk with exponentially distributed steps, for the kernel $\mathfrak{K}_t$ in the case of interacting squared Bessel diffusions. Writing out the SDE (\ref{IPS2}) in $(0,\infty)$ explicitly in the squared Bessel case:
\begin{equation}\label{BESQparticleSystem}
   d\mathsf{x}^\downarrow_k(t)=2\sqrt{\mathsf{x}^\downarrow_{k}(t)}d\mathsf{w}_k(t)+\left(\theta+2N-2k\right)dt-\frac{1}{2}d\mathfrak{l}_k^{\downarrow}(t),
\end{equation}
where we need $\theta\ge 2$, in order to satisfy our standing assumption (since $\theta+2N-2k\ge 2$ the point $0$ is an entrance boundary point, see \cite{RevuzYor}, while $\infty$ is always natural for $\mathsf{L}^{(k)}$, for $k=1,\dots,N$).  Now, for $\theta\in \mathbb{R}$, write 
$\mathcal{B}^{(\theta)}=2x\frac{d^2}{dx^2}+\theta\frac{d}{dx}$ for the generator of a squared Bessel diffusion process in $(0,\infty)$ with dimension $\theta$ killed when (if) it hits the origin and moreover write $e^{t\mathcal{B}^{(\theta)}}(x,y)$ for its transition density. Note that, for $\theta\ge 2$ the origin is almost surely never reached while for $\theta<2$ this happens almost surely, see \cite{RevuzYor,BesselSurvey}. In particular the transition density  $e^{t\mathcal{B}^{(\theta)}}(x,y)$ is sub-Markovian (integrates to less than 1) for $\theta<2$. It is given explicitly, for $\theta \ge 2$, by
\begin{equation}\label{BESQtransition}
e^{t\mathcal{B}^{(\theta)}}(x,y)=\frac{1}{2t}\left(\frac{y}{x}\right)^{\frac{\theta}{2}}e^{-\frac{(x+y)}{2t}}I_\theta\left(\frac{\sqrt{xy}}{t}\right),
\end{equation}
where $I_\theta$ is the modified Bessel function of the first kind, and for $\theta <2$ it is obtained by the symmetry property $e^{t\mathcal{B}^{(\theta)}}(x,y)=e^{t\mathcal{B}^{(4-\theta)}}(y,x)$, see \cite{RevuzYor,BesselSurvey}. The result then reads as follows. 

\begin{thm}\label{ThmBessel} Let $\theta\ge 2$. Consider the squared Bessel diffusions interacting according to the dynamics (\ref{BESQparticleSystem}) with initial condition $x=(x_1,\dots,x_N)\in \mathbb{W}_N^\downarrow$. Then, for any $t>0$, indices $1\le n_1<n_2<\cdots <n_M \le N$ and locations $z=(z_1,\dots,z_M)\in (l,r)^M$, we have
\begin{align}
 \mathbb{P}\left(\mathsf{x}^\downarrow_{n_j}(t)\ge z_j, j=1,\dots,M\right) = \det\left(\mathbf{I}-\chi_z^-\mathfrak{B}^{(\theta)}_t\chi_z^-\right)_{L^2\left(\left\{n_1,\dots, n_M\right\}\times (l,r)\right)},
\end{align}
where the kernel $\mathfrak{B}_t^{(\theta)}$ is given by
\begin{align*}
\mathfrak{B}_t^{(\theta)}\left[\left(n_1,y_1\right);\left(n_2,y_2\right)\right]&=-\frac{\left(y_1-y_2\right)^{n_2-n_1-1}}{(n_2-n_1-1)!}\mathbf{1}_{(y_2<y_1)}\mathbf{1}_{(n_2>n_1)}\\&+\partial_{y_1}^{n_1}e^{t\mathcal{B}_{y_1}^{(4-\theta-2N)}}\mathbf{E}_{\mathsf{R}_0=y_1}\left[e^{y_1-\mathsf{R}_\tau}\frac{\left(\mathsf{R}_\tau-y_2\right)^{n_2-\tau-1}}{(n_2-\tau-1)!}\mathbf{1}_{(\tau<n_2)}\right]e^{-t\mathcal{B}_{y_2}^{(\theta+2N-2n_2)}},
\end{align*}
where $\left(\mathsf{R}_k;k\ge 0\right)$ is a discrete-time random walk with exponential, with parameter $1$, steps to the left and $\tau=\tau(x)=\min\left\{k\ge 0:\mathsf{R}_k\ge x_{k+1}\right\}$ an integer-valued stopping time and $\mathbf{E}$ denotes expectation with respect to this random walk. Moreover, this kernel can be written as
\begin{equation}\label{PathIntegralKernel}
  \mathfrak{B}_t^{(\theta)}\left[\left(n_1,y_1\right);\left(n_2,y_2\right)\right]=-\partial^{-(n_2-n_1)}(y_1,y_2)\mathbf{1}_{(n_2>n_1)}+\partial_{y_1}^{-(n_2-n_1)}  \mathfrak{B}_t^{(\theta)}\left[\left(n_2,y_1\right);\left(n_2,y_2\right)\right].
\end{equation}
\end{thm}

\begin{rmk}
As in Theorems \ref{ThmOneSidedUp} and \ref{ThmOneSidedDown}, the second term in the sum in the definition of $\mathfrak{B}_t^{(\theta)}$ is shorthand for 
\begin{equation}
\partial_{y_1}^{n_1}e^{t\mathcal{B}_{y_1}^{(4-\theta-2N)}}\mathbf{E}_{\mathsf{R}_0=y_1}\left[e^{y_1-\mathsf{R}_\tau}\frac{1}{(n_2-\tau-1)!}e^{-t\mathcal{B}_{y_2}^{(\theta+2N-2n_2)}}\left(\mathsf{R}_\tau-y_2\right)^{n_2-\tau-1}\mathbf{1}_{(\tau<n_2)}\right].
\end{equation}
\end{rmk}

The formula in the theorem above is analogous and motivated by the representations in terms of random walks in \cite{NicaQuastelRemenik, KPZfixedpoint, matetski2022tasep} for the Brownian case and for models in the discrete setting. It is natural to ask if analogous representations also exist for the more general diffusions considered in this paper and we discuss this briefly in Remark \ref{RemarkProbRep}.

\paragraph{Non-colliding diffusions} We now go on to discuss our results on non-colliding diffusions. Towards this end, we denote by $\mathsf{\Delta}_N(x)$, for $x\in \mathbb{W}_N^\uparrow$, the Vandermonde determinant:
\begin{align*}
\mathsf{\Delta}_N(x)=\prod_{1\le i<j\le N}(x_j-x_i).
\end{align*}
Then, as we explain in Section \ref{SectionNonColliding} it is possible to consider the following Markov semigroup $\left(\mathsf{P}_t^{(N)};t\ge 0\right)$ in $\mathbb{W}_N^{\uparrow}$, given by its explicit transition kernel
\begin{equation}\label{NonCollidingSemigroup}
\mathsf{P}_t^{(N)}(x,dy)=e^{-t\lambda_N}\frac{\mathsf{\Delta}_N(y)}{\mathsf{\Delta}_N(x)}\det\left(e^{t\mathsf{L}}(x_i,y_j)\right)_{i,j=1}^Ndy, 
\end{equation}
with $\lambda_N=\frac{1}{6}N(N-1)(2a_2(N-2)+3b_1)$. This is the Doob $h$-transform \cite{Doob,Pinsky,RevuzYor} of $N$ independent $\mathsf{L}$-diffusions killed when they intersect (equivalently the diffusion with generator $\sum_{i=}^N \mathsf{L}_{x_i}$ with Dirichlet boundary conditions in $\mathbb{W}_N^\uparrow$) by the Vandermonde determinant $\mathsf{\Delta}_N$. The formula (\ref{NonCollidingSemigroup}) is initially defined for $x\in \mathbb{W}_N^{\uparrow,\circ}$ and then extended by L'H\^{o}pital's rule to general $x\in \mathbb{W}_N^\uparrow$, by virtue of the smoothness of $z\mapsto e^{t\mathsf{L}}(z,y)$ (weak continuity of the probability measures $x\mapsto \mathsf{P}_t\left(x,dy\right)$ will be discussed in more detail in the proof of Proposition \ref{DistributionNonColliding}).

By standard results, see \cite{Pinsky,RevuzYor}, on how the generator of a diffusion process transforms under a Doob's $h$-transform \cite{Doob} the semigroup (\ref{NonCollidingSemigroup}) corresponds to the dynamics given by the system of SDEs in $\mathbb{W}_N^{\uparrow}$:
\begin{equation}\label{NonCollidingSDE}
 d\mathsf{z}_i(t)=\sqrt{2\mathsf{a}(\mathsf{z}_i(t))}d\mathsf{w}_i(t)+\left(\mathsf{b}(\mathsf{z}_i(t))+2\mathsf{a}(\mathsf{z}_i(t))\sum_{j\neq i}\frac{1}{\mathsf{z}_i(t)-\mathsf{z}_j(t)}\right)dt,
\end{equation}
where the $\mathsf{w}_i$ are independent standard Brownian motions. For initial conditions $x\in \mathbb{W}_N^{\uparrow,\circ}$ it is easy to show, using a generic argument by virtue of the Doob $h$-transform structure, see for example \cite{SingularValuesBMDrift}, that these SDEs have a unique strong solution with a.s. no collision, in particular a.s. for all $t>0$, $\mathsf{z}(t)=\left(\mathsf{z}_1(t),\dots,\mathsf{z}_N(t)\right)\in \mathbb{W}_N^{\uparrow,\circ}$. This remains true for general initial conditions $x\in \mathbb{W}_N^{\uparrow}$ (despite initially coinciding coordinates) and it can be shown using the results of \cite{GraczykMalecki}. We discuss this in Section \ref{SectionNonColliding}.

Interacting particle systems of the form (\ref{NonCollidingSDE}) arise as the eigenvalue processes of matrix valued diffusions. In the most classical case of Brownian motion, $\mathsf{a}(x)\equiv \frac{1}{2}, \mathsf{b}(x)\equiv 0$, this is the much-studied Dyson Brownian motion \cite{Dyson,IntroductionToRandomMatrices,UniversalityBook}, the eigenvalue process of Brownian motion on the space of Hermitian matrices. We discuss more examples in Section \ref{SectionExamples}.

Our next main result, Theorem \ref{NoncollidingThm} below, says that the correlations in space and time of the interacting particle system (\ref{NonCollidingSDE}) can be computed explicitly. In order to state this result precisely we need to recall the definition of a determinantal point process as suitable for the setting of this paper, see \cite{BorodinDeterminantal,JohanssonDeterminantal} for generalities.

Consider the space $\mathfrak{X}=D\times (l,r)$ where $D$ is some discrete set and endow $\mathfrak{X}$ with the measure $\mu=\mathsf{Count}\times \mathsf{Leb}$ which is the product of counting measure on $D$ with Lebesgue measure on $(l,r)$. A point configuration in $\mathfrak{X}$ is a locally finite collection of points (also referred to as particles) in $\mathfrak{X}$, which we assume are pairwise distinct. We denote the set of all point configurations in $\mathfrak{X}$ by $\mathsf{Conf}(\mathfrak{X})$. There is a natural way to equip $\mathsf{Conf}(\mathfrak{X})$ with a Borel structure, see \cite{BorodinDeterminantal}. A random point process on $\mathfrak{X}$ is a probability measure $\mathfrak{M}$ on $\mathsf{Conf}(\mathfrak{X})$. More generally, abusing terminology (strictly speaking when there is no positivity we cannot speak about randomness), a ``signed random point process" is a (possibly) signed measure on $\mathsf{Conf}(\mathfrak{X})$ of total mass one (this will be relavant for the proofs of Theorems \ref{ThmOneSidedUp} and \ref{ThmOneSidedDown}). We define the correlation functions $\{\rho_n\}_{n=1}^\infty$ (with respect to $\mu$) of the point process associated to the measure $\mathfrak{M}$ on $\mathsf{Conf}(\mathfrak{X})$, if they exist, as follows. For any $n\ge 1$ and any $f$ a compactly supported bounded Borel function on $\mathfrak{X}^n$ we have:
\begin{align}\label{CorrFnDef}
 \int_{\mathfrak{X}^n}f(\mathfrak{z}_1,\dots,\mathfrak{z}_n)\rho_n(\mathfrak{z}_1,\dots,\mathfrak{z}_n)d\mu(\mathfrak{z}_1)\cdots d\mu(\mathfrak{z}_n)= \int_{\mathsf{Conf}(\mathfrak{X})}\sum_{\mathfrak{y}_{i_1},\dots,\mathfrak{y}_{i_n}\in Y} f(\mathfrak{y}_{i_1},\dots,\mathfrak{y}_{i_n})d\mathfrak{M}(Y),
\end{align}
where the sum is taken over all $n$-tuples $\mathfrak{y}_{i_1},\dots,\mathfrak{y}_{i_n}$ of pairwise distinct points of the point configuration $Y$. 

\begin{defn}
A ``signed random point process" on $\mathfrak{X}$ given by the (possibly signed) measure $\mathfrak{M}$ (of total mass $1$) on $\mathsf{Conf}(\mathfrak{X})$ is called a (``signed") determinantal point process if there exists a function $K:\mathfrak{X}\times \mathfrak{X}\to \mathbb{C}$ such that all the correlation functions of the point process $\{\rho_n\}_{n=1}^\infty$ (with respect to $\mu=\mathsf{Count}\times \mathsf{Leb}$) defined by (\ref{CorrFnDef}) are given by:
\begin{equation}
 \rho_n(\mathfrak{z}_1,\dots,\mathfrak{z}_n)= \det\left(K(\mathfrak{z}_i,\mathfrak{z}_j)\right)_{i,j=1}^n, \ \ \textnormal{for} \ n=1,2,\dots,
\end{equation}
in which case the function $K$ is called the correlation kernel.

\end{defn}

Our main result on the interacting particle system (\ref{NonCollidingSDE}) is then the following.

\begin{thm}\label{NoncollidingThm}
Under the standing assumption in Definition \ref{StandingAssumption}, consider the dynamics (\ref{NonCollidingSDE}), with semigroup (\ref{NonCollidingSemigroup}), starting from $x=\left(x_1,\dots,x_N\right)\in \mathbb{W}_N^\uparrow$. For any times $0<t_1<t_2<\cdots<t_M$ these dynamics give rise in a natural way to a random point process on $\{t_1,\dots,t_M\}\times (l,r)$. This point process is then determinantal with correlation kernel given by (with the first variable corresponding to time and the second space),
\begin{align}
 \mathsf{K}\left[(s,y_1);(t,y_2)\right] =-e^{(s-t)\mathsf{L}}(y_2,y_1)\mathbf{1}_{(t<s)}+\frac{1}{2\pi \textnormal{i}}\oint_{\mathsf{\Gamma}^{(N)}}e^{s\mathsf{L}}(z,y_1)e^{-t\mathsf{L}_{y_2}}(y_2-z)^{N-1}\prod_{m=1}^N\frac{1}{z-x_m}dz,
\end{align}
where $e^{-t\mathsf{L}_{y_2}}(y_2-z)^{N-1}$ denotes the application of $e^{-t\mathsf{L}}$ to $(y_2-z)^{N-1}$ as a polynomial in the $y_2$ variable and $\mathsf{\Gamma}^{(N)}$ is a positively oriented contour, encircling all the points $x_1,\dots,x_N$, in a complex neighbourhood of $(x_1,x_N)$.
\end{thm}

\begin{rmk}
As mentioned earlier, it is possible to give a formula for $\mathsf{K}$ that does not make use of complex analysis, involving divided differences \cite{DividedDifferences}, see the proofs in Section \ref{SectionNonColliding} for more details.
\end{rmk}

Although the interacting particle systems (\ref{IPS1}), (\ref{IPS2}) and (\ref{NonCollidingSDE}) might seem unrelated and of very different nature it is actually possible to couple them, for special initial conditions, in a larger interacting particle system taking values in an interlacing array. We discuss this in Section \ref{SectionConnection}, which although being independent to the rest of paper it is rather important conceptually. In the Brownian case this type of result is very well-known and there are a couple of ways of constructing this coupling, see for example \cite{OConnellYor,Warren}. Moreover, it is possible to give a direct connection between the semigroups $\mathsf{S}_t^{\uparrow,(N)}$, $\mathsf{S}_t^{\downarrow,(N)}$ and $\mathsf{P}_t^{(N)}$. We again explain this in Section \ref{SectionConnection}. Although there is a partial connection between these two types of interacting particle systems we currently do not have a direct connection between Theorems \ref{ThmOneSidedUp}, \ref{ThmOneSidedDown} and Theorem \ref{NoncollidingThm}. A key feature of the formulae appearing in all these theorems is the application of the backward in time diffusion flow to different families of polynomials but beyond that a more precise connection is elusive. It would be very interesting if one exists. In particular, is it possible to relate directly the correlation kernel $\mathfrak{K}_t$ from Theorems \ref{ThmOneSidedUp} and \ref{ThmOneSidedDown} to the correlation kernel $\mathsf{K}$ from Theorem \ref{NoncollidingThm}?

\subsection{Relation to previous results}\label{SectionPreviousResults}

Results of the form of Theorems \ref{ThmOneSidedUp} and \ref{ThmOneSidedDown} first appeared in  the discrete space setting, for special initial conditions, for a number of TASEP-like systems in \cite{Sasamoto,BorodinFerrariSasamoto,BorodinFerrariPrahoferSasamoto,BorodinFerrariPushASEP}. Then, in the continuous space setting the model of Brownian motions with one-sided collisions was solved exactly\footnote{In the sense that the finite dimensional distributions can be written as a Fredholm determinant involving an explicit kernel.} for some special initial conditions, which lead to Airy processes in the limit, in \cite{FerrariSpohnWeiss1,FerrariSpohnWeiss2,ReflectedBrownianKPZ}. In the last few years, in a breakthrough work \cite{KPZfixedpoint}, the exact solution of TASEP from general initial condition was found. Making use of this\footnote{We can obtain Brownian motions with one-sided collisions as the low density limit of TASEP.}, it is not hard to solve the Brownian model from arbitrary initial condition as well, see \cite{NicaQuastelRemenik}. Beyond this, the exact solution of particle systems (\ref{IPS1}) and (\ref{IPS2}) for all other diffusions considered here is new.

Since the appearance of \cite{KPZfixedpoint} there have been a number of works \cite{arai2020kpz,KolmogorovTASEP,matetski2022polynuclear,matetski2022tasep,NikosTASEP} which find exact solutions to some interacting particle systems, the most general setup being that of \cite{matetski2022tasep} (except in the case of discrete-time TASEP when a time-inhomogeneous version can be considered as well using the RSK correspondence, see \cite{NikosTASEP}). In \cite{matetski2022tasep} a nice theory is developed to solve exactly certain discrete particle systems, which include for example the ones studied in \cite{DiekerWarren}. We note that all of these papers consider models where the motion of individual particles (without the interactions) does not depend on their spatial location. In this work we consider and solve, for the first time, models for which the motion of particles depends in a non-trivial way on their spatial location. In particular, our results would not follow even after adapting the general theory of \cite{matetski2022tasep} to the continuous space setting.

Regarding the particle system (\ref{NonCollidingSDE}) an exact computation of the correlation kernel $\mathsf{K}$ (in equivalent form) in the Brownian case, for fixed time $s=t$, first appeared in the work of Brezin and Hikami \cite{brezin1997extension} and Johansson \cite{JohanssonUniversality} on universality for random Wigner matrices. The computation of the kernel for any times $s,t$ for both the Brownian and squared Bessel cases appeared in the works of Katori and Tanemura \cite{KatoriTanemura,KatoriTanemuraBessel}. Again, this is obtained there in an equivalent form as it is possible to give $e^{-t\mathsf{L}}p$, with $p$ a polynomial and $t>0$, as an integral expression in these cases. All other cases considered in this paper are new. 

The work of Katori on determinantal martingales \cite{DeterminantalMartingales} used to solve explicitly (in the sense of Theorem \ref{NoncollidingThm}) some non-colliding diffusions is also worth mentioning. The so-called fundamental martingale polynomials $m_n(t,x)$ therein which are a key ingredient in this theory are basically given by the simple formula $e^{-t\mathsf{L}}x^n$ (in the Brownian and squared Bessel cases which are the ones considered there), see Proposition \ref{MartingalePolynomials}, also Section \ref{RmkEquivalentRep}. Using determinantal martingales it is also possible to solve exactly some elliptic versions of space-time determinantal processes and non-colliding Brownian motions on the unit circle, see \cite{DeterminantalMartingales,KatoriElliptic}. Although these models do not directly fall in our framework it is conceivable that a backward in time diffusion flow on the unit circle applied to some sort of trigonometric polynomials, see \cite{kabluchko2022leeyang}, can be used to solve these models as well.

\subsection{On the proofs}\label{SectionProofs}

The proof of Theorems \ref{ThmOneSidedUp} and \ref{ThmOneSidedDown} is in two steps\footnote{The two-step strategy is of course not new. It was followed in the original works \cite{Sasamoto,BorodinFerrariSasamoto,BorodinFerrariPrahoferSasamoto,BorodinFerrariPushASEP} on special initial conditions and more recently in \cite{KPZfixedpoint,NicaQuastelRemenik,arai2020kpz,matetski2022tasep,NikosTASEP} for general initial conditions for certain particle systems. It is the only strategy that does not require a-priori knowledge of the explicit formula for the kernel. If one already knows the highly-nontrivial explicit form of the kernel then at least in the discrete setting \cite{KolmogorovTASEP,matetski2022polynuclear} it is possible to show that the corresponding Fredholm determinant solves the backward Kolmogorov equation with appropriate boundary conditions and the solution is unique (this is basically the only other strategy available).} %It may be possible to adapt this approach to the continuous setting and give an alternative proof of our results (assuming we already know the explicit formula for the kernel) but we expect this to be quite technical. }
: first we prove there is an underlying ``signed" determinantal process structure and then obtain an explicit formula for the correlation kernel by solving a biorthogonalization problem. Our starting point is explicit formulae for the transition semigroups $\mathsf{S}_t^{\uparrow,(N)}$, $\mathsf{S}_t^{\downarrow,(N)}$ obtained in \cite{InterlacingDiffusions}. Then, the fact that there is a determinantal structure in this non-translation invariant in space setting is non-trivial and relies on a key relation between the transition densities $e^{t\mathsf{L}^{(k)}}(x,y)$ and $e^{t\mathsf{L}^{(k+1)}}(x,y)$.

To then solve the biorthogonalization problem we first obtain, using this relation, two different representations for the functions that appear as data in the problem. Finally, finding the biorthogonal functions (polynomials) and proving the required orthogonality relation is very clean using the backward in time diffusion flow technology. To prove Theorem \ref{ThmBessel} we use some rather special properties of the squared Bessel diffusion. A key ingredient being Lemma \ref{KeyBesselLemma} which gives some kind of commutation relation between the $\partial^{-1}$ operator and squared Bessel diffusion semigroups. We note here that $\partial^{-1}$ and the semigroup of a squared Bessel diffusion do not commute exactly and this is an important difference and extra complication to the works \cite{KPZfixedpoint,NicaQuastelRemenik,matetski2022tasep} where the corresponding operators actually commute. Analogues possibly exist for other Pearson diffusions as well, see the discussion in Remark \ref{RemarkProbRep}. Now, to prove Theorem \ref{NoncollidingThm} we again follow the same two-step strategy. The fact that there is an underlying determinantal structure is basically an immediate consequence of the Karlin-McGregor formula for non-intersecting paths \cite{KarlinMcGregor} and the Eynard-Mehta theorem \cite{BorodinRains}. The non-trivial part of the proof is to compute the kernel explicitly. In order to do this we again need to find certain biorthogonal functions. Although the proof is actually presented in a somewhat different way, the essence of the key result in Section \ref{SectionNonColliding}, Proposition \ref{DistributionNonColliding}, is to find such biorthogonal functions, see the discussion preceding that proposition for more details. Using the backward in time diffusion flow finding these biorthogonal functions can again be done in a uniform and clean way.

\subsection{Scaling limits}\label{SectionScalingLimits}
Using the Fredholm determinant formula in Theorems \ref{ThmOneSidedUp} and \ref{ThmOneSidedDown} we can consider various scaling limits of the interacting diffusions considered here. Such results were proven for special initial conditions for a number of different models in the past two decades \cite{BorodinFerrariPrahofer,BorodinFerrariPrahoferSasamoto,BorodinFerrariPushASEP,FerrariSpohnWeiss1,FerrariSpohnWeiss2,ReflectedBrownianKPZ}. Then in \cite{KPZfixedpoint}, TASEP was exactly solved for general initial condition and this was used to construct and give a formula for the transition probabilities of the KPZ fixed point $\mathfrak{h}(t,x)$, by showing convergence of the Fredholm determinants corresponding to TASEP. This convergence is proven by showing trace class convergence of the rescaled kernels. This can be very technically demanding and has been worked out in detail for general initial conditions only in the case of TASEP. Convergence has also been shown for general initial conditions for the Brownian model in \cite{NicaQuastelRemenik} but some results from \cite{DirectedLandscape} are used which shorten the technical arguments. Now, at a kind of soft edge scaling we would expect that the interacting squared Bessel diffusions (\ref{BESQparticleSystem}) should also converge to the KPZ fixed point. 

At least for the initial condition $x=(0,\dots,0)$, which corresponds to looking at consecutive submatrices of a Laguerre unitary ensemble matrix \cite{ForresterBook} convergence of the correlation kernels, in an equivalent form involving Laguerre polynomials, see Section \ref{RmkEquivalentRep}, to the extended Airy kernel follows from \cite{ForresterNagao}. It should be possible to prove this for general initial conditions using the probabilistic representation in Theorem \ref{ThmBessel} but this would be long and technical to perform rigorously and will be done elsewhere. A more robust way to show convergence to the KPZ fixed point, at the appropriate scaling, would be to generalize the recent breakthrough work \cite{quastel2022convergence} to the continuous state space setting. There convergence to the KPZ fixed point is shown for finite range exclusion processes (and the KPZ equation by taking yet another limit) by comparing their transition probabilities to the ones for TASEP using energy estimates. One would expect that the role of TASEP should be played by the Brownian model. In any case, our primary interest in these models is not showing convergence to the KPZ fixed point (which would be interesting) but rather to show convergence to some novel limiting objects. In particular, we believe that at least in the squared Bessel case, in the hard edge scaling, new behaviour should emerge. We will pursue it in the future.

Now, regarding the non-colliding diffusions (\ref{NonCollidingSDE}) the asymptotics of the correlation kernel, in certain regimes, have been computed in the Brownian and squared Bessel cases in \cite{KatoriTanemura,KatoriTanemuraBessel}. This has been done using an equivalent expression, which somehow writes $e^{-t\mathsf{L}}p$, for $t>0$, as an integral (it is unclear whether this can be done in general). Returning to the general case we note that although $e^{t\mathsf{L}}(x,y)$ is explicit \cite{Wong,arista2022explicit,FisherSnedecor} for any $\mathsf{L}$ satisfying our standing assumption, the expression can be rather involved. Fortunately, working at the level of generators things become much clearer conceptually but not fully rigorous when taking limits. In particular, after a gauge transformation and rescaling of the kernel and some non-trivial formal computations it is possible to see that, at least in some examples beyond the Brownian and squared Bessel cases, the kernel $\mathsf{K}$ (recall this depends on $N$) should converge as $N\to \infty$ to a limiting kernel $\mathsf{K}_\infty$.
The following type of problem, imprecisely stated, becomes central to make this investigation rigorous. Suppose $f_N$ is a polynomial and that, in a suitable sense, $f_N \to f_\infty$ for some not necessarily polynomial but nice function $f_\infty$. Moreover, assume that the second order differential operator $\mathfrak{A}_N$, which is some transformed version of $\mathsf{L}^{(N)}=\mathsf{L}$ so that $e^{-t\mathfrak{A}_N}f_N$ makes sense for any $t\in \mathbb{R}$, satisfies $\mathfrak{A}_N \to \mathfrak{A}_\infty$, in an appropriate sense, for some limiting second order differential operator $\mathfrak{A}_\infty$. Then one would hope that, subject to certain conditions (finding the right conditions is part of the actual problem),
\begin{equation*}
   e^{-t\mathfrak{A}_N}f_N \overset{N \to \infty}{\longrightarrow} e^{-t \mathfrak{A}_\infty}f_\infty, \ \ t\in \mathbb{R},
\end{equation*}
and moreover be able to obtain strong enough estimates to prove convergence of kernels in trace class\footnote{A related example that might be good to keep in mind is the gauged transformed and appropriately rescaled Ornstein-Uhlenbeck generator converging to the Airy Hamiltonian $-\frac{d^2}{dx^2}+x$ which corresponds to the top curve of Dyson Brownian converging at the edge scaling to the Airy process, see \cite{QuastelRemenikAiry}.}. To prove this sort of result would involve some non-trivial analysis with Schr\"{o}dinger semigroups and we will pursue it elsewhere. %Finally, it is worth mentioning that a problem which is similar in spirit, but simpler technically, is that of convergence of determinantal point processes induced by spectral projections of certain self-adjoint operators if one knows convergence of the operators (in a suitable sense), see \cite{OlshanskiDifference,TaoRMT}.

\subsection{Connections to integrable systems}\label{SectionIntegrableSystems}

There is a long history of connections between random matrices and integrable systems, see for example \cite{ForresterBook}. This extends to the level of interacting particle systems as we now discuss. We first consider one-point distributions, i.e. the distribution of a single particle. We note that from Corollary \ref{CorollaryDistribution}, with $\mu$
a probability measure supported on $\mathbb{W}_N^{\uparrow,\circ}$ \footnote{But we can extend the statement to initial conditions with coinciding coordinates using an entrance law, see Remark \ref{EntranceLaw}.} and the Markov kernels $\mathsf{\Lambda}\mathsf{E}_\uparrow$ and $\mathsf{\Lambda}\mathsf{E}_\downarrow$ defined in Section \ref{SectionConnection}, we have:
\begin{equation*}
\mathbb{P}_{\mu\mathsf{\Lambda}\mathsf{E}_\uparrow }\left(\mathsf{x}_N^\uparrow(t)\le w\right)=\mathbb{P}_{\mu}\left(\mathsf{z}_N(t)\le w\right), \ \ \mathbb{P}_{\mu\mathsf{\Lambda}\mathsf{E}_\downarrow }\left(\mathsf{x}_N^\downarrow(t)\ge w\right)=\mathbb{P}_{\mu}\left(\mathsf{z}_1(t)\ge w\right), \ \ \forall t\ge 0, \ w\in (l,r).
\end{equation*}
Then, using the connection between $\left(\mathsf{z}(t);t\ge 0\right)$ and eigenvalues of matrix diffusions we obtain that, for certain special initial conditions, these probabilities (which are thus gap probabilities for eigenvalues of random matrices) can be represented
%\footnote{More precisely a transformation of them solves the equation in $w$.}
in terms of Painlev\'{e} equations \cite{ForresterBook,TracyWidom,BorodinDeift,AdlerVanMoerbekeAnnals,ForresterWitteCPAM,ForresterWitteCauchy,ForresterNagoyaCorr}. We note that there is some well-developed theory on gap probabilities of eigenvalues of random matrices and integrable systems, see for example \cite{ForresterBook,TracyWidom,BorodinDeift,AdlerVanMoerbekeAnnals,ForresterWitteCPAM,ForresterWitteCauchy,ForresterNagoyaCorr}.

It is rather remarkable, and far from well-understood, that more sophisticated connections to intergrable systems exist for the multipoint and multitime distributions, at least for some related models that we now survey. For the KPZ fixed point $\mathfrak{h}(t,x)$, with general initial condition $\mathfrak{h}_0$, it was shown in \cite{QuastelRemenikKP}, that the multipoint distributions:
\begin{equation*}
\mathbb{P}_{\mathfrak{h}_0}\left(\mathfrak{h}(t,x_1)\le w_1, \mathfrak{h}(t,x_2)\le w_2,\dots,\mathfrak{h}(t,x_m)\le w_m\right),
\end{equation*}
after an appropriate transformation, satisfy the matrix KP equation, see \cite{QuastelRemenikKP} for the precise statement. It is natural then to ask if something analogous happens for pre-limit models (multipoint in this case would mean looking at distributions of several particles at the same time). At least for TASEP, as the authors mention in \cite{matetski2022polynuclear}, it appears that the transition probabilities do not solve some classical completely integrable equation. However, it was very recently shown in \cite{matetski2022polynuclear} that the multipoint distributions of another discrete model, the polynuclear growth model (PNG), after an appropriate transformation, satisfy the non-abelian 2D Toda equation. It would be very interesting if such results exist\footnote{It is reasonable to hope this, since the continuous world of random matrices is in general nicer for explicit computations compared to the world of discrete models.} for the interacting particle systems (\ref{IPS1}) and (\ref{IPS2}). We hope to investigate this in the future.

In a different direction, for the non-colliding diffusions (\ref{NonCollidingSDE}), in the case of Ornstein-Uhlenbeck and squared radial Ornstein-Uhlenbeck processes (see Section \ref{SectionExamples}), starting from the invariant measure $\mathsf{M}_N$, see Proposition \ref{PropInvariantMeasure}, it has been shown in \cite{TracyWidom,AdlerVanMoerbekeDyson} that the multitime distributions of the top-curve $(\mathsf{z}_N(t);t\ge 0)$:
\begin{equation}
\mathbb{P}_{\mathsf{M}_N}\left(\mathsf{z}_N(t_1)\le w_1, \mathsf{z}_N(t_2)\le w_2,\dots, \mathsf{z}_N(t_m)\le w_m\right),
\end{equation}
solve a system of nonlinear partial differential equations. It is reasonable to expect that analogous results should hold for the other diffusions considered in this paper, when starting from their invariant measure. Finally, as far as we can tell, connections to integrable systems for non-colliding diffusions starting from arbitrary deterministic initial condition have not been investigated yet.

\paragraph{Organisation of the paper} The paper is organised as follows. In Section \ref{SectionExamples} we discuss examples of diffusions that fall within our framework and whether and where the corresponding interacting particle systems (\ref{IPS1}), (\ref{IPS2}), (\ref{NonCollidingSDE}) have been studied before. In Section \ref{SectionBackwardFlow} we study the diffusion flow corresponding to $\mathsf{L}$ (or $\mathsf{L}^{(k)}$) on polynomials backward in time. In Section \ref{SectionMainProofs} we prove Theorems \ref{ThmOneSidedUp}, \ref{ThmOneSidedDown} and \ref{ThmBessel}. In Section \ref{SectionDetStructure} we prove the existence of some ``signed" determinantal point process structure underlying the particle systems (\ref{IPS1}) and (\ref{IPS2}). In Section \ref{SectionBiorthogonalization} we obtain the explicit form of the kernel $\mathfrak{K}_t$ in Theorems \ref{ThmOneSidedUp} and \ref{ThmOneSidedDown}. In Section \ref{SectionRWKernel} we prove Theorem \ref{ThmBessel}. In Section \ref{SectionNonColliding} we prove Theorem \ref{NoncollidingThm}. In Section \ref{SectionConnection} we discuss the connection between the particle systems and their transition kernels.

\paragraph{Acknowledgements} I am grateful to Neil O'Connell for first introducing me to one-dimensional diffusions related to orthogonal polynomials. I am grateful to Jon Warren for discussions connected to Section 6 and in particular explanations related to Proposition \ref{PropIntertwining}. I am grateful to Alexander I. Bufetov for discussions related to Theorem \ref{NoncollidingThm}. I would also like to thank Daniel Remenik for interesting comments on a preliminary version of the paper and Brian Hall for correspondence on the backward in time diffusion flow. Finally, I am  grateful to the referees for a very careful reading of the paper and for very useful comments and suggestions which have improved the exposition.

\paragraph{Statements and declarations} The author has no conflicts of interest. No data has been used in the research described here.

\section{Some examples of diffusions and invariant measures}\label{SectionExamples}

In this section we give examples of diffusions that fall within our framework (satisfying our standing assumption in Definition \ref{StandingAssumption}), in particular to which Theorems \ref{ThmOneSidedUp}, \ref{ThmOneSidedDown} and \ref{NoncollidingThm} can all be applied. %Modulo some simple transformations these are basically all the examples that can be considered with $\mathsf{L}$ of the form (\ref{DiffusionDrift}). 
We also give references to where the explicit form of the transition densities $e^{t\mathsf{L}}(x,y)$ can be found along with other relevant facts on these diffusions, such as invariant measures. 
%As far as we can tell, except in the case of Brownian motions with one-sided collisions \cite{ReflectedBrownianKPZ,NicaQuastelRemenik} the interacting particle systems (\ref{IPS1}) and (\ref{IPS2}) have only been studied in \cite{InterlacingDiffusions} where they were introduced and their transition kernels computed explicitly, see Propositions \ref{TransDensityProp1} and \ref{TransDensityProp2} in the sequel. On the other hand, the non-colliding particle system (\ref{NonCollidingSDE}) has been introduced and studied for essentially all the diffusions we consider (with the possible exception of the Fisher-Snedecor diffusion below) although a complete exact solution in the sense of Theorem \ref{NoncollidingThm} has only appeared for the Brownian \cite{KatoriTanemura,JohanssonUniversality,brezin1997extension} and squared Bessel cases \cite{KatoriTanemuraBessel}.
Before discussing the explicit diffusion examples we record here a little result, that we refer back to below, on the invariant measure of the non-colliding SDEs (\ref{NonCollidingSDE}).

\begin{prop}\label{PropInvariantMeasure} Consider the measure on $(l,r)$ with density
\begin{equation}\label{SpeedMeasure}
\mathsf{m}(y)=\frac{1}{Z_1} \frac{1}{\left(a_2y^2+a_1y+a_0\right)}\exp\left(\int^y\frac{b_1z+b_0}{a_2z^2+a_1z+a_0}dz\right),
\end{equation}
where the indefinite integral $\int^yf(z)dz$ denotes an anti-derivative of $f$ and $Z_1$ is a normalization constant, which we assume is finite and so that $\int_l^r\mathsf{m}(y)dy=1$ (in particular $\mathsf{m}$ is unambiguously defined). Moreover, assume that this probability
measure has moments of order at least $2N-2$, namely $\int_l^r |y|^{2N-2}\mathsf{m}(y)dy<\infty$. Then, the unique invariant probability measure of the dynamics (\ref{NonCollidingSDE}) with semigroup (\ref{NonCollidingSemigroup}) in $\mathbb{W}_N^\uparrow$ is given by:
\begin{equation}\label{InvariantMeasure}
\mathsf{M}_N(dx)=\frac{1}{Z_N}\prod_{i=1}^N\mathsf{m}(x_i)\times \mathsf{\Delta}_N^2(x)\mathbf{1}_{\left(x\in \mathbb{W}_N^\uparrow\right)}dx,
\end{equation}
for some normalization constant $Z_N$.
\end{prop}

\begin{proof}
The proof is standard. The integrability assumption above gives that $\mathsf{M}_N$ in (\ref{InvariantMeasure}) is a well-defined probability measure (since the highest degree of the monomials appearing in $\mathsf{\Delta}_N^2(x)$ is $2N-2$). Note that, (\ref{SpeedMeasure}) is the formula for the density of the so-called speed measure of the $\mathsf{L}$-diffusion with respect to which $\mathsf{L}$ is reversible, see for example \cite{ItoMckean,HandbookBM,KarlinTaylor}. Then, the rest of the argument is word for word the same as the proof of Proposition 4.4 of \cite{HuaPickrell} to which we refer the reader (there a special $\mathsf{L}$-diffusion, the Hua-Pickrell diffusion below, is considered but the argument is completely generic).
\end{proof}

\begin{rmk}
The measure (\ref{InvariantMeasure}) is exactly an orthogonal polynomial ensemble with weight $\mathsf{m}$, see \cite{ForresterBook}. The corresponding point process on $(l,r)$ is determinantal with kernel given in terms of the orthogonal polynomials corresponding to $\mathsf{m}$, see \cite{ForresterBook}.
\end{rmk}

\paragraph{Brownian motion} We have the generator $\mathsf{L}=\frac{1}{2}\frac{d^2}{dx^2}$ in $(-\infty,\infty)$. Both $-\infty$ and $\infty$ are natural boundary points, see \cite{HandbookBM}. The transition kernel is simply the heat kernel and there is no invariant probability measure. As mentioned, the model of Brownian motions with one-sided collisions (\ref{IPS1}), (\ref{IPS2}), in its many equivalent forms, has been heavily studied in the last two decades, see \cite{OConnellYor,Warren,ReflectedBrownianKPZ,NicaQuastelRemenik,DirectedLandscape}. The non-colliding particle system (\ref{NonCollidingSDE}) in this case is called Dyson's Brownian motion, it arises as the eigenvalue evolution of Brownian motion on Hermitian matrices and has been intensely studied from many points of view for decades \cite{Dyson,IntroductionToRandomMatrices,UniversalityBook,JohanssonUniversality,KatoriTanemura,SpohnDyson,Tsai}. %It would also be possible to add a linear drift $\beta$ to get the generator of Brownian motion with drift $\mathsf{L}=\frac{1}{2}\frac{d^2}{dx^2}+\beta\frac{d}{dx}$ (models with drifts have also been considered in the literature).

\paragraph{Ornstein-Uhlenbeck process} This generalises the Brownian motion. We have the generator $\mathsf{L}=\frac{1}{2}\frac{d^2}{dx^2}-\gamma \frac{d}{dx}$ in $(-\infty,\infty)$. Both $-\infty$ and $\infty$ are natural boundary points for any $\gamma$, see \cite{HandbookBM}. The transition kernel can be written explicitly in terms of a series involving Hermite polynomials \cite{Wong} or as a closed expression using Mehler's formula, see \cite{HandbookBM}. For $\gamma>0$ the invariant measure is given by the Gaussian distribution and, by virtue of Proposition \ref{PropInvariantMeasure} for example, the invariant measure for the non-colliding diffusions in  (\ref{NonCollidingSDE}) is given by the law of the eigenvalues of the Gaussian unitary ensemble, see \cite{ForresterBook}. The non-colliding particle system (\ref{NonCollidingSDE}) arises as the eigenvalues of a Hermitian Ornstein-Uhlenbeck process and has been studied as much as the standard Brownian model \cite{Dyson,UniversalityBook,JohanssonHahn}. %Again, it would also be possible to add a linear drift $\beta$ to get a diffusion with generator $\mathsf{L}=\frac{1}{2}\frac{d^2}{dx^2}+\left(\beta-\gamma x\right)\frac{d}{dx}$.

\paragraph{Geometric Brownian motion} We have the generator $\mathsf{L}=\sigma^2x^2\frac{d^2}{dx^2}+\beta x \frac{d}{dx}$ in $(0,\infty)$. Both $0$ and $\infty$ are natural boundary points for any $\sigma,\beta$, see \cite{HandbookBM}. The transition kernel can easily be obtained from the heat kernel and there is no invariant probability measure, see \cite{HandbookBM}. The non-colliding particle system (\ref{NonCollidingSDE}) in this case was briefly discussed in \cite{InterlacingDiffusions} but, as far as we can tell, a canonical matrix diffusion has not been considered somewhere. We note that non-colliding geometric Brownian motions of (\ref{NonCollidingSDE}) do not arise by simply exponentiating Dyson Brownian motion (as can be checked by applying It\^{o}'s formula).

\paragraph{Squared Bessel process} We have the generator $\mathsf{L}=2x\frac{d^2}{dx^2}+\gamma \frac{d}{dx}$ in $(0,\infty)$. For $\gamma \ge 2$, the point $0$ is an entrance boundary point while $\infty$ is always natural for any $\gamma$, see \cite{HandbookBM}. The transition kernel of $\mathsf{L}$ we have already given explicitly in (\ref{BESQtransition}) and there is no invariant probability measure, see \cite{HandbookBM}. The non-colliding diffusions (\ref{NonCollidingSDE}) arise as the eigenvalues of the so-called Laguerre matrix process on positive definite Hermitian matrices introduced in \cite{KonigOConnell} and further studied in \cite{DemniLaguerre}. The corresponding diffusion on real symmetric positive definite matrices appeared earlier and is known as the Wishart process \cite{Wishart}.

\paragraph{Squared radial Ornstein-Uhlenbeck process} This generalises the squared Bessel process. We have the generator $\mathsf{L}=2x\frac{d^2}{dx^2}+\left(-\gamma_1 x+\gamma_2\right) \frac{d}{dx}$ in $(0,\infty)$. For $\gamma_2 \ge 2$, the point $0$ is an entrance boundary point while $\infty$ is always natural for any $\gamma_1,\gamma_2$, see \cite{HandbookBM}.  The transition kernel can be written explicitly in terms of a series involving Laguerre polynomials \cite{Wong} or as a closed expression involving Bessel functions, see \cite{HandbookBM}. For $\gamma_1>0$ the invariant measure is given by the Gamma distribution and, by virtue of Proposition \ref{PropInvariantMeasure} for example, the invariant measure for (\ref{NonCollidingSDE}) is given by the law of the eigenvalues of the Laguerre unitary ensemble, see \cite{ForresterBook}. The non-colliding diffusions (\ref{NonCollidingSDE}) arise as the eigenvalues of the Ornstein-Uhlenbeck analogue of the Laguerre matrix process on positive definite Hermitian matrices. %(see for example \cite{Wishart} for the real symmetric case which is completely analogous).

\paragraph{Jacobi process} We have the generator $\mathsf{L}=2x(1-x)\frac{d^2}{dx^2}+\left(-(\gamma_1+\gamma_2)x+\gamma_2\right)\frac{d}{dx}$ in $(0,1)$. For $\gamma_1,\gamma_2 \ge 2$ both the points $0$ and $1$ are entrance boundary points, see \cite{AlbaneseKuznetsov}. The transition kernel can be written explicitly as a series involving Jacobi polynomials, see \cite{Wong}. For $\gamma_1,\gamma_2>0$ the invariant measure is given by the beta distribution and, by virtue of Proposition \ref{PropInvariantMeasure} for example, the invariant measure for (\ref{NonCollidingSDE}) is given by the law of the eigenvalues of the Jacobi unitary ensemble, see \cite{ForresterBook}. The non-colliding particle system (\ref{NonCollidingSDE}) arises as the evolution of eigenvalues of the so-called matrix Jacobi diffusion introduced in \cite{Doumerc}. 

\paragraph{Inverse gamma diffusion} This is also known as inhomogeneous geometric Brownian motion. We have the generator $\mathsf{L}=2x^2\frac{d^2}{dx^2}+\left(\gamma_1x+\gamma_2\right)\frac{d}{dx}$ in $(0,\infty)$. For $\gamma_2>0$ the point $0$ is an entrance boundary while $\infty$ is always a natural boundary point for any $\gamma_1,\gamma_2$, see \cite{InhomogeneousGeomBM}. An explicit formula for the transition kernel in terms of the Bessel orthogonal polynomials and hypergeometric functions can be found in \cite{Wong}. For $\gamma_1<2$ (still assuming $\gamma_2> 0$) the invariant measure is the inverse Gamma distribution, and by virtue of Proposition \ref{PropInvariantMeasure} for example (we need a stricter restriction on the parameters to get finite moments), the invariant measure for (\ref{NonCollidingSDE}) is given by the law of eigenvalues of the inverse Laguerre ensemble, \cite{ForresterBook}. As far as we can tell, the non-colliding diffusions (\ref{NonCollidingSDE}) for this particular diffusion, and the corresponding matrix process, first appeared in the work of Rider and Valk\'{o} \cite{RiderValko} in relation to a matrix extension of Dufresne's identity \cite{Dufresne}. 

\paragraph{Hua-Pickrell diffusion} We have the generator $\mathsf{L}=(1+x^2)\frac{d^2}{dx^2}+\left(\gamma_1x+\gamma_2\right)\frac{d}{dx}$ in $(-\infty,\infty)$. Both $-\infty$ and $\infty$ are always natural boundary points for any $\gamma_1,\gamma_2$, see \cite{HuaPickrell}. An explicit formula for the transition kernel in terms of the Romanovski orthogonal polynomials and hypergeometric functions can be found in \cite{Wong,arista2022explicit}. For $\gamma_1<1$ the invariant measure is the Pearson IV distribution (the Cauchy distribution is the special case $\gamma_1=\gamma_2=0$), see \cite{HuaPickrell}, and by virtue of Proposition \ref{PropInvariantMeasure}  for example (we need a stricter restriction on the parameters to get finite moments), the invariant measure for (\ref{NonCollidingSDE}) is given by the law of eigenvalues of the so-called Hua-Pickrell or Cauchy unitary ensemble, see \cite{BorodinOlshanskiErgodic,ForresterWitteCauchy,ForresterBook,HuaPickrell}.  We introduced and studied the non-colliding diffusions (\ref{NonCollidingSDE}), and the associated matrix process, in \cite{HuaPickrell} motivated by the construction of infinite-dimensional dynamics associated to a certain determinantal point process with infinitely many points and a matrix extension \cite{MatrixBougerol} of Bougerol's identity \cite{Bougerol}. 

\paragraph{Fisher-Snedecor diffusion} We have the generator $\mathsf{L}=2x(1+x)\frac{d^2}{dx^2}+\left(\gamma_1x+\gamma_2\right)\frac{d}{dx}$ in $(0,\infty)$. For $\gamma_2 \ge 2$ the point $0$ is an entrance boundary point while $\infty$ is always natural for any $\gamma_1,\gamma_2$. An explicit formula for the transition kernel in terms of the Fisher-Snedecor orthogonal polynomials and hypergeometric functions can be found in \cite{FisherSnedecor}. For $\gamma_1,\gamma_2>0$ the invariant measure of the diffusion is the beta prime distribution and by virtue of Proposition \ref{PropInvariantMeasure} (we need a stricter restriction on the parameters to get finite moments) the invariant measure of the non-colliding diffusions (\ref{NonCollidingSDE}) is the law of the eigenvalues of the Hermitian matrix beta prime distribution \cite{GuptaNagar}. As far as we can tell, the non-colliding particle system  for this diffusion (\ref{NonCollidingSDE}) had not been considered in the literature before and a corresponding matrix diffusion has not yet been introduced.

\section{Backward in time diffusion flow on polynomials}\label{SectionBackwardFlow}

In this section we show that the definition in (\ref{powerseriesdef}) of $e^{t\mathsf{L}}p$ as a power series makes sense and it is consistent, for $t>0$, with integration against the transition kernel of the $\mathsf{L}$-diffusion. We also prove some other facts about the backward in time diffusion flow on polynomials which, although not strictly necessary for subsequent developments, might be of independent interest. Write $[z^i]p(z)$ for the coefficient of the term $z^i$ in a polynomial $p(z)$.

\begin{prop}\label{WellDefinedProp}
Let $M\ge 1$ and $p \in \mathfrak{P}_M$, where $\mathfrak{P}_M$ is the set of polynomials (with complex coefficients) of degree $M$. Let $t\in \mathbb{C}$. Then, the linear map 
\begin{equation*}
    e^{t\mathsf{L}}:\mathfrak{P}_M \longrightarrow \mathfrak{P}_M,
\end{equation*}
given by (\ref{powerseriesdef}) is well-defined, with the series (\ref{powerseriesdef}) converging uniformly on compact sets in $(t,z)\in \mathbb{C}^2$. Moreover, we have:
\begin{equation}\label{LeadingCoeff}
[z^M]e^{t\mathsf{L}}p(z)=e^{tM(b_1+a_2(M-1))}[z^M]p(z).  
\end{equation}
\end{prop}

\begin{proof} Consider a polynomial $p(z)=\sum_{i=0}^M \epsilon_i z^i \in \mathfrak{P}_M$. Differentiating we get
\begin{equation*}
    \frac{d}{dz}p(z)=\sum_{i=0}^{M-1}(i+1)\epsilon_{i+1}z^i, \ \ \frac{d^2}{dz^2}p(z)=\sum_{i=0}^{M-2}(i+1)(i+2)\epsilon_{i+2} z^i.
\end{equation*}
Define $p_j(z)=\mathsf{L}^jp(z)=\sum_{i=0}^M\epsilon_i^{(j)}z^i$ and observe that this is a polynomial of degree at most $M$.
We have the following recurrence for the coefficients $\epsilon_i^{(j)}$ in $j$:
\begin{align*}
\epsilon_i^{(j+1)}&=a_2 i(i-1)\epsilon_i^{(j)}+a_1(i+1)i\epsilon_{i+1}^{(j)}+a_0 (i+1)(i+2) \epsilon_{i+2}^{(j)}+b_1 i \epsilon_i^{(j)}+b_0(i+1)\epsilon_{i+1}^{(j)}\\
&=\epsilon_i^{(j)}[a_2i(i-1)+b_1i]+\epsilon_{i+1}^{(j)}[a_2i(i+1)+b_0(i+1)]+a_0(i+1)(i+2)\epsilon_{i+2}^{(j)}.
\end{align*}
In particular, there is some finite constant $\eta$, depending only on $a_0, a_1, a_2, b_0, b_1$ and $i$ (recall that $i=0,\dots, M$ is fixed) such that:
\begin{equation*}
\left|\epsilon_i^{(j+1)}\right|\le \eta \left(\left|\epsilon_i^{(j)}\right|+\left|\epsilon_{i+1}^{(j)}\right|+\left|\epsilon_{i+2}^{(j)}\right|\right).
\end{equation*}
By summing this relation over $i$ we get that, for a different constant $\tilde{\eta}$, independent of $j$:
\begin{equation*}
\sum_{i=0}^M\left|\epsilon_i^{(j+1)}\right|\le \tilde{\eta} \sum_{i=0}^M\left|\epsilon_i^{(j)}\right|.
\end{equation*}
By iterating, we finally get that, for any $i=1,\dots,M$, for some constant $\eta^*$ independent of $j$,
$\left|\epsilon_i^{(j)}\right| \le \sum_{m=1}^M \left|\epsilon_m^{(j)}\right| \le (\eta^*)^j$. Suppose now that $(t,z)\in \mathcal{V}$, where $\mathcal{V}$ is an arbitrary compact subset of $\mathbb{C}^2$. Then, we have the following bound for any $j$:
\begin{equation*}
 \left|\frac{t^j\mathsf{L}^j}{j!}p(z) \right|\le  \frac{|t|^j}{{j!}}\sum_{i=0}^M\left|\epsilon_i^{(j)}\right||z|^i \le\frac{|t|^j\left(\eta^*\right)^j}{{j!}}\sum_{i=0}^M|z|^i   \le \frac{C_{\mathcal{V},\eta^*}^j}{j!},
\end{equation*}
for some finite constant $C_{\mathcal{V},\eta^*}$. By the Weirstrass M-test this gives that $e^{t\mathsf{L}}p(z)$ is well-defined, being a polynomial of degree at most $M$, with the series (\ref{powerseriesdef}) converging uniformly in $(t,z)$ on compact sets in $\mathbb{C}^2$. Finally, using the recurrence for $\epsilon_M^{(j)}$ we obtain that for $p\in \mathfrak{P}_M$:
\begin{equation*}
[z^M]\mathsf{L}^jp(z)=\left(b_1 M+a_2M(M-1)\right)^j[z^M]p(z),
\end{equation*}
from which both (\ref{LeadingCoeff}) and the first claim of the proposition that the degree of $e^{t\mathsf{L}}p$ is exactly $M$ follow (since by (\ref{LeadingCoeff}) $[z^M]e^{t\mathsf{L}}p(z)\neq 0$).
\end{proof}

\begin{rmk}
The operator $e^{t\mathsf{L}}$ is well-defined as a power series not only on polynomials but more generally on analytic functions with certain growth conditions. Since we do not need this in this paper we do not pursue it further.
\end{rmk}

\begin{lem}\label{LemCoincidence}
Let $p$ be a polynomial. Then, for $t>0$ and $x\in (l,r)$, the definition of $e^{t\mathsf{L}}p(x)$ in (\ref{powerseriesdef}) coincides with 
\begin{equation}\label{ExpectationDisplay}
    \int_{l}^r e^{t\mathsf{L}}(x,y)p(y)dy,
\end{equation}
where recall that $e^{t\mathsf{L}}(x,y)$ is the transition density with respect to the Lebesgue measure of an $\mathsf{L}$-diffusion in $(l,r)$.
\end{lem}

\begin{proof}
We first claim that the classical solution to the partial differential equation with $(t,x)\in [0,\infty) \times (l,r)$:
\begin{equation}\label{pde}
 \partial_t u(t,x)=\mathsf{L}_xu(t,x)  , \ \ u(0,x)=p(x),
\end{equation}
subject to the bound, for any $T>0$, $|\partial_x u(t,x)|\le h_T(x)$ uniformly in $t\in [0,T]$ for some non-negative polynomial $h_T$, is unique and given by the probabilistic representation $u(t,x)=\mathbb{E}_x\left[p(\mathsf{x}(t))\right]$, where $(\mathsf{x}(t);t\ge 0)$ is an $\mathsf{L}$-diffusion starting from $x$, which is exactly the expression in (\ref{ExpectationDisplay}). The claim can be proven as follows. Let $u$ be an arbitrary solution of (\ref{pde}) subject to the above conditions. Applying It\^{o}'s formula \cite{RevuzYor} we see that, for any $t$, the process $\left(s\mapsto u(t-s,\mathsf{x}(s));0\le s \le t\right)$ is a local martingale by virtue of (\ref{pde}). Moreover, by virtue of the polynomial bound in the condition above and the fact that the $\mathsf{L}$-diffusion has polynomial moments we obtain that the quadratic variation of $\left(s\mapsto u(t-s,\mathsf{x}(s));0\le s \le t\right)$ is integrable and hence it is a true martingale \cite{RevuzYor}. The optional stopping theorem \cite{RevuzYor} gives the desired claim.

Finally, it is easily seen that $e^{t\mathsf{L}}p(x)$ defined in (\ref{powerseriesdef}) also solves (\ref{pde}) and satisfies the required bound; we can justify differentiating term by term by adapting the proof of Proposition \ref{WellDefinedProp}.
The conclusion follows.
\end{proof}

\begin{lem}\label{LemmaSemigroup}
Let $p$ be a polynomial and $s,t\in \mathbb{C}$. We have $e^{s\mathsf{L}}e^{t\mathsf{L}}p=e^{(s+t)\mathsf{L}}p$.
\end{lem}

\begin{proof}
This follows from the power series definition in (\ref{powerseriesdef}).
\end{proof}

\begin{prop}\label{MartingalePolynomials}
Let $p$ be a polynomial and $\left(\mathsf{x}(t);t\ge 0\right)$ be a realisation of the $\mathsf{L}$-diffusion, with $\mathsf{x}(0)=x$ deterministic, with natural filtration $\left(\mathcal{F}_s\right)_{s\ge 0}$. Then, $\left(\left[e^{-t\mathsf{L}}p\right](\mathsf{x}(t))\right)_{t\ge 0}$ is a martingale.
\end{prop}

\begin{proof}
First note that, since the $\mathsf{L}$-diffusion has polynomial moments, $\mathbb{E}_x\left[\left|e^{-t\mathsf{L}}p\left(\mathsf{x}(t)\right)\right|\right]<\infty$. Then, for $s\le t$, by the Markov property and Lemma \ref{LemmaSemigroup} we have:
\begin{equation*}
 \mathbb{E}\left[e^{-t\mathsf{L}}p(\mathsf{x}(t))\big|\mathcal{F}_s\right]=\left[e^{(t-s)\mathsf{L}}e^{-t\mathsf{L}}p\right](\mathsf{x}(s))=\left[e^{-s\mathsf{L}}p\right](\mathsf{x}(s)),
\end{equation*}
and the conclusion follows.
\end{proof}

\begin{rmk}
Suppose that $p$ is a polynomial eigenfunction of $\mathsf{L}$ with eigenvalue $\lambda$, $\mathsf{L}p(x)=\lambda p(x)$. Then, it is easy to see from (\ref{powerseriesdef}), that for any $t\in \mathbb{C}$, $e^{t\mathsf{L}}p(x)=e^{\lambda t}p(x)$.
\end{rmk}

We finally record here the equations of motion for the zeroes of $e^{t\mathsf{L}}p(z)$. This is of interest to us since applying the diffusion flow backward in time to the polynomial $\prod_{i=1}^N(z-x_i)$, namely with roots given by the initial condition $x=(x_1,\dots,x_N)\in \mathbb{W}_N^{\uparrow}$ of the SDEs (\ref{NonCollidingSDE}), plays a key role in Theorem \ref{NoncollidingThm}. Observe how similar (essentially we remove the noise and reverse time) these equations are to the interacting SDE (\ref{NonCollidingSDE}).

\begin{prop}
Let $p \in \mathfrak{P}_M$. Let $\mathfrak{z}_1(t), \dots, \mathfrak{z}_M(t)$ be the zeroes of $e^{t\mathsf{L}}p(z)$. Suppose that for some $t_0\in \mathbb{C}$ the zeroes $\mathfrak{z}_1(t_0), \dots, \mathfrak{z}_M(t_0)$ are distinct. Then, for all $t\in \mathbb{C}$ in a neighbourhood of $t_0$, each $\mathfrak{z}_i(t)$ depends holomorphically on $t$ and we have the following equations of motion, for $i=1,\dots,M$:
\begin{equation}\label{RootsEquation}
\partial_t \mathfrak{z}_i(t)=-\mathsf{b}(\mathfrak{z}_i(t))+2\mathsf{a}(\mathfrak{z}_i(t))\sum_{j\neq i} \frac{1}{\mathfrak{z}_j(t)-\mathfrak{z}_i(t)}.
\end{equation}
\end{prop}

\begin{proof}
The proof is a straightforward adaptation of the argument in \cite{TaoHeatFlow,hall2022heat} given there for the heat kernel/Brownian case, see in particular Proposition 2.7 in \cite{hall2022heat} for more details. Local holomorphicity of the roots $\mathfrak{z}_i(t)$ in $t$ is a consequence of the (holomorphic) implicit function theorem, since by assumption the $\mathfrak{z}_i(t_0)$ are distinct, we have that $\partial_z e^{t_0\mathsf{L}}p$ is non-zero at $\mathfrak{z}_i(t_0)$. Hence, implicitly differentiating  in $t$ in a neighbourhood of $t_0$ the equation $e^{t\mathsf{L}}p\left(\mathfrak{z}_i(t)\right)=0$, using the fact $\partial_t e^{t\mathsf{L}}p(z)=\mathsf{L}_z e^{t\mathsf{L}}p(z)$, for any $z\in \mathbb{C}$, we get the equation:
\begin{equation*}
\mathsf{L}_z e^{t\mathsf{L}}p\left(\mathfrak{z}_i(t)\right)+\partial_t\mathfrak{z}_i(t)\partial_z e^{t\mathsf{L}}p\left(\mathfrak{z}_i(t)\right)=0.
\end{equation*}
After some computations, see \cite{TaoHeatFlow,hall2022heat} for more details, we obtain (\ref{RootsEquation}).
\end{proof}

\section{Proof of Theorems \ref{ThmOneSidedUp}, \ref{ThmOneSidedDown} and \ref{ThmBessel}}\label{SectionMainProofs}
\subsection{Existence of determinantal structure}\label{SectionDetStructure}

In this subsection we prove that there is a hidden ``signed" determinantal point process structure underlying the interacting particle systems (\ref{IPS1}) and (\ref{IPS2}). We need some additional notation. Recall that we defined the operator $\partial^{-1}$ in (\ref{IntegrationOperDef}) by $\partial^{-1}f(x)=\int_{l}^x f(y)dy$. Moreover, define the operator $\hat{\partial}^{-1}$, on functions $f$ in $(l,r)$ which integrate polynomials, by:
\begin{equation}\label{IntegralOp}
\hat{\partial}^{-1}f(x)=-\int_{x}^rf(y)dy.
\end{equation}
It is easy to see that, for any $m\in \mathbb{N}$,
\begin{equation}
 \hat{\partial}^{-m}f(x)=\underbrace{\hat{\partial}^{-1}\cdots \hat{\partial}^{-1}}_{\textnormal{m  times}}f(x)=-\int_x^r\frac{(x-y)^{m-1}}{(m-1)!}f(y)dy.
\end{equation}
For $n\in \mathbb{N}$, we use the convention $\hat{\partial}^n=\partial^n$.

In what follows we will make frequent use of the following observations.
\begin{lem}\label{LemmaObservations}
Let $m,n \in \mathbb{N}$. Suppose $f$ is a smooth function and integrates polynomials in $(l,r)$, namely its integral against polynomials in $(l,r)$ is finite. Then, we have
\begin{align*}
\partial^m\partial^n f&=\partial^{n}\partial^{m}f=\partial^{m+n}f,\\
\partial^{-m}\partial^{-n}f&=\partial^{-n}\partial^{-m}f=\partial^{-n-m}f,\\
\partial^m\partial^{-n}f&=\partial^{m-n}f.
\end{align*}
Moreover, for $k=1,\dots,N$ and $t>0$, we have 
\begin{align}
\partial_y^{m}\partial_x^n e^{t\mathsf{L}^{(k)}}(x,y)&=\partial_x^n\partial_y^m e^{t\mathsf{L}^{(k)}}(x,y),\label{PartialCommutation}\\
\partial_y^{-m}\partial_x^n e^{t\mathsf{L}^{(k)}}(x,y)&=\partial_x^n\partial_y^{-m} e^{t\mathsf{L}^{(k)}}(x,y)\label{IntDerCommutation}.
\end{align}
All the relations above also hold with $\partial$ replaced by $\hat{\partial}$.
\end{lem}

\begin{proof}
The first three relations are immediate by the definitions. Relation (\ref{PartialCommutation}) follows from smoothness of $e^{t\mathsf{L}^{(k)}}(x,y)$ in $(x,y)\in (l,r)^2$ for $t>0$ while (\ref{IntDerCommutation}) is valid by virtue of the dominated convergence theorem since by standard estimates \cite{StroockPDEbook}, see \cite{HuaPickrell} where this is worked out in detail in the Hua-Pickrell case, we get that for a small neighbourhood $\mathcal{U}_x\subset (l,r)$ of $x$, the function $\sup_{\xi \in \mathcal{U}_x}\big|\partial_\xi^n e^{t\mathsf{L}^{(k)}}(\xi,\cdot)\big|$ integrates polynomials. Finally, it is easy to see that $\partial$ can be replaced by $\hat{\partial}$ in all the above.
\end{proof}

We note that, for $m,n \in \mathbb{N}$, $\partial^{-m} \partial^n f$ is in general not equal to $\partial^{-m+n}f=\partial^n\partial^{-m}f$. We would need to enforce boundary conditions at $l$ (respectively $r$ for $\hat{\partial}$) for $f$ for this to be true, for example $\partial^{-1}\partial f(x)=f(x)-f(l)$, $\hat{\partial}^{-1}\hat{\partial}f(x)=f(x)-f(r)$. However, we will not need to make use of this.

The starting point of our analysis are the following explicit expressions for the semigroups  $\mathsf{S}_t^{\uparrow,(N)}$ and $\mathsf{S}_t^{\downarrow,(N)}$ obtained in \cite{InterlacingDiffusions}. These are usually called Schutz type formulae because of the seminal work of Schutz \cite{Schutz} on the transition probabilities of TASEP.

\begin{prop}\label{TransDensityProp1} Under the standing assumption in Definition \ref{StandingAssumption}, the transition kernel of the interacting particle system (\ref{IPS1}) in $\mathbb{W}_N^{\uparrow}$ is given by, with $x=(x_1,\dots,x_N),y=(y_1,\dots,y_N)\in \mathbb{W}_N^\uparrow$,
\begin{equation}\label{TransitionDensityIPS1Dis}
\mathsf{S}_t^{\uparrow,(N)}(x,dy)=\det\left(\left[\partial^{j-i}_{\cdot}e^{t\mathsf{L}^{(i)}}(x_i,\cdot)\right](y_j)\right)_{i,j=1}^Ndy.
\end{equation}
\end{prop}

\begin{prop}\label{TransDensityProp2} Under the standing assumption in Definition \ref{StandingAssumption}, the transition kernel of the interacting particle system (\ref{IPS2}) in $\mathbb{W}_N^{\downarrow}$ is given by, with $x=(x_1,\dots,x_N),y=(y_1,\dots,y_N)\in \mathbb{W}_N^\downarrow$,
\begin{equation}
\mathsf{S}_t^{\downarrow,(N)}(x,dy)=\det\left(\left[\hat{\partial}^{j-i}_{\cdot}e^{t\mathsf{L}^{(i)}}(x_i,\cdot)\right](y_j)\right)_{i,j=1}^Ndy.
\end{equation}

\end{prop}

The following identities, proved in Section 13.4 of \cite{InterlacingDiffusions} and used to obtain the explicit form of the transition kernels in Propositions \ref{TransDensityProp1} and \ref{TransDensityProp2} will be important below,
\begin{align}
 e^{t\mathsf{L}^{(j)}}(x,y)&=-e^{-t\mathsf{c}^{(j)}}\partial_y^{-1}\partial_xe^{t\mathsf{L}^{(j+1)}}(x,y),\label{KeyIdentity1}\\
 e^{t\mathsf{L}^{(j)}}(x,y)&=-e^{-t\mathsf{c}^{(j)}}\hat{\partial}_y^{-1}\partial_xe^{t\mathsf{L}^{(j+1)}}(x,y).\label{KeyIdentity2}
\end{align}
Observe that, since $e^{t\mathsf{L}^{(j+1)}}$ is a bona-fide Markov semigroup ($e^{t\mathsf{L}^{(j+1)}}\mathbf{1}=\mathbf{1}$) then the two identities are actually equivalent. 

The proposition below is the key step in the proof of existence of a determinantal structure for the interacting particle system (\ref{IPS1}). In the translation invariant (in space) setting the argument has been employed several times \cite{BorodinFerrariPushASEP,BorodinFerrariPrahoferSasamoto,BorodinFerrariPrahofer,BorodinFerrariSasamoto} and is known as Sasamoto's trick \cite{Sasamoto}. The fact that something analogous can be done in the non-translation invariant setting we consider here is non-trivial and crucially relies on the equation (\ref{KeyIdentity1}) above. It is worth noting that the set $\mathcal{D}_N$ in (\ref{SetDN}) below is very closely related to the notion of interlacing arrays $\mathbb{IA}_N$ defined in (\ref{InterlacingArrayDefinition}) as we explain in Remark \ref{InterlacingComparison}. Moreover, the measure in (\ref{SignedMeasure}) below actually coincides, for special initial conditions, with the distribution at time $t$ of certain dynamics (\ref{DynamicsArray}) in interlacing arrays, see Proposition \ref{MultilevelProp} and Remark \ref{EntranceLaw}. Finally, a more subtle connection between the signed measure (\ref{SignedMeasure}) and the dynamics (\ref{DynamicsArray}) exists beyond these special initial conditions, see the discussion around Proposition \ref{PropAlternativeRoute} (which also gives an alternative approach to Proposition \ref{DeterminantalProp1}).

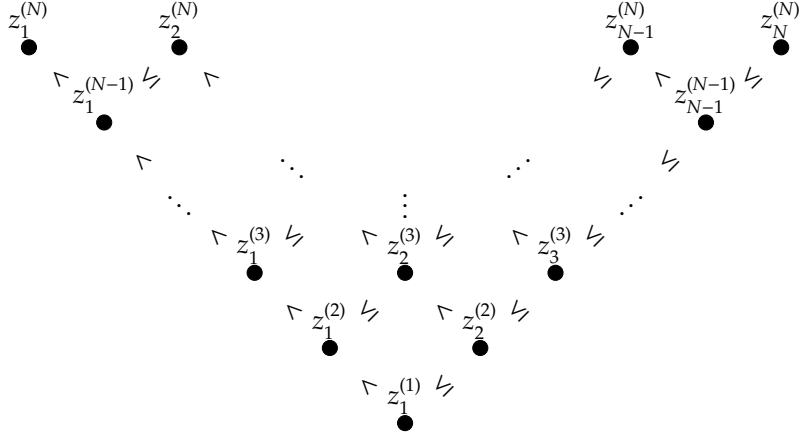
\begin{figure}
\centering
\captionsetup{singlelinecheck = false, justification=justified}
\begin{tikzpicture}

\draw[fill] (5,0) circle [radius=0.1];
\node[above ] at (5,0) {$z_1^{(1)}$};

\draw[fill] (4,1) circle [radius=0.1];
\node[above ] at (4,1) {$z_1^{(2)}$};

\draw[fill] (6,1) circle [radius=0.1];
\node[above ] at (6,1) {$z_2^{(2)}$};

\draw[fill] (3,2) circle [radius=0.1];
\node[above ] at (3,2) {$z_1^{(3)}$};

\draw[fill] (5,2) circle [radius=0.1];
\node[above ] at (5,2) {$z_2^{(3)}$};

\draw[fill] (7,2) circle [radius=0.1];
\node[above ] at (7,2) {$z_3^{(3)}$};

\node[] at (5,3) {$\vdots$};

\node[] at (8,3) {$\iddots$};

\node[] at (2,3) {$\ddots$};

\draw[fill] (1,4) circle [radius=0.1];
\node[above ] at (1,4) {$z_1^{(N-1)}$};

\draw[fill] (9,4) circle [radius=0.1];
\node[above ] at (9,4) {$z_{N-1}^{(N-1)}$};

\draw[fill] (0,5) circle [radius=0.1];
\node[above ] at (0,5) {$z_1^{(N)}$};

\draw[fill] (2,5) circle [radius=0.1];
\node[above ] at (2,5) {$z_2^{(N)}$};

\draw[fill] (8,5) circle [radius=0.1];
\node[above ] at (8,5) {$z_{N-1}^{(N)}$};

\draw[fill] (10,5) circle [radius=0.1];
\node[above ] at (10,5) {$z_{N}^{(N)}$};

\node[] at (4.5,0.5) {$\dless$};

\node[] at (3.5,1.5) {$\dless$};

\node[] at (0.4,4.6) {$\dless$};

\node[] at (1.5,3.5) {$\dless$};

\node[] at (5.5,1.5) {$\dless$};

\node[] at (4.5,2.5) {$\dless$};

\node[] at (2.4,4.6) {$\dless$};

\node[] at (2.5,2.5) {$\dless$};

\node[] at (8.4,4.6) {$\dless$};

\node[] at (6.5,2.5) {$\dless$};

\node[] at (6.5,3.5) {$\iddots$};

\node[] at (3.5,3.5) {$\ddots$};

\node[] at (5.5,0.5) {$\dle$};

\node[] at (6.5,1.5) {$\dle$};

\node[] at (7.5,2.5) {$\dle$};

\node[] at (8.5,3.5) {$\dle$};

\node[] at (9.6,4.6) {$\dle$};

\node[] at (5.5,2.5) {$\dle$};

\node[] at (4.5,1.5) {$\dle$};

\node[] at (3.5,2.5) {$\dle$};

\node[] at (1.6,4.6) {$\dle$};

\node[] at (7.6,4.6) {$\dle$};

\end{tikzpicture}
\caption{A visualisation of an element $(z_i^{(n)})_{1\le i \le n \le N}$ in $\mathcal{D}_N$. The $n$-th row corresponds to the vector $(z_1^{(n)},\dots,z_n^{(n)})\in (l,r)^n$, which by virtue of the inequalities in (\ref{SetDN}) actually belongs to $\mathbb{W}_n^{\uparrow,\circ}$. We observe from the inequalities in (\ref{SetDN}) that the coordinates strictly increase as we go down the diagonals starting from the top row and going to the right edge while they increase (not necessarily strictly) as we go up the diagonals starting from the left edge and going towards the top row as illustrated in the figure.}\label{FigureDN}
\end{figure}

\begin{prop}\label{DeterminantalProp1}
Let $x=(x_1,\dots,x_N)\in \mathbb{W}_N^{\uparrow}$. The distribution in $\mathbb{W}_N^{\uparrow}$ at time $t$ of the interacting particle system (\ref{IPS1}) starting from $x$, namely, 
\begin{equation}\label{DistributionIPS1t}
\det\left(\left[\partial^{j-i}_{\cdot}e^{t\mathsf{L}^{(i)}}(x_i,\cdot)\right]\left(y^{(j)}_j\right)\right)_{i,j=1}^Ndy_1^{(1)}dy_2^{(2)}\cdots dy_N^{(N)}
\end{equation}
is the marginal in the $(y_1^{(1)},y_2^{(2)},\dots,y_N^{(N)})$ variables of the signed measure 
\begin{align}\label{SignedMeasure}
\left(-1\right)^{\frac{N(N-1)}{2}}e^{-t\sum_{k=1}^{N-1}k\mathsf{c}^{(k)}}\det\left(\partial_{x_i}^{N-i}e^{t\mathsf{L}}\left(x_i,y_j^{(N)}\right)\right)_{i,j=1}^N \mathbf{1}_{\left((y^{(1)},\dots,y^{(N)})\in\mathcal{D}_N\right)}\prod_{1\le i \le n \le N}dy_i^{(n)},
\end{align}
where the set $\mathcal{D}_N$ is given by, see Figure \ref{FigureDN} for an illustration,
\begin{equation}\label{SetDN}
\mathcal{D}_N=\left\{z_i^{(n)}\in (l,r); i=1,\dots,n; n=1,\dots,N:z_i^{(n+1)}<z_i^{(n)}\le z_{i+1}^{(n+1)}\right\}.
\end{equation}
\end{prop}

\begin{proof} Looking at the $m$-th column of the matrix in the determinant in (\ref{DistributionIPS1t}) we claim that the $i$-th entry can be written as (where the second equality is simply writing the $\partial^{-1}$ notation out):
\begin{align*}
    &\partial_{y_m^{(m)}}^{m-i}e^{t\mathsf{L}^{(i)}}\left(x_i,y_m^{(m)}\right)=(-1)^{N-i}e^{-t\sum_{k=i}^{N-1}\mathsf{c}^{(k)}}\partial^{-(N-m)}_{y^{(m)}_m}\partial_{x_i}^{N-i}e^{t\mathsf{L}^{(N)}}\left(x_i,y_m^{(m)}\right)\\
    &=(-1)^{N-i}e^{-t\sum_{k=i}^{N-1}\mathsf{c}^{(k)}}\int_l^{y_{m}^{(m)}}\int_l^{y_m^{(m+1)}}\cdots \int_l^{y_m^{(N-1)}}\partial_{x_i}^{N-i}e^{t\mathsf{L}^{(N)}}\left(x_i,y_m^{(N)}\right)dy_m^{(N)}\cdots y_m^{(m+1)},
\end{align*}
for some dummy variables $y_m^{(m+1)},\dots,y_m^{(N)}$ so that $y_m^{(N)}\le y_m^{(N-1)}\le \cdots \le y_m^{(m+1)}\le y_m^{(m)}$. This can be seen as follows. By repeated use of the identity (\ref{KeyIdentity1}) and Lemma \ref{LemmaObservations}, we have 
\begin{align*}
e^{t\mathsf{L}^{(i)}}\left(x_i,y_m^{(m)}\right)=-e^{-\mathsf{c}^{(i)}t}\partial_{y_m^{(m)}}^{-1}\partial_{x_i}e^{t\mathsf{L}^{(i+1)}}\left(x_i,y_m^{(m)}\right)=\cdots=(-1)^{N-i}e^{-t\sum_{k=i}^{N-1}\mathsf{c}^{(k)}}\partial_{y_m^{(m)}}^{-(N-i)}\partial_{x_i}^{N-i}e^{t\mathsf{L}^{(N)}}\left(x_i,y_m^{(m)}\right).
\end{align*}
Then, applying $\partial_{y_m^{(m)}}^{m-i}$ to both sides of the identity above gives the desired claim (observe that $-(N-i)\le 0$ and hence $\partial^{m-i}\partial^{-(N-i)}=\partial^{-(N-m)}$ from Lemma \ref{LemmaObservations}).

Hence, by multilinearity of the determinant, we can take the multiple integrals outside the determinant and we can write display (\ref{DistributionIPS1t}) (suppressing $\prod_{n=1}^Ndy_n^{(n)}$) as follows:
\begin{align}\label{IntermediateIntegralDisplay}
\left(-1\right)^{\frac{N(N-1)}{2}}e^{-t\sum_{k=1}^{N-1}k\mathsf{c}^{(k)}}\int_{\widehat{\mathcal{D}}_N    
}\det\left(\partial_{x_i}^{N-i}e^{t\mathsf{L}^{(N)}}\left(x_i,y_j^{(N)}\right)\right)_{i,j=1}^N \prod_{1\le i<n \le N}dy_i^{(n)},
\end{align}
where $\widehat{\mathcal{D}}_N=\widehat{\mathcal{D}}_N\left(y_1^{(1)},\dots,y_N^{(N)}\right)$ is given by
\begin{equation*}
\widehat{\mathcal{D}}_N=\left\{z_i^{(n)}\in (l,r): z_n^{(n)}=y_n^{(n)}; z_i^{(n+1)}\le z_i^{(n)}\right\}.
\end{equation*}
Since the determinant is antisymmetric in the $y_j^{(N)}$-variables we can then use Lemma 5.6 in \cite{ReflectedBrownianKPZ} to write display (\ref{IntermediateIntegralDisplay}) as 
\begin{align*}
\left(-1\right)^{\frac{N(N-1)}{2}}e^{-t\sum_{k=1}^{N-1}k\mathsf{c}^{(k)}}\int_{\widetilde{\mathcal{D}}_N    
}\det\left(\partial_{x_i}^{N-i}e^{t\mathsf{L}^{(N)}}\left(x_i,y_j^{(N)}\right)\right)_{i,j=1}^N \prod_{1\le i<n \le N}dy_i^{(n)},
\end{align*}
where $\widetilde{\mathcal{D}}_N=\widetilde{\mathcal{D}}_N\left(y_1^{(1)},\dots,y_N^{(N)}\right)$ is given by
\begin{equation*}
\widetilde{\mathcal{D}}_N=\left\{z_i^{(n)}\in (l,r): z_n^{(n)}=y_n^{(n)}; z_i^{(n+1)}\le z_i^{(n)}<z_{i+1}^{(n+1)}\right\}.
\end{equation*}
Finally, since we are integrating a continuous function, the integral remains unchanged if we change a strict inequality ``$<$" for an inequality ``$\le$" which, after
recalling that $\mathsf{L}^{(N)}\equiv \mathsf{L}$, concludes the proof.
\end{proof}

\begin{rmk}\label{InterlacingComparison}
We note that the set $\mathcal{D}_N$ is very closely related to the notion of interlacing arrays $\mathbb{IA}_N$ defined in (\ref{InterlacingArrayDefinition}). The only difference is that for $\mathcal{D}_N$ we require some of the inequalities, the ones appearing as we move down the diagonals from the top row to the right edge as depicted in Figure \ref{FigureDN} to be strict (while for interlacing arrays they do not need to be strict). Since the measures we study have continuous densities, considering such measures over $\mathcal{D}_N$ or over interlacing arrays (\ref{InterlacingArrayDefinition}) is essentially one and the same. The main reason for using $\mathcal{D}_N$ is that its indicator function can be written as a product of determinants of special form (and we need the strict inequalities for this to be true), see Lemma \ref{LemInterlacing} below, which will allow us to apply the Eynard-Mehta theorem \cite{JohanssonDeterminantal,BorodinDeterminantal,BorodinRains} in the sequel. 
\end{rmk}

An analogous result holds for the particle system (\ref{IPS2}), making use of the identity (\ref{KeyIdentity2}) now instead.

\begin{prop}\label{DeterminantalProp2}
Let $x=(x_1,\dots,x_N)\in \mathbb{W}_N^{\downarrow}$. The distribution in $\mathbb{W}_N^{\downarrow}$ at time $t$ of the interacting particle system (\ref{IPS2}) starting from $x$, namely, 
\begin{equation}\label{DistributionIPS2t}
\det\left(\left[\hat{\partial}^{j-i}_{\cdot}e^{t\mathsf{L}^{(i)}}(x_i,\cdot)\right]\left(y^{(j)}_1\right)\right)_{i,j=1}^Ndy_1^{(1)}dy_1^{(2)}\cdots dy_1^{(N)}
\end{equation}
is the marginal in the $(y_1^{(1)},y_1^{(2)},\dots,y_1^{(N)})$ variables of the signed measure in (\ref{SignedMeasure}), with $x\in \mathbb{W}_N^{\downarrow}$ instead.
\end{prop}

\begin{proof}
We argue as in the proof of Proposition \ref{DeterminantalProp1}. Looking at the $m$-th column of the matrix in the determinant in (\ref{DistributionIPS2t}) we claim that the $i$-th entry can be written as:
\begin{align*}
    \hat{\partial}_{y_m^{(m)}}^{m-i}e^{t\mathsf{L}^{(i)}}\left(x_i,y_m^{(m)}\right)&=(-1)^{N-i}e^{-t\sum_{k=i}^{N-1}\mathsf{c}^{(k)}}\hat{\partial}^{-(N-m)}_{y^{(m)}_m}\partial_{x_i}^{N-i}e^{t\mathsf{L}^{(N)}}\left(x_i,y_m^{(m)}\right)\\
    &=(-1)^{-i-m}e^{-t\sum_{k=i}^{N-1}\mathsf{c}^{(k)}}\int_{y_1^{(m)}}^r\cdots\int_{y_{N-m}^{(N-1)}}^{r} \partial_{x_i}^{N-i}e^{t\mathsf{L}^{(N)}}\left(x_i,y_{N-m+1}^{(N)}\right)dy_{N-m+1}^{(N)}\cdots y_2^{(m+1)},
\end{align*}
for some dummy variables $y_{2}^{(m+1)},\dots,y_{N-m+1}^{(N)}$ so that $y_1^{(m)}\le y_2^{(m+1)}\le \cdots\le y_{N-m}^{(N-1)} \le y_{N-m+1}^{(N)}$. As before, the claim can be seen by repeatedly using the identity (\ref{KeyIdentity2}) and Lemma \ref{LemmaObservations}:
\begin{align*}
e^{t\mathsf{L}^{(i)}}\left(x_i,y_m^{(m)}\right)=-e^{-\mathsf{c}^{(i)}t}\hat{\partial}_{y_m^{(m)}}^{-1}\partial_{x_i}e^{t\mathsf{L}^{(i+1)}}\left(x_i,y_m^{(m)}\right)=\cdots=(-1)^{N-i}e^{-t\sum_{k=i}^{N-1}\mathsf{c}^{(k)}}\hat{\partial}_{y_m^{(m)}}^{-(N-i)}\partial_{x_i}^{N-i}e^{t\mathsf{L}^{(N)}}\left(x_i,y_m^{(m)}\right).
\end{align*}
and then applying $\hat{\partial}_{y_m^{(m)}}^{m-i}$ to both sides. Thus, by multilinearity of the determinant, we can write display (\ref{DistributionIPS2t}) (suppressing $\prod_{n=1}^Ndy_n^{(n)}$) as follows:
\begin{align}
e^{-t\sum_{k=1}^{N-1}k\mathsf{c}^{(k)}}\int_{\overline{\mathcal{D}}_N    
}\det\left(\partial_{x_i}^{N-i}e^{t\mathsf{L}^{(N)}}\left(x_i,y_{N-j+1}^{(N)}\right)\right)_{i,j=1}^N \prod_{2\le i\le n \le N}dy_i^{(n)}\\
=\left(-1\right)^{\frac{N(N-1)}{2}}e^{-t\sum_{k=1}^{N-1}k\mathsf{c}^{(k)}}\int_{\overline{\mathcal{D}}_N    
}\det\left(\partial_{x_i}^{N-i}e^{t\mathsf{L}^{(N)}}\left(x_i,y_j^{(N)}\right)\right)_{i,j=1}^N \prod_{2\le i\le n \le N}dy_i^{(n)},
\end{align}
where $\overline{\mathcal{D}}_N=\overline{\mathcal{D}}_N\left(y_1^{(1)},\dots,y_1^{(N)}\right)$ (not to be confused with the closure of $\mathcal{D}_N$) is given by
\begin{equation*}
\overline{\mathcal{D}}_N=\left\{z_i^{(n)}\in (l,r): z_1^{(n)}=y_1^{(n)}; z_i^{(n)}\le z_{i+1}^{(n+1)}\right\}.
\end{equation*}
Finally, arguing as in the proof of Proposition \ref{DeterminantalProp1}, using Lemma 5.6 of \cite{ReflectedBrownianKPZ}, we obtain the conclusion.
\end{proof}

We will shortly need the following well-known lemma.

\begin{lem}\label{LemInterlacing}
We have the representation
\begin{equation*}
\mathbf{1}_{\left((y^{(1)},\dots,y^{(N)})\in\mathcal{D}_N\right)}=\prod_{n=1}^{N}\det\left(\partial^{-1}\left(y_i^{(n-1)},y_j^{(n)}\right)\right)_{i,j=1}^n,
\end{equation*}
with the convention that $y_n^{(n-1)}$ are some ``virtual" variables so that by definition $\partial^{-1}(y_n^{(n-1)},z)\equiv 1$.
\end{lem}

\begin{proof}
This follows from the well-known linear algebraic fact,
\begin{equation*}
\mathbf{1}_{(y_1<x_1\le y_2<x_2\le \cdots \le y_n<x_n\le y_{n+1})}=\det\left(\partial^{-1}\left(x_i,y_j\right)\right)_{i,j=1}^{n+1},
\end{equation*}
with the convention $\partial^{-1}(x_{n+1},z)\equiv1$, for any $z$.
\end{proof}

\begin{rmk}\label{InterlacingConstraintsRmk}
We note that we also have the following representation 
\begin{equation*}
\mathbf{1}_{\left((y^{(1)},\dots,y^{(N)})\in\mathcal{D}_N\right)}=\prod_{n=1}^{N}\det\left(\hat{\partial}^{-1}\left(y_i^{(n-1)},y_j^{(n)}\right)\right)_{i,j=1}^n,
\end{equation*}
where $\hat{\partial}^{-1}\left(x,y\right)=-\mathbf{1}_{x\le y}$ is the integral kernel of (\ref{IntegralOp}) and by definition $\hat{\partial}^{-1}\left(y_n^{(n-1)},z\right)\equiv 1$.
\end{rmk}

To proceed we require a special case of (one of the many variants of) the Eynard-Mehta theorem, see for example Lemma 3.5 in \cite{ReflectedBrownianKPZ}. For the convenience of the reader we give the statement from \cite{ReflectedBrownianKPZ} explicitly.

\begin{prop}\label{EynardMehtaEdge} Assume we have a signed measure (of total mass one) on $(l,r)\times (l,r)^2 \times \cdots \times (l,r)^N$ of the form
\begin{equation}\label{MeasureSigned}
    \frac{1}{Z} \prod_{n=1}^N\det\left(\phi\left(y_i^{(n-1)},y_j^{(n)}\right)\right)_{i,j=1}^n \det\left(\Psi_{N-i}^{(N)}\left(y_j^{(N)}\right)\right)_{i,j=1}^N,
\end{equation}
where the $y_{n}^{(n-1)}$ are some ``virtual" variables (in particular we assume $\phi\left(y_n^{(n-1)},z\right)$ has been defined and is independent of $n$) and $Z$ is a non-zero normalization constant. Define, for $n \in \mathbb{N}$:
\begin{equation*}
\phi^{(n)}(y_1,y_2)=\left(\phi*\cdots*\phi\right)(y_1,y_2),
\end{equation*}
where $\phi$ appears $n$ times in the convolutions $\left(\phi*\psi\right)(y_1,y_2)=\int_l^r\phi(y_1,z)\psi(z,y_2)dz$. Moreover, define for $n=1,\dots, N-1$, the functions
\begin{equation}\label{PsiRecursion}
    \Psi_{n-j}^{(n)}(z)=\left(\phi^{(N-n)}*\Psi_{N-j}^{(N)}\right)(z), \ \ j=1,\dots,N.
\end{equation}
For each $n=1,\dots,N$, the functions
\begin{equation}
    \left\{\phi^{(n)}\left(y_1^{(0)},z\right),\phi^{(n-1)}\left(y_2^{(1)},z\right),\dots,\phi\left(y_n^{(n-1)},z\right)\right\}
\end{equation}
are linearly independent and generate some $n$-dimensional space $V_n$. Moreover, define (uniquely) functions $\Phi_{n-j}^{(n)}(z)$, $j=1,\dots,n$ spanning $V_n$, with
\begin{align}\label{PrelimBiorthogonality}
\int_l^r\Phi_{n-i}^{(n)}(z)\Psi_{n-j}^{(n)}(z)dz=\mathbf{1}_{(i=j)},
\end{align}
for $1\le i,j\le n$. Finally, assume that $\phi\left(y_n^{(n-1)},z\right)=c_n\Phi_{0}^{(n)}(z)$, for some $c_n\neq 0$, for $n=1,\dots,N$. Then, the correlation functions of the ``signed random point process" induced by (\ref{MeasureSigned}) are determinantal and moreover the correlation kernel is given by
\begin{align}
K\left[\left(n_1,y_1\right);\left(n_2,y_2\right)\right]=-\phi_{n_2-n_1}\left(y_1,y_2\right)\mathbf{1}_{(n_2>n_1)}+\sum_{k=1}^{n_2}\Psi_{n_1-k}^{(n_1)}(y_1)\Phi_{n_2-k}^{(n_2)}(y_2).
\end{align}
\end{prop}

The following result proves that there is some determinantal structure associated to the interacting particle systems (\ref{IPS1}) and (\ref{IPS2}). However, the correlation kernel at this stage is still given implicitly in terms of the solution of a biorthogonalisation problem which will be solved explicitly in the next subsection.

\begin{thm}\label{ThmDetStructure}
Consider the interacting particle system $\left(\left(\mathsf{x}_1^\uparrow(t),\dots,\mathsf{x}_N^\uparrow(t)\right);t\ge 0\right)$ evolving according to the dynamics (\ref{IPS1}) with initial condition $x=(x_1,\dots,x_N)\in \mathbb{W}_N^\uparrow$. For any $t>0$, indices $1\le n_1<n_2<\cdots <n_M \le N$ and locations $z=(z_1,\dots,z_M)\in (l,r)^M$, we have
\begin{align}\label{Fredholm1}
 \mathbb{P}\left(\mathsf{x}^\uparrow_{n_j}(t)\le z_j, j=1,\dots,M\right) = \det\left(\mathbf{I}-\chi_z^+K\chi_z^+\right)_{L^2\left(\left\{n_1,\dots, n_M\right\}\times (l,r)\right)},
\end{align}
where $\det$ is the Fredholm determinant, with
\begin{align}\label{IntermediateKernel}
K\left[\left(n_1,y_1\right);\left(n_2,y_2\right)\right]&=-\partial^{-(n_2-n_1)}\left(y_1,y_2\right)\mathbf{1}_{(n_2>n_1)}+\sum_{k=1}^{n_2}\mathsf{\Psi}_{n_1-k}^{(n_1)}(y_1)\Phi_{n_2-k}^{(n_2)}(y_2),
\end{align}
where the functions $\mathsf{\Psi}$ are given by:
\begin{align}\label{PsiFunctionsDef}
\mathsf{\Psi}_{N-i}^{(N)}(w)=\partial_{x_i}^{N-i}e^{t\mathsf{L}}(x_i,w) \ \textnormal{and} \ \mathsf{\Psi}_{n-i}^{(n)}(w)=\left(\partial^{-(N-n)}\mathsf{\Psi}_{N-i}^{(N)}\right)(w), \ n=1,\dots, N-1, i=1,\dots, N,
\end{align}
and the functions $\Phi_i^{(n)}$, $i=0,\dots,n-1$, are uniquely determined by the following conditions:
\begin{enumerate}
    \item $\textnormal{span}\left\{\Phi_0^{(n)}(w),\dots,\Phi_{n-1}^{(n)}(w) \right\}=\mathfrak{P}_{\le n-1}$, the space of polynomials of degree at most $n-1$,
    \item The biorthogonality relations $\int_l^r\Phi_i^{(n)}(w)\mathsf{\Psi}_{j}^{(n)}(w)dw=\mathbf{1}_{i=j}$, $0\le i,j\le n-1$, hold. 
\end{enumerate}
Moreover, if we consider the interacting particle system $\left(\left(\mathsf{x}_1^\downarrow(t),\dots,\mathsf{x}_N^\downarrow(t)\right);t\ge 0\right)$ evolving according to the dynamics (\ref{IPS2}) with initial condition $x=(x_1,\dots,x_N)\in \mathbb{W}_N^\downarrow$, we have
\begin{align}\label{Fredholm2}
 \mathbb{P}\left(\mathsf{x}^\downarrow_{n_j}(t)\ge z_j, j=1,\dots,M\right) = \det\left(\mathbf{I}-\chi_z^-K\chi_z^-\right)_{L^2\left(\left\{n_1,\dots, n_M\right\}\times (l,r)\right)},
\end{align}
where $K$ is constructed as above, with $x\in \mathbb{W}_N^\downarrow$ instead.
\end{thm}

\begin{proof} The proof is standard given all the preliminary results. We apply Proposition \ref{EynardMehtaEdge}, noting Lemma \ref{LemInterlacing}, to the measure (\ref{SignedMeasure}) (which after symmetrization is exactly of the form (\ref{MeasureSigned})). This measure due to Propositions \ref{DeterminantalProp1} and \ref{DeterminantalProp2} has the distributions of the interacting particle systems (\ref{IPS1}) and (\ref{IPS2}), at fixed time $t>0$, as marginals, either as the right-most $(y_1^{(1)},y_2^{(2)},\dots,y_n^{(n)})$ or left-most $(y_1^{(1)},y_1^{(2)},\dots,y_1^{(n)})$ coordinates respectively. 

In applying Proposition \ref{EynardMehtaEdge} we first identify the $\phi(y_1,y_2)$ function with $\partial^{-1}(y_1,y_2)$ so that in particular 
$\phi^{(n)}(y_1,y_2)=\partial^{-n}(y_1,y_2)$, with the convention that $\phi(y_n^{(n-1)},y)\equiv 1$. Then, we can identify the $\Psi^{(n)}_{n-i}$ functions with the functions $\mathsf{\Psi}_{n-i}^{(n)}$ given in (\ref{PsiFunctionsDef}) by virtue of (\ref{PsiRecursion}). Moreover, as in \cite{ReflectedBrownianKPZ,FerrariSpohnWeiss1,FerrariSpohnWeiss2,NicaQuastelRemenik} (we have the same $\phi$ function) the space $V_n$ from Proposition \ref{EynardMehtaEdge} is given by
\begin{equation}
   V_n=\textnormal{span}\left\{y^{n-1},\dots,y,1\right\}=\mathfrak{P}_{\le n-1}.
\end{equation}
Now, we note that by virtue of display (\ref{PsiRep2}) in Proposition \ref{PropPsiRep} below and due to the fact that $e^{t\mathsf{L}^{(i)}}(x_i,\cdot)$ integrates to $1$ in $(l,r)$ we have (since we can take the $\partial^{k}_{x_i}$ derivatives outside the integral which is identically $1$):
\begin{equation*}
\int_{l}^r\mathsf{\Psi}_0^{(n)}(w)dw=(-1)^{N-n}e^{t\sum_{j=n}^{N-1}\mathsf{c}^{(j)}}, \ \ \int_l^r \mathsf{\Psi}_j^{(n)}(w)dw=0, \ \ \textnormal{ for } j=1,\dots,n-1,
\end{equation*}
which together with (\ref{PrelimBiorthogonality}) leads to $\Phi_0^{(n)}(y)=(-1)^{n-N}e^{-t\sum_{j=n}^{N-1}\mathsf{c}^{(j)}}$. Hence, the final assumption $\phi(y_n^{(n-1)},y)=c_n\Phi_{0}^{(n)}(y)$ of Proposition \ref{EynardMehtaEdge} is satisfied.

We can thus write down the correlation kernel associated to the signed measure (\ref{SignedMeasure}) as in (\ref{IntermediateKernel}), with the $\mathsf{\Psi}^{(n)}_{i}(z)$ functions given by (\ref{PsiFunctionsDef}) and the $\Phi_j^{(n)}(z)$ functions determined by the conditions in the statement of the proposition.
Finally, the Fredholm determinant identity (\ref{Fredholm1}) is a standard consequence of looking at an edge  marginal (either right-most or left-most particles) of a determinantal point process, see for example \cite{BorodinDeterminantal,JohanssonDeterminantal,ReflectedBrownianKPZ}. To conclude, the argument for the interacting particle system (\ref{IPS2}) and identity (\ref{Fredholm2}) is completely analogous by using Proposition \ref{DeterminantalProp2} now instead.
\end{proof}

\begin{rmk}
Observe that, we could have used in the proof above the representation for the $\phi$-function from Remark \ref{InterlacingConstraintsRmk} which would have given a slightly different correlation kernel (giving rise to the same ``signed" point process of course).
\end{rmk}

\subsection{Biorthogonalization and proof of Theorems \ref{ThmOneSidedUp} and \ref{ThmOneSidedDown}}\label{SectionBiorthogonalization}

We begin by giving two representations of $\mathsf{\Psi}_{n-i}^{(n)}$. One valid for any $i=1,\dots,N$ which we have already used in the proof of Theorem \ref{ThmDetStructure} above and will also use in computing the formula for the correlation kernel $\mathfrak{K}_t$ and one valid only for $i=1,\dots,n$ which is key in finding the biorthogonal polynomials below.

\begin{prop}\label{PropPsiRep}
Let $n=1,\dots,N$ be arbitrary. Then, we have the following representations for the $\mathsf{\Psi}$-functions from (\ref{PsiFunctionsDef}),
\begin{align}
 \mathsf{\Psi}_{n-i}^{(n)}(z)&=(-1)^{N-n}e^{t\sum_{j=n}^{N-1}\mathsf{c}^{(j)}}\partial_{x_i}^{n-i}e^{t\mathsf{L}^{(n)}}(x_i,z), \ \ \textnormal{ for } i=1,\dots,n,\label{PsiRep1}\\
 \mathsf{\Psi}_{n-i}^{(n)}(z)&=(-1)^{N-i}e^{t\sum_{j=i}^{N-1}\mathsf{c}^{(j)}}\partial_z^{n-i}e^{t\mathsf{L}^{(i)}}(x_i,z), \ \ \textnormal{ for } i=1,\dots,N.\label{PsiRep2}
\end{align}
\end{prop}

\begin{proof}
Recall that by definition $\mathsf{\Psi}_{n-i}^{(n)}(z)=\left(\partial^{-(N-n)}\mathsf{\Psi}_{N-i}^{(N)}\right)(z)$. Now, observe that for any choice of $i,n=1,\dots, N$, since $-(N-i)\le 0$, we have $\partial_{z}^{-(N-n)}=\partial_z^{n-i}\partial_z^{-(N-i)}$ from Lemma \ref{LemmaObservations}. Hence, by making repeated use of equation (\ref{KeyIdentity1}) and Lemma \ref{LemmaObservations}:
\begin{align*}
\mathsf{\Psi}_{n-i}^{(n)}(z)=\partial_z^{-(N-n)}\partial_{x_i}^{N-i}e^{t\mathsf{L}^{(N)}}(x_i,z)&=\partial_z^{n-i}\partial_z^{-(N-i)}\partial_{x_i}^{N-i}e^{t\mathsf{L}^{(N)}}(x_i,z)\\
&=(-1)^{N-i}e^{t\sum_{j=i}^{N-1}\mathsf{c}^{(j)}}\partial_{z}^{n-i}e^{t\mathsf{L}^{(i)}}(x_i,z),
\end{align*}
giving (\ref{PsiRep2}). To prove the other identity, suppose $i=1,\dots,n$, so that in particular $\partial_z^{N-i}=\partial_z^{n-i}\partial_z^{N-n}$ from Lemma \ref{LemmaObservations}. Then, again by repeated use of equation (\ref{KeyIdentity1}) and Lemma \ref{LemmaObservations} we can compute
\begin{align*}
 \mathsf{\Psi}_{n-i}^{(n)}(z)&=\partial_z^{-(N-n)}\partial_{x_i}^{N-i}e^{t\mathsf{L}^{(N)}}(x_i,z)=\partial_z^{-(N-n)}\partial_{x_i}^{n-i}\partial_{x_i}^{N-n}e^{t\mathsf{L}^{(N)}}(x_i,z)\\&=\partial_{x_i}^{n-i}\partial_{z}^{-(N-n)}\partial_{x_i}^{N-n}e^{t\mathsf{L}^{(N)}}(x_i,z)
 =(-1)^{N-n}e^{t\sum_{j=n}^{N-1}\mathsf{c}^{(j)}}\partial_{x_i}^{n-i}e^{t\mathsf{L}^{(n)}}(x_i,z),
\end{align*}
giving (\ref{PsiRep1}).
\end{proof}

We are now led to the definition of certain polynomials $\mathsf{\Phi}_i^{(n)}(z)$. The choice of notation is clearly suggestive. We will shortly identify $\mathsf{\Phi}_i^{(n)}$ with the polynomials $\Phi_i^{(n)}$ from Theorem \ref{ThmDetStructure}. The computation in Proposition \ref{BiorthogonalProp} below, albeit very simple, is the crux of the argument.

\begin{defn} Let $x=(x_1,\dots,x_N) \in (l,r)^N$. For $n=1,\dots,N$ and $i=0,1,\dots,n-1$ define the following polynomials (we drop dependence on $x$ in the notation):
\begin{equation}
\mathsf{\Phi}_i^{(n)}(z)=(-1)^{i+n-N}e^{-t\sum_{j=n}^{N-1}\mathsf{c}^{(j)}}e^{-t\mathsf{L}^{(n)}}\mathsf{q}_i^{(n)}(z;x),
\end{equation}
where recall that the polynomial $\mathsf{q}_i^{(n)}(z)=\mathsf{q}_i^{(n)}(z;x)$ was defined in Definition \ref{PolyDefinition}. Observe that, $\mathsf{\Phi}_0^{(n)}(z)\equiv (-1)^{n-N}e^{-t\sum_{j=n}^{N-1}\mathsf{c}^{(j)}}$.
\end{defn}

\begin{prop}\label{BiorthogonalProp}
Let $n=1,\dots,N$. Then, the functions $\mathsf{\Psi}_i^{(n)}$ and polynomials $\mathsf{\Phi}_j^{(n)}$ are biorthogonal with respect to Lebesgue measure in $(l,r)$:
\begin{equation}\label{BiorthogonalityRelation}
\int_{l}^r \mathsf{\Psi}_i^{(n)}(z)\mathsf{\Phi}_j^{(n)}(z)dz=\mathbf{1}_{(i=j)},  
\end{equation}
for $i,j=0,1,\dots,n-1$.
\end{prop}

\begin{proof}
Using the representation (\ref{PsiRep1}), we can compute by virtue of Lemma \ref{LemmaSemigroup}, after taking the derivative $\partial_{x_{n-i}}^i$ outside the integral,
\begin{align*}
  \int_{l}^r \mathsf{\Psi}_i^{(n)}(z)\mathsf{\Phi}_j^{(n)}(z)dz&=\int_l^r\partial_{x_{n-i}}^i e^{t\mathsf{L}^{(n)}}(x_{n-i},z)(-1)^ie^{-t\mathsf{L}^{(n)}}\mathsf{q}_j^{(n)}(z;x)dz \\
  &=(-1)^i\partial_z^{i}\mathsf{q}_j^{(n)}(z;x)\big|_{z=x_{n-i}}=\mathbf{1}_{(i=j)},
\end{align*}
for any $i,j=0,1,\dots,n-1$.
\end{proof}

We can finally prove our main results.

\begin{proof}[Proof of Theorem \ref{ThmOneSidedUp}]
We apply Theorem \ref{ThmDetStructure}. Observe that, since by Proposition \ref{WellDefinedProp}, $\mathsf{\Phi}_j^{(n)}$ is a polynomial of degree $j$, we have
\begin{equation*}
   \textnormal{span}\left\{\mathsf{\Phi}_0^{(n)}(y),\dots,\mathsf{\Phi}_{n-1}^{(n)}(y)\right\} =\mathfrak{P}_{\le n-1},
\end{equation*}
and moreover by virtue of the biorthogonality relation (\ref{BiorthogonalityRelation}) in Proposition \ref{BiorthogonalProp} we can then identify the $\Phi_j^{(n)}$ functions in Theorem \ref{ThmDetStructure} with $\mathsf{\Phi}_j^{(n)}$. Let us denote the correlation kernel $K$ from Theorem \ref{ThmDetStructure} by $\mathfrak{K}_t$. It remains to simplify the sum term in (\ref{IntermediateKernel}) as follows, recalling the representation (\ref{PsiRep2}) for $\mathsf{\Psi}_{n-i}^{(n)}$, by observing that
\begin{align*}
\sum_{k=1}^{n_2}\mathsf{\Psi}_{n_1-k}^{(n_1)}(y_1)\mathsf{\Phi}_{n_1-k}^{(n_2)}(y_2)=\partial_{y_1}^{n_1}\left[\sum_{k=1}^{n_2}e^{t\sum_{j=k}^{n_2-1}\mathsf{c}^{(j)}}\partial_{y_1}^{-k}e^{t\mathsf{L}^{(k)}}(x_k,y_1)\mathsf{q}_{n_2-k}^{(n_2)}(y_2;x)\right]e^{-t\mathsf{L}_{y_2}^{(n_2)}},
\end{align*}
which gives the form of the function $\mathfrak{G}_{n_2}(y_1,y_2)$ and concludes the proof.
\end{proof}

\begin{proof}[Proof of Theorem \ref{ThmOneSidedDown}]
The proof is word for word the same as the one of Theorem \ref{ThmOneSidedUp} above.
\end{proof}

\subsection{Connection to orthogonal polynomials}\label{RmkEquivalentRep} We now briefly explain how one can recover from $\mathfrak{K}_t$ the standard forms (directly related to random matrices as we now point out) of the correlation kernel for the most classical initial condition $x=(0,\dots,0)$ in the Brownian and squared Bessel cases. Using the results of Section 6 it can be shown that the correlation kernel $\mathfrak{K}_t$ equivalently governs the distribution, at a fixed time $t>0$, of some stochastic dynamics (\ref{DynamicsArray}) in an interlacing array (the measure (\ref{SignedMeasure}) coincides with the measure (\ref{EvolvedGibbs}) in Proposition \ref{MultilevelProp} as $x\to (0,\dots,0)$, see also Remark \ref{EntranceLaw}). This is known, see Remark \ref{RemarkRMTconnection}, to match the distribution of the eigenvalues of consecutive principal submatrices of the Gaussian and Laguerre unitary ensembles respectively, see \cite{Warren,InterlacingDiffusions}. The correlation kernels can be computed \cite{GUEminor,GUEminorErratum,DysonMinor,PartialConnection,forrester2011determinantal} in terms of Hermite and Laguerre polynomials respectively.

To obtain these formulae from $\mathfrak{K}_t$ in the squared Bessel/Laguerre case one needs the following ingredients. Recall that, as observed in (\ref{AllCoordEqPoly}), for $x=(0,\dots,0)$, $\mathsf{q}_k^{(n)}(z;x)=(-1)^k\frac{z^k}{k!}$. We then have the following identity (obtained by direct computation or see \cite{LaguerreIdentity}):
\begin{align}\label{LaguerreMonomial}
e^{-t\mathcal{B}^{(\theta)}}z^n=n!(-2t)^nL_n^{(\frac{\theta}{2}-1)}\left(\frac{z}{2t}\right),  
\end{align}
where $L_n^{(\alpha)}(z)=\sum_{k=0}^n(-1)^k \binom{n+\alpha}{n-k}\frac{z^k}{k!}$ is the Laguerre polynomial. Moreover, we have the identity:
\begin{equation}\label{LaguerreDerivativeIdentity}
\partial_z^ne^{t\mathcal{B}^{(\theta)}}(z,y)\big|_{z=0}=\frac{n!(-1)^n\Gamma\left(\frac{d}{2}\right)}{(2t)^n\Gamma\left(\frac{d}{2}+n\right)}L_n^{\left(\frac{\theta}{2}-1\right)}\left(\frac{y}{2t}\right)e^{t\mathcal{B}^{(\theta)}}(0,y).  
\end{equation}
This can be proven as follows. We make use of the formula
\begin{equation*}
 \partial_z^n e^{t\mathcal{B}^{(\theta)}}(z,y)=\frac{1}{(2t)^n}\sum_{k=0}^n e^{t\mathcal{B}^{(\theta+2k)}}(z,y)\binom{n}{k}(-1)^{n-k},
\end{equation*}
obtained by iterating the following relation (itself proven by direct computation using the explicit formula (\ref{BESQtransition}) and basic properties of derivatives of Bessel functions):
\begin{equation*}
\partial_z e^{t\mathcal{B}^{(\theta)}}(z,y)=\frac{1}{2t}\left[e^{t\mathcal{B}^{(\theta+2)}}(z,y)-e^{t\mathcal{B}^{(\theta)}}(z,y)\right]
\end{equation*}
and recalling that, see for example \cite{RevuzYor},
\begin{equation*}
e^{t\mathcal{B}^{(\theta)}}(0,y)=\frac{1}{(2t)^\frac{\theta}{2}\Gamma\left(\frac{\theta}{2}\right)} y^{\frac{\theta}{2}-1}e^{-\frac{y}{2}t}.
\end{equation*}
Putting everything together gives the explicit form of the kernel in terms of Laguerre polynomials. The situation is very similar in the Brownian case using classical identities, analogues of (\ref{LaguerreMonomial}), (\ref{LaguerreDerivativeIdentity}), that relate the heat kernel, $e^{t\frac{1}{2}\partial^2}(z,y)$, to the Hermite polynomials:
\begin{align*}
\partial_z^{n}e^{t\frac{1}{2}\partial^2}(z,y)=t^{-\frac{n}{2}}H_n\left(\frac{y-z}{\sqrt{t}}\right)e^{t\frac{1}{2}\partial^2}(z,y), \ \  e^{-t\frac{1}{2}\partial^2}z^n=t^nH_n\left(\frac{z}{t}\right),
\end{align*}
where $H_n(z)=n!\sum_{m=0}^{\lfloor \frac{n}{2} \rfloor}\frac{(-1)^m}{m!(n-2m)!}\frac{z^{n-2m}}{2^m}$ is the Hermite polynomial.

\subsection{Proof of Theorem \ref{ThmBessel}}\label{SectionRWKernel}

In this subsection we prove Theorem \ref{ThmBessel}. The lemma below is the key ingredient in the proof.

\begin{lem}\label{KeyBesselLemma}
Let $\theta \ge 2$. Then, for any $k \in \mathbb{N}$ we have, with $x,z\in (0,\infty)$, where $\delta_x$ is the delta function at $x$,
\begin{equation}
\partial_z^{-k}e^{t\mathcal{B}^{(\theta)}}(x,z)=\left[e^{t\mathcal{B}^{(4-\theta-2k)}}\partial^{-k}\delta_x\right](z). 
\end{equation}
\end{lem}

\begin{proof}
We begin with the following symmetry identity. For $\gamma \ge 2$ and $x,z\in (0,\infty)$ we have:
\begin{equation}\label{BesselSymmetry}
 e^{t\mathcal{B}^{(\gamma)}}(x,z)=e^{t\mathcal{B}^{(4-\gamma)}}(z,x),
\end{equation}
see Proposition 3 in \cite{BesselSurvey}. Moreover, we have the following identity, for $x,z\in (0,\infty)$,
\begin{equation}\label{BesselInductive}
\int_0^z e^{t\mathcal{B}^{(\gamma)}}(x,y)dy=\int_x^\infty e^{t\mathcal{B}^{(\gamma+2)}}(y,z)dy,
\end{equation}
which can be seen as follows. First note that, the identity (\ref{KeyIdentity1}) in the squared Bessel case can be written as $e^{t\mathcal{B}^{(\gamma+2)}}(y,z)=-\partial_z^{-1}\partial_y e^{t\mathcal{B}^{(\gamma)}}(y,z)$
and  then integrate in $y$ from $x$ to $\infty$ (note that the boundary term at $\infty$ on the right hand side vanishes). The $k=1$ case of the proposition then easily follows by using (\ref{BesselSymmetry}) and (\ref{BesselInductive}) in the third and second equalities below respectively:
\begin{align*}
 \partial_z^{-1}e^{t\mathcal{B}^{(\theta)}}(x,z)=\int_0^ze^{t\mathcal{B}^{(\theta)}}(x,y)dy=\int_{x}^\infty e^{t\mathcal{B}^{(\theta+2)}}(y,z)dy&=\int_{x}^\infty e^{t\mathcal{B}^{(2-\theta)}}(z,y)dy\\
 &=\left[e^{t\mathcal{B}^{(2-\theta)}}\partial^{-1}\delta_x\right](z).
\end{align*}
The general case is analogous, by making repeated use of (\ref{BesselInductive}) in the second equality below and then (\ref{BesselSymmetry}) we obtain:
\begin{align*}
  \partial_z^{-k}e^{t\mathcal{B}^{(\theta)}}(x,z)&=\int_0^z\int_0^{y_k}\cdots \int_0^{y_2}e^{t\mathcal{B}^{(\theta)}}(x,y_1)dy_1\cdots dy_k\\&=\int_x^\infty\int_{y_k}^{\infty}\cdots \int_{y_2}^{\infty}e^{t\mathcal{B}^{(\theta+2k)}}(y_1,z)dy_1\cdots dy_k
  \\
  &=\int_x^\infty\int_{y_k}^{\infty}\cdots \int_{y_2}^{\infty}e^{t\mathcal{B}^{(4-2k-\theta)}}(z,y_1)dy_1\cdots dy_k
  \\
  &=\int_x^{\infty}e^{t\mathcal{B}^{(4-\theta-2k)}}(z,y)\frac{(y-x)^{k-1}}{(k-1)!}dy= \left[e^{t\mathcal{B}^{(4-\theta-2k)}}\partial^{-k}\delta_x\right](z),
\end{align*}
since we note that for a non-negative function $f$
\begin{align*}
  \int_x^{\infty}\int_{y_k}^{\infty} \cdots \int_{y_2}^{\infty} f(y_1)dy_1\cdots dy_k= \int_x^{\infty} f(y)\frac{(y-x)^{k-1}}{(k-1)!}dy,
\end{align*}
and this completes the proof.
\end{proof}

We will also need the following important lemma from \cite{NicaQuastelRemenik}, see Proposition 5.7 therein which we have translated in our notation. Analogous results in the discrete setting have appeared in \cite{KPZfixedpoint,matetski2022tasep}.

\begin{lem}\label{RWrepLem} Let $x=(x_1,\dots,x_N)\in \mathbb{W}_N^{\downarrow}$ and $n=1,\dots,N$. Let $\left(\mathsf{R}_k;k\ge 0\right)$ be a random walk with exponential random variable with parameter $1$ steps to the left and denote by $\mathbf{E}$ expectation with respect to it. Let $\tau=\min\left\{k\ge 0:\mathsf{R}_k\ge x_{k+1}\right\}$. Then, for any $y_1,y_2$, we have
\begin{align}\label{RWrepDisplay}
  \sum_{k=0}^{n-1}\left[\partial^{-(n-k)}\delta_{x_{n-k}}\right](y_1)\mathsf{q}^{(n)}_{k}(y_2;x)=\mathbf{E}_{\mathsf{R}_0=y_1}\left[e^{y_1-\mathsf{R}_\tau}\frac{\left(\mathsf{R}_\tau-y_2\right)^{n-\tau-1}}{(n-\tau-1)!}\mathbf{1}_{(\tau<n)}\right].
\end{align}
\end{lem}

We can now prove Theorem \ref{ThmBessel}.

\begin{proof}[Proof of Theorem \ref{ThmBessel}]
We apply Theorem \ref{ThmOneSidedDown}. Observe that we have $\mathsf{L}^{(k)}=\mathcal{B}^{(\theta+2N-2k)}$ and $\mathsf{c}^{(k)}\equiv 0$. Moreover, since $\theta\ge 2$ the left boundary $0$ is an entrance boundary point while $\infty$ is always natural for all $k=1,\dots,N$, so that in particular our standing assumption is satisfied. We now rewrite the function $\mathfrak{G}_{n_2}(y_1,y_2)$ in the form that appears in the definition of the kernel $\mathfrak{B}_t^{(\theta)}$. We compute
\begin{align*}
\mathfrak{G}_{n_2}(y_1,y_2)&=\sum_{k=0}^{n_2-1}\partial_{y_1}^{-(n_2-k)}e^{t\mathsf{L}^{(n_2-k)}}(x_{n_2-k},y_1)\mathsf{q}_k^{(n_2)}(y_2;x)\\
&=\sum_{k=0}^{n_2-1}\partial_{y_1}^{-(n_2-k)}e^{t\mathcal{B}^{(\theta+2N-2n_2+2k)}}(x_{n_2-k},y_1)\mathsf{q}_k^{(n_2)}(y_2;x)\\
&=e^{t\mathcal{B}_{y_1}^{(4-\theta-2N)}}\sum_{k=0}^{n_2-1}\left[\partial^{-(n_2-k)}\delta_{x_{n_2-k}}\right](y_1)\mathsf{q}^{(n_2)}_{k}(y_2;x)\\
&=e^{t\mathcal{B}_{y_1}^{(4-\theta-2N)}}\mathbf{E}_{\mathsf{R}_0=y_1}\left[e^{y_1-\mathsf{R}_\tau}\frac{\left(\mathsf{R}_\tau-y_2\right)^{n-\tau-1}}{(n-\tau-1)!}\mathbf{1}_{(\tau<n)}\right],
\end{align*}
where we have used $\partial_z^{-(n_2-k)}e^{t\mathcal{B}^{(\theta+2N-2n_2+2k)}}(x_{n_2-k},y_1)=\left[e^{t\mathcal{B}^{(4-\theta-2N)}}\partial^{-(n_2-k)}\delta_{x_{n_2-k}}\right](y_1)$ by virtue of Lemma \ref{KeyBesselLemma} (note that the semigroup on the right hand side is independent of $k$ and thus can be taken out of the sum) and then Lemma \ref{RWrepLem}. This gives the first expression for $\mathfrak{B}_t^{(\theta)}$.

It remains to establish the equivalent form of $\mathfrak{B}_t^{(\theta)}$ given in (\ref{PathIntegralKernel}). Observe that for $n_1\ge n_2$ the equality is obvious since $\partial^{-(n_2-n_1)}\partial^{n_2}=\partial^{n_1}$ in this case. For $n_2>n_1$, we claim that for $n_2=0,1,\dots, N$, we have for any $z\in (l,r)$
\begin{equation}\label{DerivativeAtZero}
  \partial_{y_1}^{n_2}e^{t\mathcal{B}^{(4-\theta-2N)}}(y_1,z)\big|_{y_1=0}=0,
\end{equation}
from which the desired result follows by virtue of the fact that for $0\le k \le m$, and $f$ satisfying $\partial^if(l)=0$, for $i=0,\dots,m-1$, we have $\partial^{-k}\partial^{m} f=\partial^{m-k}f$. This fact can easily be proven by induction: $\partial^{-k}\partial^m f=\partial^{-(k-1)} \partial^{-1}\partial \partial^{m-1}f=\partial^{-(k-1)}\partial^{m-1}f$, since $\partial^{m-1}f(l)=0$, and so on. Now, equality (\ref{DerivativeAtZero}) can be seen as follows. From the symmetry property (\ref{BesselSymmetry}) we have:
\begin{equation*}
  \partial_{y_1}^{n_2}e^{t\mathcal{B}^{(4-\theta-2N)}}(y_1,z)= \partial_{y_1}^{n_2}e^{t\mathcal{B}^{(2N+\theta)}}(z,y_1).
\end{equation*}
Then, observe that by iterating the equality (proven by direct computation using properties of Bessel functions)
\begin{equation*}
  \partial_y e^{t\mathcal{B}^{(\theta)}}(z,y)=\frac{1}{2t}\left(-e^{t\mathcal{B}^{(\theta)}}(z,y)+e^{t\mathcal{B}^{(\theta-2)}}(z,y)\right)
\end{equation*}
we obtain the formula
\begin{equation*}
    \partial_{y_1}^{n_2}e^{t\mathcal{B}^{(2N+\theta)}}(z,y_1) = \frac{1}{(2t)^{n_2}}\sum_{j=0}^{n_2}(-1)^{n_2-j}\binom{n_2}{j}e^{t\mathcal{B}^{(2N+\theta-2j)}}(z,y_1).
\end{equation*}
The claim follows since $e^{t\mathcal{B}^{(2N+\theta-2j)}}(z,0)=0$, for $j=0,1,\dots,N$.
\end{proof}

\begin{rmk}\label{RemarkProbRep}
Some formal computations (one needs to be careful with boundary conditions) show that it is possible to find an analogue of Lemma \ref{KeyBesselLemma} for general $\mathsf{L}^{(k)}$-diffusions as in our standing assumption. Further computations then lead to an expression of the form 
\begin{equation}\label{NotUsefulDisplay}
 \sum_{k=0}^{n-1}g(k)\left[\partial^{-(n-k)}\delta_{x_{n-k}}\right](y_1)\mathsf{q}^{(n)}_{k}(y_2;x),\end{equation}
 appearing in the formula for the correlation kernel, for a certain function $g(k)$ depending on $\mathsf{L}^{(k)}$ (in the squared Bessel and Brownian cases $g(k)\equiv 1$). Following the working in \cite{NicaQuastelRemenik, KPZfixedpoint, matetski2022tasep} it is not hard to give an intermediate probabilistic representation for (\ref{NotUsefulDisplay}) in terms of an exponential random walk $\left(\mathsf{R}^*_k;k\ge 0\right)$ with steps only to the right; after conjugation $\left[\partial^{-(n-k)}\delta_{x_{n-k}}\right](y_1)$ can be written as a transition probability for $\mathsf{R}^*_k$ and $\mathsf{q}_k^{(n)}(y_2)$ in terms of a related stopping time. However, in order to obtain an expression as in Lemma \ref{RWrepLem}, one needs to reverse time (the direction of the walk) and with non-constant $g(k)$ this involves a last exit time for the walk $\left(\mathsf{R}_k;k\ge 0\right)$ which is a not a Markov time (we do not have the strong Markov property starting from a last exit time) and this is not conducive for the rest of the argument. It might still be possible to give a somewhat different (but still potentially useful for asymptotics) random walk representation to Lemma \ref{RWrepLem} but we do not pursue it further here.
\end{rmk}

\section{Proof of Theorem \ref{NoncollidingThm}}\label{SectionNonColliding}

We begin with a brief discussion on why the semigroup $\left(\mathsf{P}_t^{(N)};t\ge 0\right)$ from (\ref{NonCollidingSemigroup}) is well-defined and on well-posedness of the SDE (\ref{NonCollidingSDE}), even from points with multiple coinciding coordinates. The following lemma is an immediate consequence of the differentiation formulae for $\mathsf{\Delta}_N(x)$ in \cite{Doumerc}, by virtue of the polynomial form of $\mathsf{a}$ and $\mathsf{b}$.

\begin{lem}
Let $N\in \mathbb{N}$. Consider $\mathsf{L}$ with $\mathsf{a}$ and $\mathsf{b}$ as in (\ref{DiffusionDrift}). We have the following relation, with $x=(x_1,\dots,x_N)\in (l,r)^N$,
\begin{equation*}
 \left[\sum_{i=1}^N \mathsf{L}_{x_i}\right] \mathsf{\Delta}_N(x)=\lambda_N \mathsf{\Delta}_N(x), \ \ \textnormal{ with } \lambda_N=\frac{1}{6}N(N-1)(2a_2(N-2)+3b_1).
\end{equation*}
\end{lem}

Thus, $\mathsf{\Delta}_N(x)$ is a strictly positive eigenfunction in $\mathbb{W}_N^{\uparrow,\circ}$ of the generator $\sum_{i=1}^N \mathsf{L}_{x_i}$ of $N$ independent $\mathsf{L}$-diffusions in $\mathbb{W}_N^{\uparrow}$ killed when they intersect (namely with Dirichlet boundary conditions on the boundary of $\mathbb{W}_N^{\uparrow}$). We can then define the Doob $h$-transform \cite{Doob,RevuzYor,Pinsky} by $\mathsf{\Delta}_N$ of the corresponding sub-Markov semigroup, with transition density with respect to the Lebesgue measure given by the Karlin-McGregor formula \cite{KarlinMcGregor} as $\det\left(e^{t\mathsf{L}}(x_i,y_j)\right)_{i,j=1}^N$, which gives the definition of $\mathsf{P}_t^{(N)}$ in (\ref{NonCollidingSemigroup}).

We also have the following result on the SDEs. As mentioned, for initial conditions $\mathsf{z}(0)\in \mathbb{W}_N^{\uparrow,\circ}$ a direct generic argument is available that uses the Doob $h$-transform structure, see Appendix A in \cite{SingularValuesBMDrift}. Instead, for $\mathsf{z}(0)\in \mathbb{W}_N^{\uparrow}$ we apply the general results of \cite{GraczykMalecki} by checking the conditions therein. These conditions have already been checked in previous works \cite{GraczykMalecki,HuaPickrell} for some generators of the form $\mathsf{L}$ and the general case is analogous. 

\begin{prop}
Let $\mathsf{z}(0)\in \mathbb{W}_N^{\uparrow}$. Then, the system of SDEs (\ref{NonCollidingSDE}) has a unique strong solution with no collisions for positive time, namely almost surely, for all $t>0$, $\mathsf{z}(t)\in \mathbb{W}_N^{\uparrow,\circ}$.
\end{prop}

\begin{proof}
We apply Theorem 2.2 in \cite{GraczykMalecki}. The Yamada-Watanabe conditions (C1)-(C2) therein hold by our standing assumption (namely the form of the polynomial functions $\mathsf{a}$ and $\mathsf{b}$). Moreover, note that we can write the interacting drift term in the SDE (\ref{NonCollidingSDE}) as follows, for $i=1,\dots,N$:
\begin{equation*}
 2\mathsf{a}(x_i)\sum_{j\neq i}\frac{1}{x_i-x_j}=\sum_{j\neq i}\frac{2a_2x_ix_j+a_1(x_i+x_j)+2a_0}{x_i-x_j}+2(N-1)a_2x_i+(N-1)a_1.
\end{equation*}
We can thus take the function $H_{ij}(x,y)\equiv H(x,y)$ in \cite{GraczykMalecki} to be equal to
\begin{equation*}
H(x,y)=2a_2xy+a_1(x+y)+2a_0.
\end{equation*}
If we write
\begin{equation*}
H_1(x,y)=2a_2xy, \ H_2(x,y)=a_1(x+y), \ H_3(x,y)=2a_0,  
\end{equation*}
we note that conditions (A1), (A2), (A3) in \cite{GraczykMalecki} have already been checked for each of $H_1, H_2, H_3$ in \cite{GraczykMalecki, HuaPickrell} from which the corresponding condition for $H$ follows simply by adding up the individual inequalities in those conditions. Finally, the sets in condition (A4) in \cite{GraczykMalecki} are empty since, by our standing assumption, $\mathsf{a}(x)>0$ for all $x\in (l,r)$ and moreover condition (A5) is vacuous since the drift $\mathsf{b}$ does not depend depend on $i$. The conclusion of the proposition follows from the aforementioned theorem.
\end{proof}

Moving on towards proving Theorem \ref{NoncollidingThm} the following proposition plays a key role. In the Brownian and squared Bessel cases the proposition below specialises, in an equivalent form (as in these cases it is possible to write $e^{-t\mathsf{L}}p$, with $p$ a polynomial and $t>0$, as a certain integral expression, see \cite{KatoriTanemura,KatoriTanemuraBessel}), to some key results from \cite{KatoriTanemura,KatoriTanemuraBessel}. The main idea is to rewrite the measure from (\ref{NonCollidingDistribution}), (\ref{preliminarydensity}) below in an equivalent form so that the functions appearing in the first and last determinants in (\ref{DisplayNonColliding1}) are biorthogonal with respect to the transition density $e^{(t_M-t_1)\mathsf{L}}(y_1,y_2)$ as we show later on. This working is basically equivalent (but better adapted to the present setting) to finding biorthogonal functions as we did in the previous section and will allow for an application of yet another variant of the Eynard-Mehta theorem \cite{JohanssonDeterminantal,BorodinDeterminantal,BorodinRains}.

\begin{prop}\label{DistributionNonColliding}
Let $x=(x_1,\dots,x_N)\in \mathbb{W}_N^{\uparrow}$ and $M\in \mathbb{N}$ be arbitrary. The distribution of the dynamics (\ref{NonCollidingSDE}), starting from initial condition $x$, at times $0<t_1< t_2 < \cdots < t_M$,
can be written as follows, with respect to Lebesgue measure on the $M$-fold product $\mathbb{W}_N^\uparrow \times \cdots \times \mathbb{W}_N^\uparrow$:
\begin{align}
 \frac{1}{Z_{N,t_M}} \det\left(\mathcal{Q}_i^{(t_1)}\left(y^{(1)}_j;x\right)\right)_{i,j=1}^N \det\left(e^{(t_2-t_1)\mathsf{L}}\left(y^{(1)}_i,y^{(2)}_j\right)\right)_{i,j=1}^N\times \cdots \nonumber \\ \times\det\left(e^{(t_M-t_{M-1})\mathsf{L}}\left(y^{(M-1)}_i,y^{(M)}_j\right)\right)_{i,j=1}^N \left(\mathcal{P}_{i}^{(t_M)}\left(y_j^{(M)};x\right)\right)_{i,j=1}^N,\label{DisplayNonColliding1}
\end{align}
for some normalization constant $Z_{N,t_M}$ independent of the x variables and the functions $\mathcal{Q}_i^{(t)}\left(\cdot\right)=\mathcal{Q}_i^{(t)}\left(\cdot;x\right)$ and polynomials $\mathcal{P}_i^{(t)}\left(\cdot\right)=\mathcal{P}_i^{(t)}\left(\cdot;x\right)$, for $i=1,\dots,N$, are given by
\begin{align}
   \mathcal{Q}_i^{(t)}(y;x) &=\frac{1}{2\pi \textnormal{i}}\oint_{\mathsf{\Gamma}_i^{(N)}}e^{t\mathsf{L}}(z,y)\frac{1}{\prod_{k=1}^i(z-x_k)}dz,\label{QfnComplexVarRep}\\
 \mathcal{P}_{i}^{(t)}(y;x)&=e^{-t\mathsf{L}_y}\prod_{k=1}^{i-1}(y-x_k), \label{PfnRep}
\end{align}
with the convention $\mathcal{P}_1^{(t)}\equiv 1$ and where $\mathsf{\Gamma}_i^{(N)}$ is a positively oriented contour around $x_1,\dots,x_i$ in a complex neighbourhood of $(x_1,x_N)$. Alternatively, $\mathcal{Q}_i^{(t)}(y;x)$ can be written as 
\begin{equation}\label{DividedDiffRep}
    \mathcal{Q}_i^{(t)}(y;x)= e^{t\mathsf{L}}(\cdot,y)[x_1,\dots,x_i],
\end{equation}
where $f[x_1,\dots,x_i]$ denotes the divided difference \cite{DividedDifferences} of a function $f$ on $(l,r)$ evaluated at points $(x_1,\dots,x_i)$.
\end{prop} 

\begin{proof} Observe that, by the Markov property the distribution of the dynamics (\ref{NonCollidingSDE}), starting from initial condition $x=(x_1,\dots,x_N)\in \mathbb{W}_N^\uparrow$, at times $0<t_1 < t_2 < \cdots < t_M$ is given by:
\begin{equation}\label{NonCollidingDistribution}
\mathsf{P}_{t_1}^{(N)}\left(x,dy^{(1)}\right)\mathsf{P}_{t_2-t_1}^{(N)}\left(y^{(1)},dy^{(2)}\right)\cdots \mathsf{P}_{t_M-t_{M-1}}^{(N)}\left(y^{(M-1)},dy^{(M)}\right).
\end{equation}
Consider $x\in \mathbb{W}_N^{\uparrow,\circ}$ first. From the explicit formula in (\ref{NonCollidingSemigroup}) the above display (disregarding the $dy^{(i)}$ terms to ease notation), can be written as:
\begin{align}
\frac{1}{Z_{N,t_M}}\frac{1}{\mathsf{\Delta}_N(x)}\det\left(e^{t_1\mathsf{L}}\left(x_i,y^{(1)}_j\right)\right)_{i,j=1}^N\det\left(e^{(t_2-t_1)\mathsf{L}}\left(y^{(1)}_i,y^{(2)}_j\right)\right)_{i,j=1}^N \times \cdots\nonumber\\ 
\times \det\left(e^{(t_M-t_{M-1})\mathsf{L}}\left(y^{(M-1)}_i,y^{(M)}_j\right)\right)_{i,j=1}^N \mathsf{\Delta}_N\left(y^{(M)}\right).\label{preliminarydensity}
\end{align}
Here, $Z_{N,t_M}$ is a normalization constant which may change from line to line. We concentrate on the first two factors. Using the identity $\mathsf{\Delta}_N(x)=\prod_{i=2}^N \prod_{m=1}^{i-1}(x_i-x_m)$ we can write
\begin{equation*}
 \frac{1}{\mathsf{\Delta}_N(x)}\det\left(e^{t_1\mathsf{L}}\left(x_i,y^{(1)}_j\right)\right)_{i,j=1}^N=\det\left(e^{t_1\mathsf{L}}\left(x_i,y_j^{(1)}\right)\frac{1}{\prod_{m=1}^{i-1}(x_i-x_m)}\right)_{i,j=1}^N.
\end{equation*}
Then, by row operations we get that this is equal to:
\begin{equation*}
 \det\left(\sum_{m=1}^ie^{t_1\mathsf{L}}(x_m,y^{(1)}_j)\frac{1}{\prod_{\substack{k=1\\ k\neq m}}^i(x_m-x_k)}\right)_{i,j=1}^N.  
\end{equation*}
By the residue theorem, recall that $z\mapsto e^{t\mathsf{L}}(z,y)$ is analytic in a complex neighbourhood of $(x_1,x_N)$, we can write:
\begin{equation}\label{QSumRep}
 \sum_{m=1}^ie^{t_1\mathsf{L}}(x_m,y)\frac{1}{\prod_{\substack{k=1\\ k\neq m}}^i(x_m-x_k)}=\frac{1}{2\pi \textnormal{i}}\oint_{\mathsf{\Gamma}_i^{(N)}}e^{t\mathsf{L}}(z,y)\frac{1}{\prod_{k=1}^i(z-x_k)}dz.
\end{equation}
This gives the complex variables representation in (\ref{QfnComplexVarRep}) for the $\mathcal{Q}$-functions, while in (\ref{DividedDiffProof}) below we will see the divided differences representation (\ref{DividedDiffRep}). 

Now, we turn to the last factor $\mathsf{\Delta}_N\left(y^{(M)}\right)$. Recall that from Proposition \ref{WellDefinedProp}, we have that $e^{-t_M\mathsf{L}_y}\prod_{k=1}^{i-1}(y-x_k)$ is a polynomial of degree $i-1$ in $y$, with leading order coefficient $e^{t_M(i-1)(b_1+a_2(i-2))}$. Thus, by row operations $\mathsf{\Delta}_N\left(y^{(M)}\right)$ is equal, up to a multiplicative constant depending only on $t_M, b_1, a_2, N$, to
\begin{equation}\label{PpolynomialInterExpression}
 \det\left(e^{-t_M\mathsf{L}_{y_{j}^{(M)}}}\prod_{k=1}^{i-1}\left(y_j^{(M)}-x_k\right)\right)_{i,j=1}^N.
\end{equation}
This gives the required representation (\ref{PfnRep}) for the $\mathcal{P}$-polynomials.

We now extend the statement to general $x\in \mathbb{W}_N^{\uparrow}$ by continuity as follows. We show that the measure on $\mathbb{W}_N^\uparrow \times \cdots \times \mathbb{W}_N^\uparrow$ in (\ref{NonCollidingDistribution}), given by its density (\ref{preliminarydensity}) for $x\in \mathbb{W}_N^{\uparrow,\circ}$, which we have equivalently rewritten in (\ref{DisplayNonColliding1}), is weakly continuous in the initial condition $x$. We first show pointwise continuity in $x$ for the density as written in (\ref{DisplayNonColliding1}). Towards this end, observe that the expression for $\mathcal{P}_i^{(t)}(\cdot;x)$ in (\ref{PpolynomialInterExpression}) is easily seen to be continuous in $x=(x_1,\dots,x_N)\in \mathbb{W}_N^\uparrow$ since we can write by linearity
\begin{equation}\label{PpolyBackwardInTimeFormula}
   e^{-t_M\mathsf{L}_y}\prod_{k=1}^{i-1}(y-x_k) = \sum_{m=0}^{i-1} (-1)^{1-i-m}\mathfrak{e}_{i-1-m}(x_1,\dots,x_{i-1}) e^{-t_M\mathsf{L}_y}y^m,
\end{equation}
where  $\mathfrak{e}_k(x_1,\dots,x_{i-1})$ is the $k$-th elementary symmetric polynomial which is continuous in $x\in \mathbb{W}_N^{\uparrow}$. To see that the left hand side of (\ref{QSumRep}) is continuous in $x\in \mathbb{W}_N^\uparrow$ is slightly more involved (from the contour integral expression on the right hand side of (\ref{QSumRep}) this is straightforward but the real variable arguments below will also be used at the end of the proof) and we need some preliminaries. We define the divided difference $f[y_1,\dots,y_n]$ of a smooth function $f$ on $(l,r)$ at (distinct) points $(y_1,\dots,y_n)\in \mathbb{W}_n^{\uparrow,\circ}$ by the following procedure:
\begin{equation*}
f[y_1,y_2]=\frac{f(y_2)-f(y_1)}{y_2-y_1}, \ \ f[y_1,y_2,y_3]=\frac{f[y_2,y_3]-f[y_1,y_2]}{y_3-y_1},
\end{equation*}
and so on until
\begin{equation*}
f[y_1,\dots,y_n]=\frac{f[y_2,\dots,y_n]-f[y_1,\dots,y_{n-1}]}{y_n-y_1}.
\end{equation*}
The key observation, see for example \cite{DividedDifferences} for the general formula, is that for $x\in \mathbb{W}_N^{\uparrow,\circ}$, for $i=1,\dots, N$, we have (with the obvious convention $f[y_1]=f(y_1)$):
\begin{equation}\label{DividedDiffProof}
\sum_{m=1}^i e^{t_1\mathsf{L}}(x_m,y)\frac{1}{\prod_{\substack{k=1\\ k\neq m}}^i(x_m-x_k)}= e^{t_1\mathsf{L}}(\cdot,y)[x_1,\dots,x_i].
\end{equation}
Then, using the fact that divided differences are continuous for smooth functions, even at points with coinciding coordinates, see $\cite{DividedDifferences}$, we can extend the left hand side of (\ref{QSumRep}), namely $\mathcal{Q}_i^{(t)}\left(\cdot;x\right)$, for $i=1,\dots,N$, by continuity to general $x\in \mathbb{W}_N^\uparrow$. This gives pointwise continuity of the densities.

Now, consider points $x^{(n)}\in \mathbb{W}_N^{\uparrow,\circ}$ converging to $x\in \mathbb{W}_N^\uparrow$ as $n\to \infty$. To complete the proof we need to dominate the density given in (\ref{DisplayNonColliding1}), where we plug in $x=x^{(n)}$ in the formula, independently of $n$ by some integrable function on $\mathbb{W}_N^\uparrow \times \cdots \times \mathbb{W}_N^\uparrow$ (recall that for weak convergence we are testing against some continuous bounded function $\mathsf{F}$). By taking absolute values and expanding the determinans it suffices to dominate, uniformly in $n$, each of the following terms,
\begin{align*}
    \prod_{i=1}^N \left|\mathcal{Q}_i^{(t_1)}\left(y_{\sigma_1(i)}^{(1)};x^{(n)}\right)\right| \prod_{i=1}^N e^{(t_2-t_1)\mathsf{L}}\left(y_i^{(1)},y_{\sigma_2(i)}^{(2)}\right) &\times \cdots \\ \cdots \times 
\prod_{i=1}^N &e^{(t_M-t_{M-1})\mathsf{L}}\left(y_i^{(M-1)},y_{\sigma_M(i)}^{(M)}\right) \prod_{i=1}^N \left|\mathcal{P}_i^{(t_M)}\left(y_{\sigma_{M+1}(i)}^{(M)};x^{(n)}\right)\right|,
\end{align*}
by an integrable function on $\mathbb{W}_N^\uparrow \times \cdots \times \mathbb{W}_N^\uparrow$, where the $\sigma_1, \sigma_2, \dots,\sigma_{M+1}$ are arbitrary permutations of $\left\{1,2,..., N \right\}$. Moreover, it suffices to bound, uniformly in $n$, each factor of the form, for any $i,j=1,\dots,N$,
\begin{equation}
     \left|\mathcal{Q}_i^{(t_1)}\left(y_1;x^{(n)}\right)\right|  e^{(t_2-t_1)\mathsf{L}}\left(y_1,y_2\right) \cdots 
 e^{(t_M-t_{M-1})\mathsf{L}}\left(y_{M-1},y_{M}\right)  \left|\mathcal{P}_j^{(t_M)}\left(y_M;x^{(n)}\right)\right|
\end{equation}
by an integrable function of $y=(y_1,\dots,y_M) \in (l,r)^M$. Towards this end, note that $|e^{-t_M\mathsf{L}_{y_M}}y_M^m|$ is dominated by a non-negative polynomial $h_m$ of degree $2\left\lfloor \frac{m+1}{2} \right\rfloor$. Hence, from (\ref{PpolyBackwardInTimeFormula}) the sequence $\left\{\left|\mathcal{P}_j^{(t)}\left(y_M;x^{(n)}\right)\right|\right\}_{n=1}^\infty$ is dominated by the polynomial:
\begin{equation}\label{DominatingFunctionP}
\sum_{m=0}^{j-1} \sup_{n\ge 1}\left|\mathfrak{e}_{i-1-m}\left(x^{(n)}_1,\dots,x^{(n)}_{j-1}\right)\right| h_m(y_M). 
\end{equation}
Now, using the mean value theorem for divided differences, see \cite{DividedDifferences}, there exists some $\xi^{(n)}_i \in \left(x^{(n)}_1,x^{(n)}_i\right)$ such that:
\begin{equation}\label{MeanValueDividedDiff}
 \mathcal{Q}_i^{(t)}\left(y_1;x^{(n)}\right)=e^{t_1\mathsf{L}}(\cdot,y_1)\left[x^{(n)}_1,\dots,x^{(n)}_i\right]=\frac{1}{i!}\partial_{\xi^{(n)}_i}^{i-1}e^{t_1\mathsf{L}}\left(\xi^{(n)}_i,y_1\right).
\end{equation}
In particular, the sequence $\left\{\left|\mathcal{Q}_i^{(t)}\left(y_1;x^{(n)}\right)\right|\right\}_{n=1}^\infty$ is dominated by $\sup_{\xi \in \mathcal{U}_i} \frac{1}{i!}\left|\partial_{\xi}^{i-1}e^{t_1 \mathsf{L}}(\xi,y_1)\right|$, with $\mathcal{U}_i=\left(\inf_{n\ge 1}x_1^{(n)},\sup_{n\ge 1}x_i^{(n)}\right)$. Putting everything together it thus suffices to show that 
\begin{equation}\label{IntegrableDisplay}
 \sup_{\xi \in \mathcal{U}_i} \left|\partial_{\xi}^{i-1}e^{t_1 \mathsf{L}}(\xi,y_1)\right|   e^{(t_2-t_1)\mathsf{L}}\left(y_1,y_2\right) \cdots 
 e^{(t_M-t_{M-1})\mathsf{L}}\left(y_{M-1},y_{M}\right) h_j\left(y_M\right)
\end{equation}
is integrable on $(l,r)^M$. Then, using Tonelli's theorem (observe that all factors are non-negative), noting that,
\begin{equation}
\int_l^r \cdots \int_l^r e^{(t_2-t_1)\mathsf{L}}(y_1,y_2)\cdots e^{(t_M-t_{M-1})\mathsf{L}}(y_{M-1},y_M)h_j(y_M)dy_2\cdots dy_M=e^{(t_M-t_1)\mathsf{L}}h_j(y_1)
\end{equation}
is a polynomial and the fact that by standard estimates estimates \cite{StroockPDEbook}, see \cite{HuaPickrell} where this is worked out in detail in the Hua-Pickrell case, $\sup_{\xi \in \mathcal{U}_i} \left|\partial_{\xi}^{i-1}e^{t_1 \mathsf{L}}(\xi,\cdot)\right|$ integrates polynomials in $(l,r)$ we get that (\ref{IntegrableDisplay}) is integrable. The dominated convergence theorem gives the desired weak convergence (and we have already shown convergence of the densities to the desired formula), extending the result to general $x\in \mathbb{W}_N^\uparrow$ and this completes the proof.
\end{proof}

The following two lemmas, along with the version of the Eynard-Mehta theorem stated in Proposition \ref{EynardMehtaNonColliding}, essentially explain why we wrote the distribution of the dynamics at different times in the form given in Proposition \ref{DistributionNonColliding}.

\begin{lem}\label{LemOrhogonalityNonColliding}
Let $t>0$ and $x\in \mathbb{W}_N^{\uparrow}$. The functions $\mathcal{Q}_i^{(t)}(\cdot;x)$ and polynomials $\mathcal{P}_j^{(t)}(\cdot;x)$ are biorthogonal with respect to the Lebesgue measure on $(l,r)$:
\begin{equation*}
 \int_{l}^r \mathcal{Q}_i^{(t)}(y;x)\mathcal{P}_j^{(t)}(y;x)dy=\mathbf{1}_{(i=j)},
\end{equation*}
for $i,j=1,\dots,N$.
\end{lem}

\begin{proof}
We first assume $x \in \mathbb{W}_N^{\uparrow,\circ}$. We use the representation of $\mathcal{Q}$ as a sum on the left hand side of (\ref{QSumRep}). We thus have, using Lemmas \ref{LemCoincidence} and \ref{LemmaSemigroup},
\begin{align*}
\int_{l}^r \mathcal{Q}_i^{(t)}(y;x)\mathcal{P}_j^{(t)}(y;x)dy&=\int_l^r \sum_{m=1}^ie^{t\mathsf{L}}(x_m,y)\frac{1}{\prod_{\substack{k=1\\ k\neq m}}^i(x_m-x_k)}e^{-t\mathsf{L}_y}\prod_{k=1}^{j-1}(y-x_k)dy\\
 &= \sum_{m=1}^i\frac{1}{\prod_{\substack{k=1\\ k\neq m}}^i(x_m-x_k)}\int_{l}^r e^{t\mathsf{L}}(x_m,y)e^{-t\mathsf{L}_y}\prod_{k=1}^{j-1}(y-x_k)dy\\
 &=\sum_{m=1}^i\frac{1}{\prod_{\substack{k=1\\ k\neq m}}^i(x_m-x_k)}\prod_{k=1}^{j-1}(x_m-x_k)=\frac{1}{2\pi \textnormal{i}}\oint_{\mathsf{\Gamma}_i^{(N)}}\frac{\prod_{k=1}^{j-1}(z-x_k)}{\prod_{k=1}^{i}(z-x_k)}=\mathbf{1}_{(i=j)},
\end{align*}
the last equality being a simple fact from complex analysis. We finally extend the result to general $x\in \mathbb{W}_N^{\uparrow}$ by continuity, arguing using the dominated convergence theorem by virtue of (\ref{DominatingFunctionP}) and (\ref{MeanValueDividedDiff}), as in the proof of Proposition \ref{DistributionNonColliding}.
\end{proof}

\begin{lem}\label{PropagationLemNonColliding}
Let $x\in \mathbb{W}_N^\uparrow$. Let $0\le s\le t$ and $i=1,\dots,N$. Then, we have
\begin{align*}
\mathcal{Q}_i^{(s)}e^{(t-s)\mathsf{L}}(y;x)=\mathcal{Q}_i^{(t)}(y;x), \ \ e^{(t-s)\mathsf{L}}\mathcal{P}_i^{(t)}(y;x)=\mathcal{P}_i^{(s)}(y;x).
\end{align*}
\end{lem}

\begin{proof}
We first prove the identity for $\mathcal{Q}$ and assume $x\in \mathbb{W}_N^{\uparrow,\circ}$. We use the representation of $\mathcal{Q}$ as a sum on the left hand side of (\ref{QSumRep}) to compute, using the semigroup property,
\begin{align*}
  \int_l^r\mathcal{Q}_i^{(s)}(w)e^{(t-s)\mathsf{L}}(w,y)dw=\sum_{m=1}^i\frac{1}{\prod_{\substack{k=1\\ k\neq m}}^i(x_m-x_k)}\int_l^{r}e^{s\mathsf{L}}(x_m,w)e^{(t-s)\mathsf{L}}(w,y)dw =\mathcal{Q}_i^{(t)}(y).
\end{align*}
We then extend to general $x\in \mathbb{W}_N^\uparrow$ by continuity as in the proof of Proposition \ref{DistributionNonColliding}. Finally, the relation $e^{(t-s)\mathsf{L}}\mathcal{P}_i^{(t)}=\mathcal{P}_i^{(s)}$, for any $x\in \mathbb{W}_N^\uparrow$, is an immediate consequence of Lemma \ref{LemmaSemigroup}.
\end{proof}

We require the following special case (yet another variation) of the Eynard-Mehta theorem, see for example \cite{JohanssonDeterminantal,BorodinDeterminantal,BorodinRains}.

\begin{prop}\label{EynardMehtaNonColliding}
Consider a probability measure on the $M$-fold product $(l,r)^N\times \cdots \times (l,r)^N$ of the form 
\begin{align}\label{Measure}
\frac{1}{Z}\det\left(\mathcal{H}_i\left(y_j^{(1)}\right)\right)_{i,j=1}^N\prod_{r=1}^{M-1}\det\left(\phi_{r,r+1}\left(y_i^{(r)},y_j^{(r+1)}\right)\right)_{i,j=1}^N \det\left(\widetilde{\mathcal{H}}_i\left(y_j^{(M)}\right)\right)_{i,j=1}^N,
\end{align}
with $Z$ a non-zero normalization constant. For $n>m$, define
\begin{equation*}
\phi_{m,n}(y_1,y_2)=\left(\phi_{m,m+1}*\cdots *\phi_{n-1,n}\right)(y_1,y_2),
\end{equation*}
where $\left(\phi*\psi\right)(y_1,y_2)=\int_l^r\phi(y_1,z)\psi(z,y_2)dz$. Suppose we have, for $i,j=1,\dots,N$,
\begin{equation*}
  \int_l^r \int_l^r\mathcal{H}_i(y_1)\phi_{1,M}(y_1,y_2)\widetilde{\mathcal{H}}_j(y_2)dy_1dy_2=\mathbf{1}_{(i=j)}.
\end{equation*}
Finally, we define for 
\begin{align*}
 \mathcal{H}_i^{(r)}(y)=\left(\mathcal{H}_i*\phi_{1,r}\right)(y), \ \  \widetilde{\mathcal{H}}_i^{(r)}(y)=\left(\phi_{r,M}*\widetilde{\mathcal{H}}_i\right)(y).
\end{align*}
Then, the point process on $\{1,\dots,M\}\times (l,r)$ induced by (\ref{Measure}) is determinantal with correlation kernel:
\begin{equation}
K\left[(m,y_1);(n,y_2)\right]=\sum_{i=1}^N\mathcal{H}_i^{(m)}(y_1)\widetilde{\mathcal{H}}_i^{(n)}(y_2)-\phi_{n,m}(y_2,y_1)\mathbf{1}_{(n<m)}.
\end{equation}
\end{prop}

We are now in a position to prove Theorem \ref{NoncollidingThm}.

\begin{proof}[Proof of Theorem \ref{NoncollidingThm}]
We apply Proposition \ref{EynardMehtaNonColliding}, with the obvious identifications (observe that after symmetrization the probability measure (\ref{DisplayNonColliding1}) is exactly of the form (\ref{Measure})), since by virtue of Proposition \ref{DistributionNonColliding} and Lemmas \ref{LemOrhogonalityNonColliding} and \ref{PropagationLemNonColliding} we have:
\begin{align*}
 \int_{l}^r \int_l^r \mathcal{Q}_i^{(t_1)}(y_1;x)e^{(t_M-t_1)\mathsf{L}}(y_1,y_2)\mathcal{P}_j^{(t_M)}(y_2;x) dy_1 dy_2   =\mathbf{1}_{(i=j)}.
\end{align*}
To obtain the desired form of the kernel it only remains to simplify the sum as follows:
\begin{align*}
 \sum_{i=1}^N\mathcal{Q}_i^{(s)}\left(y_1;x\right)\mathcal{P}_i^{(t)}\left(y_2;x\right)&=\sum_{i=1}^{N}\frac{1}{2\pi \textnormal{i}}\oint_{\mathsf{\Gamma}_i^{(N)}}e^{s\mathsf{L}}(z,y_1)\frac{1}{\prod_{m=1}^{i}(z-x_m)}dz \times e^{-t\mathsf{L}_{y_2}}\prod_{m=1}^{i-1}(y_2-x_m)\\
 &=\frac{1}{2\pi \textnormal{i}}\oint_{\mathsf{\Gamma}^{(N)}}e^{s\mathsf{L}}(z,y_1)e^{-t\mathsf{L}_{y_2}}\sum_{i=1}^N\frac{\prod_{m=1}^{i-1}(y_2-x_m)}{\prod_{m=1}^{i}(z-x_m)}dz,
\end{align*}
since we can enlarge the contour from $\mathsf{\Gamma}_i^{(N)}$ to $\mathsf{\Gamma}^{(N)}=\mathsf{\Gamma}_N^{(N)}$ without encountering any poles and thus can bring the sum over $i$ inside the integral. Then, using the identity
\begin{align*}
\sum_{i=1}^{N}\frac{\prod_{m=1}^{i-1}(w-x_m)}{\prod_{m=1}^{i}(z-x_m)}=\left(\prod_{m=1}^N\frac{w-x_m}{z-x_m}-1\right)\frac{1}{w-z}=(w-z)^{N-1}\prod_{m=1}^N\frac{1}{z-x_m},
\end{align*}
we get the conclusion. Finally, we note that we could have used instead in the sum above the representation of $\mathcal{Q}_i^{(s)}\left(y;x\right)$ as a divided difference from (\ref{DividedDiffRep}) which avoids the use of complex variables but the expression is possibly less amenable to asymptotic analysis.
\end{proof}

\section{A partial connection between the interacting particle systems}\label{SectionConnection}

In this section we explain a rather partial, but still non-trivial connection, between the interacting particle systems with one-sided collisions (\ref{IPS1}), (\ref{IPS2}) and the non-colliding SDEs (\ref{NonCollidingSDE}). We need some notation and terminology. We say that $x \in \mathbb{W}_n^{\uparrow}$ and $y\in \mathbb{W}_{n+1}^{\uparrow}$ interlace and write $x\prec y$ if the following inequalities hold:
\begin{equation*}
y_1 \le x_1 \le y_2 \le \cdots \le y_n \le x_n \le y_{n+1}.
\end{equation*}
We then define interlacing arrays (of length $N$) by:
\begin{equation}\label{InterlacingArrayDefinition}
\mathbb{IA}_{N}=\left\{X^{(N)}=(x^{(1)},x^{(2)},\dots, x^{(N)})\in \mathbb{W}_1^{\uparrow}\times\mathbb{W}_2^{\uparrow}\times\cdots \times \mathbb{W}_N^{\uparrow}  :x^{(1)}\prec x^{(2)} \prec \cdots \prec x^{(N)}\right\}.
\end{equation}
We observe that a similar object to interlacing arrays, the set $\mathcal{D}_N$, has appeared before in (\ref{SetDN}), see Remark \ref{InterlacingComparison} for more details on this connection. We call the $x^{(i)}$'s the rows of the array.
We consider the following system of SDEs with reflection in $\mathbb{IA}_N$:
\begin{equation}\label{DynamicsArray}
d\mathsf{x}_i^{(k)}(t)=\sqrt{2\mathsf{a}\left(\mathsf{x}_i^{(k)}(t)\right)}d\mathsf{w}_i^{(k)}(t)+\mathsf{b}^{(k)}\left(\mathsf{x}_i^{(k)}(t)\right)dt+\frac{1}{2}d\mathfrak{l}_i^{\uparrow,(k)}(t)-\frac{1}{2}d\mathfrak{l}_i^{\downarrow,(k)}(t),
\end{equation}
with the $\mathsf{w}_i^{(k)}$ being independent standard Brownian motions and the finite variation terms $\mathfrak{l}_i^{\uparrow,(k)}, \mathfrak{l}_i^{\downarrow,(k)}$, which increase only when particles collide in order to preserve the interlacing can be identified with the semimartingale local times:
\begin{align*}
\mathfrak{l}_i^{\uparrow,(k)}&=\textnormal{sem. loc. time of } \mathsf{x}_i^{(k)}-\mathsf{x}_{i-1}^{(k-1)} \textnormal{ at } 0,\\
    \mathfrak{l}_i^{\downarrow,(k)}&=\textnormal{sem. loc. time of } \mathsf{x}_i^{(k)}-\mathsf{x}_{i}^{(k-1)} \textnormal{ at } 0.
\end{align*}
For indices which overflow or underflow, the corresponding terms are identically zero. 
 See Figure \ref{FigureInterlacing} for an illustration of the dynamics (\ref{DynamicsArray}). These SDEs have a unique strong solution (by virtue of the polynomial form of $\mathsf{a}$ and $\mathsf{b}$), see \cite{InterlacingDiffusions}, up until the stopping time
\begin{equation*}
\tau_{\textnormal{col}}=\inf\left\{t\ge 0: \exists \  (n,i,j), \ 2 \le n \le N-1, \ 1 \le i <j \le n,  \textnormal{ such that } \mathsf{x}_i^{(n)}(t)=\mathsf{x}_j^{(n)}(t)
\right\},
\end{equation*}
which corresponds to the problematic situation when a particle $\mathsf{x}_i^{(k)}$ gets trapped between $\mathsf{x}_{i-1}^{(k-1)}$ and $\mathsf{x}_i^{(k-1)}$ and gets pushed in opposing directions. Under the initial conditions we will consider, see equation (\ref{GibbsMeasure}) below, $\tau_{\textnormal{col}}=\infty$ almost surely and thus this situation does not arise.

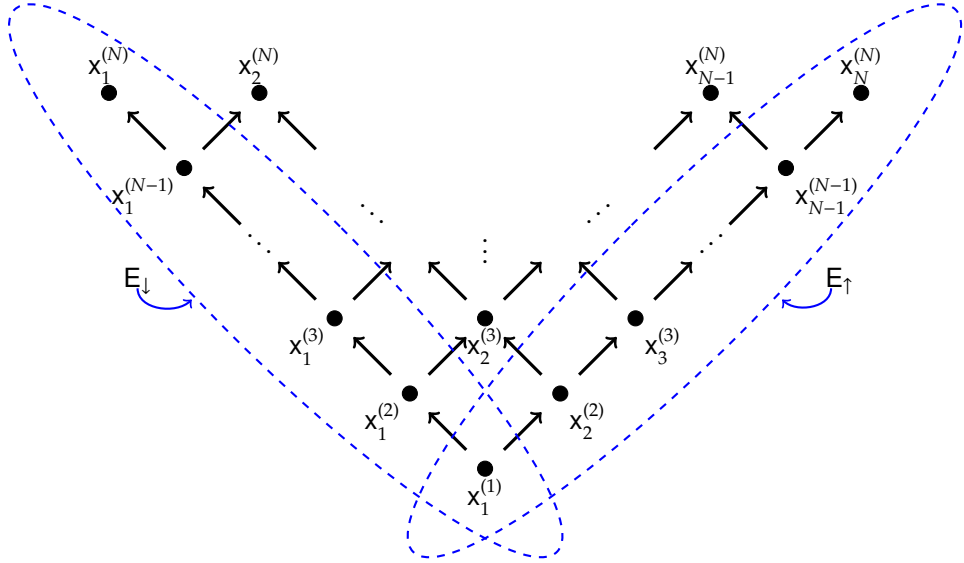
\begin{figure}
\centering
\captionsetup{singlelinecheck = false, justification=justified}
\begin{tikzpicture}

\draw[fill] (5,0) circle [radius=0.1];
\node[below ] at (5,0) {$\mathsf{x}_1^{(1)}$};

\draw[fill] (4,1) circle [radius=0.1];
\node[below left] at (4,1) {$\mathsf{x}_1^{(2)}$};

\draw[fill] (6,1) circle [radius=0.1];
\node[below right ] at (6,1) {$\mathsf{x}_2^{(2)}$};

\draw[fill] (3,2) circle [radius=0.1];
\node[below left ] at (3,2) {$\mathsf{x}_1^{(3)}$};

\draw[fill] (5,2) circle [radius=0.1];
\node[below ] at (5,2) {$\mathsf{x}_2^{(3)}$};

\draw[fill] (7,2) circle [radius=0.1];
\node[below right ] at (7,2) {$\mathsf{x}_3^{(3)}$};

\node[] at (5,3) {$\vdots$};

\node[] at (8,3.1) {$\iddots$};

\node[] at (2,3.1) {$\ddots$};

\draw[fill] (1,4) circle [radius=0.1];
\node[below left] at (1,4) {$\mathsf{x}_1^{(N-1)}$};

\draw[fill] (9,4) circle [radius=0.1];
\node[below right ] at (9,4) {$\mathsf{x}_{N-1}^{(N-1)}$};

\draw[fill] (0,5) circle [radius=0.1];
\node[above ] at (0,5) {$\mathsf{x}_1^{(N)}$};

\draw[fill] (2,5) circle [radius=0.1];
\node[above ] at (2,5) {$\mathsf{x}_2^{(N)}$};

\draw[fill] (8,5) circle [radius=0.1];
\node[above ] at (8,5) {$\mathsf{x}_{N-1}^{(N)}$};

\draw[fill] (10,5) circle [radius=0.1];
\node[above ] at (10,5) {$\mathsf{x}_{N}^{(N)}$};

\node[] at (6.5,3.5) {$\iddots$};

\node[] at (3.5,3.5) {$\ddots$};

\draw[->,very thick] (5.25,0.25) to (5.75,0.75);

\draw[->,very thick] (6.25,1.25) to (6.75,1.75);

\draw[->,very thick] (7.25,2.25) to (7.75,2.75);

\draw[->,very thick] (8.25,3.25) to (8.75,3.75);

\draw[->,very thick] (9.25,4.25) to (9.75,4.75);

\draw[->,very thick] (4.25,1.25) to (4.75,1.75);

\draw[->,very thick] (5.25,2.25) to (5.75,2.75);

\draw[->,very thick] (3.25,2.25) to (3.75,2.75);

\draw[->,very thick] (7.25,4.25) to (7.75,4.75);

\draw[->,very thick] (1.25,4.25) to (1.75,4.75);

\draw[->,very thick] (4.75,0.25) to (4.25,0.75);

\draw[->,very thick] (5.75,1.25) to (5.25,1.75);

\draw[->,very thick] (3.75,1.25) to (3.25,1.75);

\draw[->,very thick] (2.75,2.25) to (2.25,2.75);

\draw[->,very thick] (1.75,3.25) to (1.25,3.75);

\draw[->,very thick] (0.75,4.25) to (0.25,4.75);

\draw[->,very thick] (2.75,4.25) to (2.25,4.75);

\draw[->,very thick] (8.75,4.25) to (8.25,4.75);

\draw[->,very thick] (4.75,2.25) to (4.25,2.75);

\draw[->,very thick] (6.75,2.25) to (6.25,2.75);

\draw[blue, thick, dashed, rotate around={135:(7.65,2.5)}] (7.65,2.5) ellipse (1.02cm and 5.1cm);

\draw[blue, thick, dashed, rotate around={45:(2.35,2.5)}] (2.35,2.5) ellipse (1.02cm and 5.1cm);

\node[right] at (9.4,2.5) {$\mathsf{E}_{\uparrow}$};

\node[left] at (0.7,2.5) {$\mathsf{E}_{\downarrow}$};

\draw[blue, ->, thick]  (9.6,2.4) to[out=-90,in=-45] (8.95,2.25);

\draw[blue, ->, thick]  (0.38,2.4) to[out=-90,in=-135] (1.1,2.25);

\end{tikzpicture}\caption{The figure gives a cartoon description of the dynamics (\ref{DynamicsArray}) in $\mathbb{IA}_N$. The particles at row $k$ evolve as independent $\mathsf{L}^{(k)}$-diffusions modulo the following interactions. Particle $\mathsf{x}_i^{(k)}$ is autonomous except for its interaction with its two nearest neighbours in row $(k-1)$, namely $\mathsf{x}_{i-1}^{(k-1)}$ and $\mathsf{x}_{i}^{(k-1)}$: it receives an infinitesimal push, corresponding to the local times $\mathfrak{l}_i^{\uparrow,(k)}$ and $\mathfrak{l}_i^{\downarrow,(k)}$ respectively, only when it collides with them, denoted in the figure by the incoming arrows, in order for the interlacing to remain true. The projection maps $\mathsf{E}_{\uparrow}$ and $\mathsf{E}_{\downarrow}$ defined later on in (\ref{RightProjection}) and (\ref{LeftProjection}) simply return the coordinates at the right and left edges of the array respectively as encircled in the figure. It is clear that the corresponding particle systems $(\mathsf{x}_1^{(1)}(t),\dots,\mathsf{x}_N^{(N)}(t))$ and $(\mathsf{x}_1^{(1)}(t),\dots,\mathsf{x}_1^{(N)}(t))$ on the right and left edges of the array only interact among themselves (observe that there are no incoming arrows from other coordinates) and thus their evolution is Markovian and upon relabelling governed by the equations (\ref{IPS1}) and (\ref{IPS2}). }\label{FigureInterlacing}
\end{figure}

The connection of the dynamics (\ref{DynamicsArray}) to the interacting particle systems (\ref{IPS1}) and (\ref{IPS2}) is clear. They are simply the projections to the autonomous coordinates on the right and left edges of the array respectively, see Figure \ref{FigureInterlacing} for an illustration. A connection to the particle system (\ref{NonCollidingSDE}) on the other hand is not obvious at all. The next result explains it.

\begin{prop}\label{MultilevelProp} Let $\mu$ be a probability measure supported on $\mathbb{W}_N^{\uparrow,\circ}$. Under the standing assumption in Definition \ref{StandingAssumption}, suppose the SDEs (\ref{DynamicsArray}) in $\mathbb{IA}_N$ are initialized according to:
\begin{equation}\label{GibbsMeasure}
\mu\left(dx^{(N)}\right) \frac{\prod_{j=1}^{N-1}j!}{\mathsf{\Delta}_N\left(x^{(N)}\right)}\mathbf{1}_{\left(x^{(1)}\prec \cdots \prec x^{(N)}\right)} dx^{(1)}\cdots dx^{(N-1)}.
\end{equation}
Then, $\tau_{\textnormal{col}}=\infty$ almost surely, the projection on the top row $\left(\mathsf{x}^{(N)}(t);t\ge 0\right)$ is distributed as a Markov process with semigroup $\left(\mathsf{P}_t^{(N)};t\ge 0\right)$ from (\ref{NonCollidingSDE}) and for fixed time $T\ge 0$ the distribution of $\left(\mathsf{x}^{(1)}(T),\dots,\mathsf{x}^{(N)}(T)\right)$ in $\mathbb{IA}_N$ is given by:
\begin{equation}\label{EvolvedGibbs}
\left[\mu\mathsf{P}_T^{(N)}\right]\left(dx^{(N)}\right) \frac{\prod_{j=1}^{N-1}j!}{\mathsf{\Delta}_N\left(x^{(N)}\right)}\mathbf{1}_{\left(x^{(1)}\prec \cdots \prec x^{(N)}\right)} dx^{(1)}\cdots dx^{(N-1)}.
\end{equation}
\end{prop}

\begin{proof}
We apply Proposition 13.9 in \cite{InterlacingDiffusions}, see Section 13.3.1 there for the general setup. We take the $L_n$ generator in display (13.31) therein our $\mathsf{L}^{(n)}$-diffusion. Then, the so-called dual or conjugate diffusion $\widehat{\mathsf{L}^{(n)}}$ has generator:
\begin{equation*}
\widehat{\mathsf{L}^{(n)}}=\mathsf{a}(x)\frac{d^2}{dx^2}+\left[\mathsf{a}'(x)-\mathsf{b}^{(n)}(x)\right]\frac{d}{dx}=\mathsf{a}(x)\frac{d^2}{dx^2}-\left[\mathsf{b}(x)+(N-n-1)\mathsf{a}'(x)\right]\frac{d}{dx},
\end{equation*}
and note that if a boundary point is natural for $\mathsf{L}^{(n)}$ it remains natural for $\widehat{\mathsf{L}^{(n)}}$ and if it is entrance for $\mathsf{L}^{(n)}$ it is exit (see \cite{ItoMckean, HandbookBM, EthierKurtz,KarlinTaylor} for this terminology) for $\widehat{\mathsf{L}^{(n)}}$, see \cite{InterlacingDiffusions} for justifications. Associated to the diffusion $\widehat{\mathsf{L}^{(n)}}$ is a certain positive measure, called the speed measure (see also Proposition \ref{PropInvariantMeasure} and its proof where this is also discussed), having density with respect to Lebesgue measure in $(l,r)$ given by (denoted as $\widehat{m^n}$ in the notation of \cite{InterlacingDiffusions}), see \cite{InterlacingDiffusions}:
\begin{equation*}
\widehat{\mathsf{m}^{(n)}}(x)=\exp\left(-\int_{\zeta}^x\frac{\mathsf{b}^{(n)}(y)}{\mathsf{a}(y)}dy\right).
\end{equation*}
Here, $\zeta \in (l,r)$ is an arbitrary (fixed) point; changing $\zeta$ amounts to changing $\widehat{\mathsf{m}^{(n)}}$ by a multiplicative constant but this plays no role in what follows. We then pick the functions $g_n, G_n$ and constants $c_n$ in Proposition 13.9 (more precisely in Section 13.3.1 therein) of \cite{InterlacingDiffusions} as follows:
\begin{align*}
    g_n(x_1,\dots,x_n)&=\prod_{i=1}^n\left(\widehat{\mathsf{m}^{(n+1)}}(x_i)\right)^{-1}\mathsf{\Delta}_n(x),\\
    G_n(x_1,\dots,x_n)&=\frac{1}{(n-1)!}\mathsf{\Delta}_n(x),\\
    c_n&=a_2\frac{(n+1)n(n-1)}{3}+(b_1+2a_2(N-n-1))\frac{(n+1)n}{2}.
\end{align*}
Observe that, we have
\begin{equation*}
  \mathsf{\Delta}_n(x)=(n-1)!\int_{y\prec x} \prod_{i=1}^{n-1} \widehat{\mathsf{m}^{(n)}}(y_i) \prod_{i=1}^n \left(\widehat{\mathsf{m}^{(n)}}(y_i)\right)^{-1} \mathsf{\Delta}_{n-1}(y)dy_1\cdots dy_{n-1}
\end{equation*}
and hence the relation (13.33) in \cite{InterlacingDiffusions} holds. Moreover, a simple computation shows that $\left(\widehat{\mathsf{m}^{(n+1)}}\right)^{-1}$ is an eigenfunction of $\widehat{\mathsf{L}^{(n+1)}}$ with eigenvalue $\mathsf{c}^{(n)}=2(N-n-1)a_2+b_1$ and the corresponding Doob $h$-transform \cite{Doob,RevuzYor,Pinsky} by this function gives a $\mathsf{L}^{(n)}$-diffusion, $\widehat{\mathsf{m}^{(n+1)}}\circ \widehat{\mathsf{L}^{(n+1)}}\circ\left(\widehat{\mathsf{m}^{(n+1)}}\right)^{-1}-\mathsf{c}^{(n)}=\mathsf{L}^{(n)}$. At the level of transition densities in $(l,r)$ we then have:
\begin{align*}
e^{-\mathsf{c}^{(n)}t}\frac{\widehat{\mathsf{m}^{(n+1)}}(x)}{\widehat{\mathsf{m}^{(n+1)}}(y)}e^{t\widehat{\mathsf{L}^{(n+1)}}}(x,y)&=e^{t\mathsf{L}^{(n)}}(x,y).
\end{align*}
We note that if a boundary point is exit for $\widehat{\mathsf{L}^{(n+1)}}$ its transition kernel has an atom there. This is disregarded when we consider the transition density $e^{t\widehat{\mathsf{L}^{(n+1)}}}(x,y)$ in $(l,r)$, which equivalently can be thought of as the transition kernel of a $\widehat{\mathsf{L}^{(n+1)}}$-diffusion killed (instead of absorbed) when it reaches an exit boundary point. Thus, we readily check that 
\begin{align*}
 e^{c_{n-1}t}\frac{\mathsf{\Delta}_N(y)}{\mathsf{\Delta}_{N}(x)}\det\left(e^{t\mathsf{L}^{(n)}}(x_i,y_j)\right)_{i,j=1}^n=e^{-c_n t}\frac{\prod_{i=1}^n\widehat{\mathsf{m}^{(n+1)}}(x_i)\mathsf{\Delta}_N(y)}{\prod_{i=1}^n\widehat{\mathsf{m}^{(n+1)}}(y_i)\mathsf{\Delta}_{N}(x)}\det\left(e^{t\widehat{\mathsf{L}^{(n+1)}}}(x_i,y_j)\right)_{i,j=1}^n,
\end{align*}
and so display (13.34) in \cite{InterlacingDiffusions} holds. Finally, the assumptions (\textbf{R}), (\textbf{BC}+) and (\textbf{YW}) in Proposition 13.9 in \cite{InterlacingDiffusions} all hold by our standing assumption in Definition \ref{StandingAssumption}. We can thus apply Proposition 13.9 of \cite{InterlacingDiffusions}, from which the conclusion follows by noting that 
\begin{equation*}
 \prod_{n=1}^{N-1}\mathfrak{L}_n^{n+1}\left(x^{(n+1)},dx^{(n)}\right)= \frac{\prod_{j=1}^{N-1}j!}{\mathsf{\Delta}_{N}\left(x^{(N)}\right)}\mathbf{1}_{\left(x^{(1)}\prec \cdots \prec x^{(N)}\right)}dx^{(1)}\cdots dx^{(N-1)},
\end{equation*}
where the Markov kernel $\mathfrak{L}^{n}_{n-1}$ from $\mathbb{W}_n^{\uparrow}$ to $\mathbb{W}^{\uparrow}_{n-1}$ is given by the formula (for $x\in \mathbb{W}_N^{\uparrow,\circ}$ but can be extended to general $x\in \mathbb{W}_N^{\uparrow}$ by continuity, see for example \cite{HuaPickrell}):
\begin{equation*}
  \mathfrak{L}_{n-1}^n(x,dy)=\frac{\prod_{i=1}^{n-1}\widehat{\mathsf{m}^{(n)}}(y_i)g_{n-1}(y_1,\dots,y_{n-1})}{G_{n}(x_1,\dots,x_n)}\mathbf{1}_{(y\prec x)}dy= \frac{(n-1)!\mathsf{\Delta}_{n-1}(y)}{\mathsf{\Delta}_n(x)}\mathbf{1}_{(y\prec x)}dy.
\end{equation*}
\end{proof}

\begin{rmk}\label{HistoryRmk}
Proposition \ref{MultilevelProp} first appeared in the case of Brownian motion in the work of Warren \cite{Warren}. Further examples (the Ornstein-Uhlenbeck, squared Bessel, radial Ornstein-Uhlenbeck and Jacobi cases) were discussed in \cite{InterlacingDiffusions}. 
\end{rmk}

\begin{rmk}\label{EntranceLaw}
The dynamics (\ref{DynamicsArray}) can also be studied from singular initial conditions, when some of the coordinates of the top row coincide, using an entrance law, see \cite{RevuzYor} for the terminology, which is built from an entrance law for the dynamics of the top row $\left(\mathsf{P}_t^{(N)};t\ge 0\right)$, see \cite{InterlacingDiffusions,SingularValuesBMDrift} for more details. In the particular case of all the coordinates of the top row initially equal to some $x_*\in (l,r)$, the resulting measure in (\ref{EvolvedGibbs}) is seen to coincide (after a little computation), subject to identifying the set $\mathcal{D}_N$ from (\ref{SetDN}) with $\mathbb{IA}_N$ as discussed in Remark \ref{InterlacingComparison}, with the measure in (\ref{SignedMeasure}) with $x_i\equiv x_*$. For other initial conditions these two measures are different. Nevertheless, they are still in some sense related and we explain this in more detail around Proposition \ref{PropAlternativeRoute} below.
\end{rmk}

\begin{rmk}\label{RemarkRMTconnection} By virtue of Proposition \ref{MultilevelProp} the projection of the dynamics (\ref{DynamicsArray}) in $\mathbb{IA}_N$ on the top row match the evolution of eigenvalues of certain Hermitian matrix valued diffusions, or equivalently the dynamics (\ref{NonCollidingSDE}), see Section \ref{SectionExamples} for references (also we note that although the result is stated only for the top level analogous results hold for the $n$-th level by replacing $ \mathsf{L}$ by $\mathsf{L}^{(n)}$). Moreover, it can be shown (and has been shown in the cases mentioned in Remark \ref{HistoryRmk}) that the measure (\ref{EvolvedGibbs}) is also the distribution of the eigenvalues of consecutive sub-matrices of such a matrix valued diffusion at a fixed time $T\ge 0$. This is certainly true in the general setting of this paper (assuming one introduces the right Hermitian matrix valued diffusion corresponding to $\mathsf{L}$) but we will not pursue it further here. Finally, we note that the joint dynamics of the eigenvalues of consecutive sub-matrices are different from (\ref{DynamicsArray}), even in the case of Hermitian Brownian motion, see \cite{ConsecutiveMinorsDyson}.
\end{rmk}

We now give a more direct connection between the semigroups of the interacting particle systems in (\ref{IPS1}), (\ref{IPS2}) and (\ref{NonCollidingSDE}). Define the projections $\mathsf{E}_{\uparrow}:\mathbb{IA}_N \to \mathbb{W}_N^{\uparrow}$ and $\mathsf{E}_{\downarrow}:\mathbb{IA}_N \to \mathbb{W}_N^{\downarrow}$ as follows, see Figure \ref{FigureInterlacing} for an illustration,
\begin{align}
\mathsf{E}_{\uparrow}\left[(x^{(1)},x^{(2)},\dots,x^{(N)})\right]&=(x_1^{(1)},x_2^{(2)},\dots,x_N^{(N)}),\label{RightProjection} \\
\mathsf{E}_{\downarrow}\left[(x^{(1)},x^{(2)},\dots,x^{(N)})\right]&=(x_1^{(1)},x_1^{(2)},\dots,x_1^{(N)}).\label{LeftProjection}
\end{align}
Observe that, we can also think of $\mathsf{E}_{\uparrow}, \mathsf{E}_{\downarrow}$ as Markov kernels. Define the Markov kernel $\mathsf{\Lambda}$ from $\mathbb{W}_N^{\uparrow,\circ}$ to $\mathbb{IA}_N$ by, with $Y^{(N)}=(y^{(1)},\dots,y^{(N)})$:
\begin{equation}\label{ArrayMarkovKernel}
\mathsf{\Lambda}\left(x,dY^{(N)}\right)=\frac{\prod_{j=1}^{N-1}j!}{\mathsf{\Delta}_N(x)} \mathbf{1}_{\left(y^{(N)}=x,Y^{(N)}\in \mathbb{IA}_N\right)}dy^{(1)}\cdots dy^{(N-1)}dy^{(N)}.
\end{equation}
Note that, $\mathsf{\Lambda}\left(x,\cdot\right)$ is simply the distribution of a uniformly random interlacing array in $\mathbb{IA}_N$ with top arrow $x$. A key role is played by the kernels $\mathsf{\Lambda}\mathsf{E}_{\uparrow}$ and $\mathsf{\Lambda}\mathsf{E}_{\downarrow}$ from $\mathbb{W}_N^{\uparrow,\circ}$ to $\mathbb{W}_N^{\uparrow}$ and to $\mathbb{W}_N^{\downarrow}$ respectively: we simply pick a random array with fixed top row and then project to either edge. We have the following intertwining relations.

\begin{prop}\label{PropIntertwining} Let $t \ge 0$. Then, under the standing assumption in Definition \ref{StandingAssumption}, we have the intertwinings
\begin{align}
 \mathsf{P}_t^{(N)}\mathsf{\Lambda}\mathsf{E}_{\downarrow}=\mathsf{\Lambda}\mathsf{E}_{\downarrow} \mathsf{S}_t^{\downarrow,(N)},\\
 \mathsf{P}_t^{(N)}\mathsf{\Lambda}\mathsf{E}_{\uparrow}=\mathsf{\Lambda}\mathsf{E}_{\uparrow} \mathsf{S}_t^{\uparrow,(N)}.
\end{align}
\end{prop}

\begin{proof}
Let $\left(\mathsf{A}_t^{(N)};t\ge 0\right)$ be the semigroup\footnote{We note that there is no explicit expression for $\mathsf{A}_t^{(N)}$ but there is a formula for the transition kernel of two consecutive levels $((\mathsf{x}^{(k)}(t),\mathsf{x}^{(k+1)}(t));t\ge 0)$ in the dynamics (\ref{DynamicsArray}), see \cite{InterlacingDiffusions}.} of the process $(\mathsf{X}^{(N)}(t)=(\mathsf{x}^{(1)}(t),\dots,\mathsf{x}^{(N)}(t));t\ge 0)$ following the dynamics (\ref{DynamicsArray}) in the array $\mathbb{IA}_N$. Observe that, since the projections $\left(\mathsf{E}^{\uparrow}\left(\mathsf{X}^{(N)}(t)\right);t\ge 0\right)$ and $\left(\mathsf{E}^{\downarrow}\left(\mathsf{X}^{(N)}(t)\right);t\ge 0\right)$ are autonomous and given by the interacting particle systems (\ref{IPS1}) and (\ref{IPS2}) respectively, by Dynkin's criterion \cite{Dynkin} we have the intertwinings, for any $t\ge 0$: 
\begin{align}
 \mathsf{A}_t^{(N)}\mathsf{E}_{\downarrow}&=\mathsf{E}_{\downarrow}\mathsf{S}_t^{\downarrow,(N)},\label{Dynkin1}\\
  \mathsf{A}_t^{(N)}\mathsf{E}_{\uparrow}&=\mathsf{E}_{\uparrow}\mathsf{S}_t^{\uparrow,(N)}.\label{Dynkin2}
\end{align}

We now exhibit an intertwining between $\mathsf{A}_t^{(N)}$ and $\mathsf{P}_t^{(N)}$ which is less obvious. We follow the filtering framework for intertwinings presented in \cite{CarmonPetitYor}. Suppose we are in the setting of Proposition \ref{MultilevelProp}, with the dynamics (\ref{DynamicsArray}) initialized according to a probability measure of the form (\ref{GibbsMeasure}). Consider the natural filtrations $\left(\mathcal{F}_t\right)_{t\ge 0}$ and $\left(\mathcal{G}_t\right)_{t\ge 0}$ of $\mathsf{X}^{(N)}$ and $\mathsf{x}^{(N)}$:
\begin{align*}
\mathcal{F}_t=\sigma\left(\mathsf{X}^{(N)}(s)|s\le t\right), \ \ \mathcal{G}_t=\sigma\left(\mathsf{x}^{(N)}(s)|s\le t\right).
\end{align*}
By Proposition \ref{MultilevelProp}, $\left(\mathsf{x}^{(N)}(t);t\ge 0\right)$ is Markov with respect to its natural filtration $\left(\mathcal{G}_t\right)_{t\ge 0}$ with semigroup $\left(\mathsf{P}^{(N)}_t;t\ge 0\right)$. Moreover, by construction, see Corollary 13.1 and the argument for the proof of Proposition 13.9 in \cite{InterlacingDiffusions}, we have that the conditional distribution of $\mathsf{X}^{(N)}(t)$ given $\mathcal{G}_t$, namely the history of $(\mathsf{x}^{(N)}(s);s\le t)$, is uniform on arrays in $\mathbb{IA}_N$ with top row $\mathsf{x}^{(N)}(t)$ (the same holds if we only condition on $\widetilde{\mathcal{G}}_t=\sigma\left(\mathsf{x}^{(N)}(t)\right)$ and this would also give the result below). More formally, for any $t\ge 0$ and Borel function $f:\mathbb{IA}_N \to \mathbb{R}_+$ we have:
\begin{equation*}
\mathbb{E}\left[f\left(\mathsf{X}^{(N)}(t)\right)\big| \mathcal{G}_t\right]=\mathsf{\Lambda}f\left(\mathsf{x}^{(N)}(t)\right).
\end{equation*}
Let $\mathsf{x}^{(N)}(0)=x\in \mathbb{W}_N^{\uparrow,\circ}$ and a Borel function $f:\mathbb{IA}_N \to \mathbb{R}_+$ be arbitrary. Then, we can compute the following expectation in two ways:
\begin{align*}
\mathbb{E}\left[f\left(\mathsf{X}^{(N)}(t)\right)\right]&=\left[\mathsf{\Lambda}\mathsf{A}_t^{(N)}f\right] (x),\\
\mathbb{E}\left[f\left(\mathsf{X}^{(N)}(t)\right)\right]&=\mathbb{E}\left[\mathbb{E}\left[f\left(\mathsf{X}^{(N)}(t)\right)\big|\mathcal{G}_t\right]\right]=\mathbb{E}\left[\mathsf{\Lambda}f\left(\mathsf{x}^{(N)}(t)\right)\right]=\left[\mathsf{P}_t^{(N)}\mathsf{\Lambda}f\right](x).
\end{align*}
Hence, since $x\in \mathbb{W}_N^{\uparrow,\circ}$ and $f:\mathbb{IA}_N \to \mathbb{R}_+$ were arbitrary we have the intertwining, for $t\ge 0$:
\begin{align}\label{IntertwiningArrayTop}
\mathsf{P}_t^{(N)}\mathsf{\Lambda}=\mathsf{\Lambda}\mathsf{A}_t^{(N)}.
\end{align}
By putting (\ref{Dynkin1}) and (\ref{Dynkin2}) together with (\ref{IntertwiningArrayTop}) we get the statement of the proposition.
\end{proof}

We now discuss an alternative route to Proposition \ref{DeterminantalProp1}. For concreteness we focus on the right edge dynamics $\mathsf{S}_t^{\uparrow,(N)}$ but the situation for $\mathsf{S}_t^{\downarrow,(N)}$ is completely analogous. Define the operator $\left(\mathsf{\Lambda}\mathsf{E}_\uparrow\right)^{-1}$ acting on sufficiently smooth functions $g$ on $\mathbb{W}_N^{\uparrow,\circ}$ as follows:
\begin{equation}\label{InverseKernel}
\left(\mathsf{\Lambda}\mathsf{E}_\uparrow\right)^{-1}g(x)=\frac{1}{\prod_{j=1}^{N-1}j!}\prod_{i=1}^{N-1}\left(-\partial_{x_i}\right)^{N-i}\left[g(x)\mathsf{\Delta}_N(x)\right].
\end{equation}
Then, $\left(\mathsf{\Lambda}\mathsf{E}_\uparrow\right)^{-1}$ is a left inverse to the Markov kernel $\mathsf{\Lambda}\mathsf{E}_\uparrow$ as shown below.

\begin{lem}\label{InverseMarkovKernel}
Let $f$ be a continuous function on $\mathbb{W}_N^\uparrow$. Then, with $x\in \mathbb{W}_N^{\uparrow,\circ}$, we have
\begin{equation*}
\left[\left(\mathsf{\Lambda}\mathsf{E}_\uparrow\right)^{-1}\mathsf{\Lambda}\mathsf{E}_\uparrow f\right](x) =f(x).
\end{equation*}
\end{lem}

\begin{proof}
We prove this by induction. Observe that, we can write explicitly:
\begin{equation}\label{IntegralExpression}
\mathsf{\Lambda}\mathsf{E}_\uparrow f(x)=\frac{\prod_{j=1}^{N-1}j!}{\mathsf{\Delta}_N(x)}\int_{x^{(1)}\prec x^{(2)} \prec \cdots \prec x^{(N)}=x}f\left(x_1^{(1)},x_2^{(2)},\dots,x_N^{(N)}\right)dx^{(1)}dx^{(2)}\cdots dx^{(N-1)}.
\end{equation}
Note that the above integral is $2^{-1}N(N-1)$-dimensional. The base case of the induction is then equivalent to the identity
\begin{equation*}
-\partial_{x_1}\int_{x_1}^{x_2}f(y,x_2)dy=f(x_1,x_2).
\end{equation*}
For the inductive step we first apply $\left(-\partial_{x_1}\right)^{N-1}$ to the iterated integral in (\ref{IntegralExpression}) by sequentially applying $(N-1)$ times the operator $-\partial_{x_1}$ (also observe that the Vandermonde determinants cancel out). This annihilates the integrals in the variables $x_1^{(N-1)},x_1^{(N-2)},\dots,x_1^{(1)}$ and plugs in $x_1$ for $x_1^{(1)}$ in $f$. Then, after appropriate relabelling of the variables, we are left with a $2^{-1}(N-1)(N-2)$-dimensional integral which is exactly in the form required to apply the inductive hypothesis (we are basically looking at the interlacing array of length $N-1$ and top row $\big(x_2^{(N)},\dots,x_N^{(N)}\big)=(x_2,\dots,x_N)$ obtained by removing the left edge coordinates). This completes the proof.
\end{proof}

Hence, by virtue of Proposition \ref{PropIntertwining}, by applying $\left(\mathsf{\Lambda}\mathsf{E}_\uparrow\right)^{-1}$ on both sides of the relevant intertwining, we obtain 
\begin{equation}\label{AlternativeSemigroup}
\mathsf{S}_t^{\uparrow,(N)}=\left(\mathsf{\Lambda}\mathsf{E}_{\uparrow}\right)^{-1}\mathsf{P}_t^{(N)}\mathsf{\Lambda}\mathsf{E}_{\uparrow}.
\end{equation}
This gives an alternative expression to the one in (\ref{TransitionDensityIPS1Dis}) for $\mathsf{S}_t^{\uparrow,(N)}$ (by combining Propositions \ref{DeterminantalProp1} and \ref{PropAlternativeRoute} below it follows directly that the right hand side of (\ref{AlternativeSemigroup}) coincides with the expression in (\ref{TransitionDensityIPS1Dis})). We note that the argument just used to obtain (\ref{AlternativeSemigroup}) appears to have been first used in \cite{DiekerWarren} to study the transition kernels of some discrete interacting particle systems related to the RSK algorithm. Our interest in expression (\ref{AlternativeSemigroup}) is that it provides a rather clear path to Proposition \ref{DeterminantalProp1}. Namely, consider the following measure on $\mathbb{IA}_N$ given by, with $x\in \mathbb{W}_N^{\uparrow,\circ}$  (this extends to $x\in \mathbb{W}_N^\uparrow$ as the Vandermonde determinants in the formula cancel out),
\begin{equation}\label{AlternativeMeasure}
\left(\mathsf{\Lambda}\mathsf{E}_{\uparrow}\right)^{-1}\mathsf{P}_t^{(N)}\mathsf{\Lambda}\left(x,\cdot\right). 
\end{equation}
Now, by virtue of (\ref{AlternativeSemigroup}), $\mathsf{S}_t^{\uparrow,(N)}(x,\cdot)$ is seen to be the right edge marginal of (\ref{AlternativeMeasure}) in $\mathbb{IA}_N$. The following proposition then gives an alternative, rather roundabout (given all the preliminaries we need) way, of obtaining Proposition \ref{DeterminantalProp1} but which is maybe conceptually more satisfying.

\begin{prop}\label{PropAlternativeRoute}
 The measure (\ref{SignedMeasure}), subject to identifying $\mathcal{D}_N$ from (\ref{SetDN}) with $\mathbb{IA}_N$, coincides with the measure (\ref{AlternativeMeasure}).   
\end{prop}

\begin{proof}
This follows by simply writing out the explicit expressions for $\mathsf{P}_t^{(N)}$, $\mathsf{\Lambda}$ and $\left(\mathsf{\Lambda}\mathsf{E}_\uparrow\right)^{-1}$ given in (\ref{NonCollidingSemigroup}), (\ref{ArrayMarkovKernel}) and (\ref{InverseKernel}) respectively, noting that,
\begin{equation*}
 \prod_{i=1}^{N-1}\left(-\partial_{x_i}\right)^{N-i}=(-1)^{\frac{N(N-1)}{2}}\prod_{i=1}^{N-1}\partial_{x_i}^{N-i} \ \textnormal{ and } \ e^{-t\sum_{k=1}^{N-1}k\mathsf{c}^{(k)}}= e^{-t\lambda_N},
\end{equation*}
and comparing with the explicit expression in (\ref{SignedMeasure}).
\end{proof}

Finally, observe that for deterministic initial condition $x\in \mathbb{W}_N^{\uparrow,\circ}$ for the top row, the measure (\ref{GibbsMeasure}) on $\mathbb{IA}_N$ can be written as $\mathsf{\Lambda}(x,\cdot)$, and by Proposition \ref{MultilevelProp} its evolution under the dynamics (\ref{DynamicsArray}) after time $t$  is given by $\mathsf{P}_t^{(N)}\mathsf{\Lambda}(x,\cdot)$ which is related (by virtue of the above) to the signed measure (\ref{SignedMeasure}) by an application of $\left(\mathsf{\Lambda}\mathsf{E}_\uparrow\right)^{-1}$.

We conclude with the following immediate corollary of Proposition \ref{MultilevelProp}. Of course, this was known for the diffusions discussed in Remark \ref{HistoryRmk}. For the Brownian case this goes back even earlier by using certain path transformations related to the RSK algorithm, see \cite{BougerolJeulin,OConnellYorRepresentation}.

\begin{cor}\label{CorollaryDistribution}
Let $\mu$ be a probability measure supported on $\mathbb{W}_N^{\uparrow,\circ}$. Suppose the particle system $\left(\mathsf{z}(t);t\ge 0\right)$ following (\ref{NonCollidingSDE}) is initialized according to $\mu$ and the particle systems $\left(\mathsf{x}^\downarrow(t);t \ge 0\right)$ and $\left(\mathsf{x}^\uparrow(t);t \ge 0\right)$ following (\ref{IPS1}) and (\ref{IPS2}) respectively are initialized according to $\mu\mathsf{\Lambda}\mathsf{E}_{\uparrow}$ and $\mu\mathsf{\Lambda}\mathsf{E}_{\downarrow}$. Then, we have the following equality in distribution at the process level:
\begin{align*}
    \left(\mathsf{z}_1(t);t\ge 0\right)&\overset{\textnormal{d}}{=} \left(\mathsf{x}_N^\downarrow(t);t \ge 0\right),\\
     \left(\mathsf{z}_N(t);t\ge 0\right)&\overset{\textnormal{d}}{=} \left(\mathsf{x}_N^\uparrow(t);t \ge 0\right).
\end{align*}
\end{cor}

\begin{rmk}
For completeness we also record explicitly the probabilistic interpretation of the intertwinings (\ref{Dynkin1}) and (\ref{Dynkin2}). These relations imply that the projections under $\mathsf{E}_{\downarrow}$ and $\mathsf{E}_{\uparrow}$ of the full interlacing Markov process (\ref{DynamicsArray}) are Markovian with semigroups $\mathsf{S}_t^{\downarrow,(N)}$ and $\mathsf{S}_t^{\uparrow,(N)}$ respectively (this of course can also be  observed directly from the SDEs, as we have done already). This is an instance of Dynkin's criterion \cite{Dynkin} for when a process given as a function (in this case $\mathsf{E}_\downarrow$ or $\mathsf{E}_{\uparrow}$) of a Markov process is Markov itself.
\end{rmk}

\bibliographystyle{acm}
\bibliography{References}

\bigskip 

\noindent{\sc School of Mathematics, University of Edinburgh, James Clerk Maxwell Building, Peter Guthrie Tait Rd, Edinburgh EH9 3FD, U.K.}\newline
\href{mailto:theo.assiotis@ed.ac.uk}{\small theo.assiotis@ed.ac.uk}

\end{document}